\documentclass[a4paper,english,final]{amsart}

\usepackage{etex}
\usepackage[utf8]{inputenc}

\usepackage{multicol}

\usepackage{graphicx}
\usepackage{amsmath}
\usepackage{amssymb}
\usepackage{mathrsfs}
\usepackage{lmodern}
\usepackage{fontenc}[T1]
\usepackage{enumerate}
\usepackage{enumitem}
\usepackage[english]{babel}
\usepackage{ifthen}
\usepackage{oubraces}
\usepackage{nicefrac}
\usepackage{mathdots}
\usepackage{tcolorbox}

\usepackage{multirow}
\usepackage{xspace}
\usepackage{verbatim}

\newenvironment{me}{\noindent\color{red}}{}
\newenvironment{vd}{\noindent\color{blue}}{}

\makeatletter
\providecommand{\leadsfrom}{%
  \mathrel{\mathpalette\reflect@squig\relax}%
}
\newcommand{\reflect@squig}[2]{%
  \reflectbox{$\m@th#1\leadsto$}%
}
\makeatother

\usepackage{pdfsync}
\DeclareMathOperator{\Aut}{Aut}
\DeclareMathOperator{\End}{End}

\usepackage{tikz}
\usetikzlibrary{shapes}
\usetikzlibrary{arrows,decorations.markings}
\usepackage{pgf}
\usetikzlibrary{automata}
\usetikzlibrary{chains,decorations.pathmorphing,positioning,fit}
\usetikzlibrary{decorations.shapes,calc,backgrounds}
\usetikzlibrary{decorations.text,matrix}
\usetikzlibrary{arrows,shapes.geometric,shapes.symbols,scopes}
\usetikzlibrary{decorations.markings}

\colorlet{FillRed}{red!50}
\definecolor{blue}{rgb}{0.211,0.211,1.0} 
\colorlet{FillBrown}{gray!40}

\definecolor{darkgreen}{rgb}{0.0,0.7,0.0}

\newtheorem{theorem}{Theorem}[section]
\newtheorem{proposition}[theorem]{Proposition}
\newtheorem{lemma}[theorem]{Lemma}
\newtheorem{corollary}[theorem]{Corollary}
\newtheorem{example}[theorem]{Example}
\newtheorem{remark}[theorem]{Remark}
\newtheorem{definition}[theorem]{Definition}

\usepackage{prettyref}
\newcommand{\prref}[1]{\prettyref{#1}}
\newif{\ifshort}\shorttrue
\shorttrue
\ifshort
\newrefformat{alg}{Alg.~\ref{#1}}
 \newrefformat{thm}{Thm.~\ref{#1}}
 \newrefformat{lem}{Lem.~\ref{#1}}
 \newrefformat{def}{Def.~\ref{#1}}
 \newrefformat{cor}{Cor.~\ref{#1}} 
 \newrefformat{prop}{Prop.~\ref{#1}}
 \newrefformat{sec}{Sect.~\ref{#1}}
  \newrefformat{subsec}{Subsect.~\ref{#1}}
 \newrefformat{kap}{Chap.~\ref{#1}}
 \newrefformat{ex}{Ex.~\ref{#1}}
 \newrefformat{app}{App.~\ref{#1}}
 \newrefformat{alg}{Alg.~\ref{#1}}
 \newrefformat{eq}{(\ref{#1})}%
 \newrefformat{equiv}{(\ref{#1})}
 \newrefformat{tab}{Tab.~\ref{#1}}
 \newrefformat{fig}{Fig.~\ref{#1}}
 \newrefformat{rem}{Rem.~\ref{#1}}
 \newrefformat{ap}{{\sc Appendix}}

\else
\newrefformat{alg}{Algorithm~\ref{#1}}
\newrefformat{thm}{Theorem~\ref{#1}}
\newrefformat{lem}{Lemma~\ref{#1}}
\newrefformat{def}{Definition~\ref{#1}}
\newrefformat{cor}{Corollary~\ref{#1}}
\newrefformat{prop}{Proposition~\ref{#1}}
\newrefformat{sec}{Section~\ref{#1}}
\newrefformat{subsec}{Subsection~\ref{#1}}
\newrefformat{kap}{Chapter~\ref{#1}}
\newrefformat{ex}{Example~\ref{#1}}
\newrefformat{rem}{Remark~\ref{#1}}
\newrefformat{fig}{Figure~\ref{#1}}
\newrefformat{app}{Appendix~\ref{#1}}
\newrefformat{eq}{Equation~(\ref{#1})}
\newrefformat{equiv}{Equivalence~(\ref{#1})}
\newrefformat{tab}{Table~\ref{#1}}
\newrefformat{ap}{{\sc Appendix}}
\fi
\newcommand\EE{E}

\newcommand{\subauto}{subautomaton\xspace}
\newcommand{\subauta}{subautomata\xspace} 

\newcommand{\esolu}{entire solution\xspace}
\newcommand{\fopro}{forward property\xspace}

\newcommand{\FOTh}{\mathrm{FOTh}}

\newcommand{\Bnew}{\ensuremath{B_{\mathrm{new}}}}

\newcommand{\Bold}{\ensuremath{B_{\mathrm{old}}}}

\newcommand{\Xnew}{\ensuremath{\cX_{\mathrm{new}}}}

\newcommand{\wrt}{with respect to\xspace}
\newcommand{\ie}{i.e.\xspace}
\newcommand{\Ip}{In particular,\xspace}
\newcommand{\ip}{in particular,\xspace}
\newcommand{\comp}{compression\xspace}
\newcommand{\subst}{substitution\xspace}

\newcommand{\exe}{extended equation\xspace}
\newcommand{\exes}{extended equations\xspace}
\newcommand{\equ}{equation\xspace}
\newcommand{\equs}{equations\xspace}

\newcommand{\twequs}{twisted equations\xspace}

\newcommand{\solu}{solution\xspace}
\newcommand{\invol}{involution\xspace}
\newcommand{\fg}{f.g.~}

\newcommand{\delper}{$\del$-periodic\xspace}
\newcommand{\cgood}{good\xspace} 

\newcommand{\Psiben}{\Psi_{\text{Ben}}}

\newcommand{\Psimon}{\Psi_{\text{mon}}}
\newcommand{\Psimonnu}{\Psi_{\text{mon},\nu}}

\newcommand{\edtol}{EDT0L\xspace} 
\newcommand{\tra}{transition\xspace} 
\newcommand{\tras}{transitions\xspace}

\newcommand{\IFF}{if and only if\xspace}
\renewcommand{\hom}{homomorphism\xspace}
\newcommand{\homs}{homomorphisms\xspace}
\newcommand{\iso}{iso\-momor\-phism\xspace}

\newcommand{\Endo}{endomorphism\xspace}
\newcommand{\Endos}{endomorphisms\xspace}

\newcommand{\morph}{morphism\xspace}
\newcommand{\Hmorph}{compatible morphism\xspace}

\newcommand{\morphs}{morphisms\xspace}
\newcommand{\auto}{automorphism\xspace}
\newcommand{\autos}{automorphisms\xspace}

\newcommand{\velo}{very long\xspace}
\newcommand{\llong}{long\xspace}
\newcommand{\lds}{, \ldots ,}

\newcommand{\ra}{\longrightarrow}

\usepackage[framemethod=tikz]{mdframed} 

\newcommand{\slz}{\SL(2,\Z)}
\newcommand{\SLZ}{$\SL(2,\Z)$\xspace}

\newcommand{\SG}{\mathop{\text{SG}}}
\newcommand{\ST}{\mathop{\text{ST}}}

\newcommand{\HN}{$H$-$N$-monoid\xspace}

\newcommand{\HNs}{$H$-$N$-monoids\xspace}
\newcommand{\HNalp}{$H$-$N$-alphabet\xspace}
\newcommand{\HNalps}{$H$-$N$-alphabets\xspace}

\newcommand{\init}{\mathrm{init}}

\newcommand{\Winit}{W_{\mathrm{init}}}
\newcommand{\Einit}{E_{\mathrm{init}}}
\newcommand{\Efin}{E_{\mathrm{fin}}}
\newcommand{\alpfin}{\alp_{\mathrm{fin}}}
\newcommand{\Wfin}{W_{\mathrm{fin}}}

\newcommand{\arc}[1]{\overset{#1}\ra}
\newcommand{\darc}{\leftrightarrow}

\newcommand{\set}[2]{\left\{#1\mathrel{\left|\vphantom{#1}\vphantom{#2}\right.}#2\right\}}

\newcommand{\oneset}[1]{\left\{\mathinner{#1}\right\}}
\newcommand{\os}{\oneset}
\newcommand{\sm}{\setminus}
\newcommand{\es}{\emptyset}
\newcommand{\sse}{\subseteq}
\newcommand{\ssneq}{\varsubsetneq}

\newcommand{\smallset}[1]{\{\mathinner{#1}\}}

\newcommand{\vdmatrix}[4]{\left(\begin{smallmatrix}#1 & #2\\ #3 & #4\end{smallmatrix}\right)}

\newcommand{\abs}[1]{\left|\mathinner{#1}\right|}
\newcommand{\Abs}[1]{\left\Vert\mathinner{#1}\right\Vert}
\newcommand{\Abseq}[1]{\left\Vert\mathinner{#1}\right\Vert_{\text{eq}}}
\newcommand{\Absrat}[1]{\left\Vert\mathinner{#1}\right\Vert_{\text{rat}}}
\newcommand{\Absin}[1]{\left\Vert\mathinner{#1}\right\Vert_{\text{in}}}

\newcommand{\Absbin}[1]{\left\Vert\mathinner{#1}\right\Vert_{\text{bin}}} 
\newcommand{\Absone}[1]{\left\Vert\mathinner{#1}\right\Vert_{1}} 

\newcommand{\floor}[1]{\left\lfloor\mathinner{#1} \right\rfloor}

\newcommand{\gen}[1]{\left< \mathinner{#1} \right>}

\newcommand{\N}{\ensuremath{\mathbb{N}}}
\newcommand{\Z}{\ensuremath{\mathbb{Z}}}
\newcommand{\Q}{\ensuremath{\mathbb{Q}}}
\newcommand{\R}{\ensuremath{\mathbb{R}}}

\newcommand{\B}{\ensuremath{\mathbb{B}}}
\newcommand{\F}{\ensuremath{\mathbb{F}}}

\newcommand{\PSPACE}{\ensuremath{\mathsf{PSPACE}}}

\newcommand{\NP}{\ensuremath{\mathsf{NP}}}
\renewcommand{\P}{\ensuremath{\mathsf{P}}}
\newcommand{\NSPACE}{\ensuremath{\mathsf{NSPACE}}}

\newcommand{\Bn}{\B^{n\times n}}

\renewcommand{\phi}{\varphi}
\newcommand{\eps}{\varepsilon}

\newcommand{\alp}{\alpha}
\newcommand{\bet}{\beta}
\newcommand{\gam}{\gamma}
\newcommand{\del}{\delta}
\newcommand{\lam}{\lambda}
\newcommand{\sig}{\sigma}
\newcommand{\opphi}{\phi^{\text{op}}}

\newcommand{\Sig}{\Sigma}
\newcommand{\Gam}{\GG}
\newcommand{\Del}{\Delta}

\newcommand\GG{\Gamma}

\newcommand\Lam{\Lambda}
\newcommand\OO{\Omega}

\newcommand\stnf{\mathop\mathrm{snf}} 
\newcommand\rednf\stnf

\newcommand\SL{\mathop\mathrm{SL}} 
\newcommand\PSL{\mathop\mathrm{PSL}}

\newcommand{\Oh}{\mathcal{O}}

\newcommand{\wt}[1]{\widetilde{ #1 }}
\newcommand{\wh}[1]{\widehat{ #1 }}

\newcommand{\id}[1]{\mathrm{id}_{#1}}
\newcommand{\pr}[1]{\mathrm{pr}_{#1}}
\newcommand{\cA}{\mathcal{A}}
\newcommand{\cAcS}{\mathcal{A}_{\cS}}
\newcommand{\cB}{\mathcal{B}}

\newcommand{\cF}{\mathcal{F}}

\newcommand{\cG}{\mathcal{G}}

\newcommand{\cP}{\mathcal{P}}
\newcommand{\cS}{\mathcal{S}}
\newcommand{\cSPhi}{\cS_\Phi}
\newcommand{\cStri}{\cS_{\text{tri}}}
\newcommand{\cSfin}{\cS_{\text{fin}}}
\newcommand{\cSfinnu}{\cS_{\text{fin},\nu}}
\newcommand{\cT}{\mathcal{T}}
\newcommand{\cX}{\mathcal{X}}
\newcommand{\cY}{\mathcal{Y}}

\newcommand{\cR}{\mathcal{R}}
\newcommand{\cV}{\mathcal{V}}

\newcommand{\cSol}{\mathrm{Sol}}

\newcommand{\ov}[1]{\overline{#1}}

\newcommand{\oi}[1]{{#1}^{-1}}

\newcommand{\Rat}{\mathrm{Rat}}

\newcommand\RAS[2]{\overset{#1}{\underset{#2}{\Rightarrow}}}
\newcommand\LAS[2]{\overset{#1}{\underset{#2}{\Leftarrow}}}

\newcommand\finF{\mathscr F} 
\newcommand\inI{\mathscr I} 

\renewcommand{\leq}{\leqslant}
\renewcommand{\geq}{\geqslant}

\begin{document}
\iftrue
\title[Solutions to twisted word equations]{Solutions to twisted word equations and equations in virtually free groups}

\author{Volker Diekert}

\address{Institut f\"ur Formale Methoden der Informatik,  Universit\"at Stuttgart,
  Universit\"atsstr. 38, D-70569 Stuttgart, Germany}
\email{ diekert@fmi.uni-stuttgart.de} 

\author{Murray Elder}

\address{School of Mathematical and Physical Sciences, University of Technology Sydney,
Broadway NSW 2007, Australia}
\email{murray.elder@uts.edu.au}

\date{\today}
\thanks{Research supported by Australian Research Council (ARC) Project DP 160100486 and German Research Foundation (DFG) Project DI 435/7-1.}

\maketitle

\begin{abstract}
It is well known that the problem solving equations in virtually free groups can be reduced to the problem of solving twisted word equations with regular constraints over free monoids with involution. 
In this paper we prove that the set of all solutions  of  a twisted word equation is an 
 EDT0L language  whose specification can be computed in $\PSPACE$.
 Within the same complexity bound we can decide whether the solution set is empty, finite, or infinite.

In the second part of the paper we apply the results for twisted equations to obtain in $\PSPACE$ an EDT0L description of the solution set of equations with rational constraints for finitely generated virtually free groups  in standard normal forms with respect to a natural  set of generators.
 If the rational constraints are given by a \hom into a fixed (or ``small enough'') finite monoid, then our  
algorithms can be implemented in {$\NSPACE(n^2\log n)$, that is, in quasi-quadratic nondeterministic space.}

Our results 
generalize the work by  Lohrey and S\'enizergues (ICALP 2006) and Dahmani and Guirardel (J. of Topology 2010) with respect to both complexity and  expressive power. 
Neither paper gave any concrete complexity bound and the results in these papers are stated for subsets of solutions only, whereas our results concern all solutions.

	\bigskip
	


\end{abstract}

\section*{Introduction}\label{sec:intro}
For a given semigroup $S$ the decision problem \emph{\sc WordEquation} is the following:
 on input two words $U$ and $V$ in variables together with letters from a generating set $\Sig\sse S$, decide whether or not there exists a substitution $\sig$ of  variables by elements in $S$ which yields a true identity $\sig(U)= \sig(V)$ in $S$. Here, $\sig$ is extended by $\sig(s) = s$ for all $s\in \Sig$. 

In a seminal paper \cite{mak77} Makanin showed  that {\sc WordEquation} is decidable for  free semigroups.
The first complexity estimation of the problem was a tower of several exponential functions, but this dropped down to \PSPACE\ by Plandowski \cite{pla04jacm}
using compression. The insight that long solutions of word equations can be efficiently compressed is due to \cite{pr98icalp} which also led to the still standing
 conjecture that {\sc WordEquation} is $\NP$-complete for free semigroups (and free groups). Until 2013 the known decidability proofs for solving word equations were long and technical with an accompanied 
 reputation for being difficult. This 
 changed drastically when Je\.z applied his \emph{recompression} technique: he presented an 
$\NSPACE(n \log n)$ algorithm to solve word equations \cite{jez16jacm}\footnote{In \cite{jez17icalp} Je\.z improved the complexity to $\NSPACE(n)$.}. Actually his method achieves more: it describes all solutions, copes with rational constraints (which is essential in applications), it extends to free groups, and to free monoids with \invol \cite{DiekertJP16}. 
Refining Je\.z's method, Ciobanu and the present authors showed 
that the full solution set of a given word equation over a free monoid with \invol with rational constraints is  EDT0L  \cite{CiobanuDiekertElder2016ijac}. As a consequence, the same is true in free monoids without \invol and for free groups where the constraints are used to ensure that \solu{s} are given by reduced words.
Previously this was only known for quadratic word equations \cite{FerteMarinSenizerguesTocs14}.
EDT0L languages are defined by a certain type of Lindenmayer system. There is a vast literature on Lindenmayer systems, see \cite{RozS86},  but all we need here is that an EDT0L language is specified by a nondeterministic finite automaton accepting  \Endos over a free monoid and by some initial word. Applying the set of 
accepted  \Endos to the initial word yields the language. 

The original motivation for \cite{CiobanuDiekertElder2016ijac} was to prove that the
full solution set in reduced words of equations in free groups is an indexed language, a problem which was open at that time \cite{GilPC12,JainMS2012Lics}. However, the result in \cite{CiobanuDiekertElder2016ijac}  is stronger 
since  EDT0L forms a strict subclass of indexed languages \cite{EhrRoz77}.

Transfer results as in \cite{gut2000stoc,CiobanuDiekertElder2016ijac} from words to free groups have a long history. In the 1980s Makanin showed that the existential and positive theories of free groups  are decidable \cite{mak84}. In 1987 Razborov gave a description of all solutions for an equation in a free group via ``Makanin-Razborov'' diagrams \cite{raz87,raz93} which formed a cornerstone in the  work of 
Kharlampovich/Myasnikov \cite{KMIV06} and Sela \cite{sela13}
on the positive solution of Tarski's conjectures about the elementary theory in free groups.

The motivation for the present paper is along this line. We show that 
given a finitely generated virtually free group $G$ there is an $\NSPACE(n^2\log n)$ algorithm which produces for a given equation with (small) rational constraints an effective description of an \edtol language which describes the solution set in standard normal forms over a natural set of generators. 
Moreover, the same complexity is enough to decide whether the \solu set is empty, finite or infinite.
No $\PSPACE$ algorithm, in fact no concrete complexity bound was known for deciding emptiness before.

In this paper, we define an 
 $\NSPACE(s(n))$ algorithm 
\label{nspace-explanation}
{to be a partially defined {single-valued} function $f$ computed by a nondeterministic Turing machine consisting of three tapes: 
a one-way-read-only  input tape, a two-way-read-write work tape, and a one-way-write-only  output tape. If the length of the  word written on the input tape is $n$,  the work tape is restricted to having length $s(n)$.
If the machine halts on input $w$ at some point, then the contents $w'$ on the output tape satisfies $w'=f(w)$. In general such a device might specify a partially defined \emph{multi-valued} function, where several outputs are possible from the same input. However, in our case, we require that the output is unique. 
The domain of the partially defined function $f$ computed by the machine is the halting set of the machine, and for each $w$ in the domain there is a single output $f(w)$.  
This is the standard definition of a nondeterministic transducer which computes partially defined single-valued function. For nondeterministic 
polynomial time, formal definitions go back to \cite{BookLS84}; see also \cite{Selman94,Selman1996}.
It is clear that this formalism applies to other nondeterministic complexity classes as well. Every $\NSPACE(s(n))$ transducer} can be simulated by a deterministic transducer using at most working space $s(n)^2$ (Savitch's Theorem), and also by a deterministic Turing machine which uses a time bound in $2^{\Oh(s(n))}$, 
see \cite{pap94} for more details. Thus, every  $\PSPACE$ algorithm can be implemented such that it runs in deterministic singly exponential time  $2^{\text{poly}(n)}$.

Several remarks are in order here, which point to some additional difficulties in our framework.
First, in general virtually free groups have torsion, which is a serious obstacle to applying the known techniques. The reason to study virtually free groups is motivated by  the
ubiquitous presence of word hyperbolic groups \cite{gro87}. Solving equations in torsion-free hyperbolic groups reduces to solving equations in free groups \cite{rs95}, but solving equations in word hyperbolic groups with torsion reduces to solving 
equations in virtually free groups which in turn reduces to solving \emph{twisted} word equations with rational constraints \cite{DahmaniGui10}.
The question of whether solving twisted word equations is decidable was asked  by Makanin (\cite{Kourovka1986} Problem 10.26(b)). It was solved in  \cite{DahmaniGui10},  thereby showing that 
the set of solvable equations over a f.g.~virtually free group is decidable. This result  was also independently shown by Lohrey and 
S{\'e}nizergues \cite{LohSen06}. (Actually,  \cite{LohSen06} proves a more general transfer result.) What is in common: both papers are based on  \cite{dgh05IC}
and use explicitly (\cite{LohSen06} only implicitly due to \cite{dgh05IC}) the  so-called ``exponent of periodicity''. 
Because of this neither  paper describes all solutions, nor gives
 any concrete complexity bounds.

The result for virtually free groups is obtained 
by a 
reduction to the problem to describe the solution sets of twisted word equations with regular constraints, following standard techniques.
So, the main new contribution is our approach to solving twisted word equations. 
We follow the approach in  \cite{CiobanuDiekertElder2016ijac} to define a sound and complete algorithm to produce an  NFA $\cA$ describing all solutions, however  in the setting of twisted equations the technical details are quite far from previous methods.
For example, for readers familiar with previous methods, in twisted equations it does not make sense to ``uncross'' pairs $ab$ where $a$, $b$ are different letters because once all pairs $ab$ are uncrossed the twisting may produce new crossing pairs $ba$, uncrossing them leads to new crossing pairs  $ab$ etc. Thus, our underlying method is quite  different from the original recompression due to Je\.z.  

The class of  f.g.~virtually free groups
appears in many different ways. 
For example, a fundamental theorem of Muller and Schupp (relying originally on  \cite{Dunwoody85}) says that a f.g.~group is virtually free \IFF it is \emph{context-free} \cite{ms83}.
This means that, given any set of monoid generators $A$,   
the set of words $w\in A^*$ which represent $1\in V$ forms a context-free language.  
Other characterizations include:
(1) fundamental groups of finite graphs of finite groups \cite{Karrass73},  
(2) f.g.~groups having a Cayley graph with finite treewidth \cite{KuskeL05},
(3) universal groups of finite pregroups \cite{Rimlinger87a}, 
(4) groups having a finite presentation by some geodesic string rewriting system \cite{GilHHR07}, and
(5) f.g.~groups having a Cayley graph with decidable monadic second-order theory \cite{KuskeL05}.
Proofs for the most important equivalences are in \cite{DiekertW17crm}. The transformations are effective. For example, starting from a context-free grammar for the word problem, we can construct a representation as a fundamental group of finite graphs of finite groups. However, the finite graphs of finite groups can be much larger than the size of the context-free grammar: the result in \cite{sen96dimacs} showed a primitive recursive bound on the blow-up. It was only very recently that S\'enizergues and Wei\ss{} showed in \cite{SenizerguesW18} that the blow-up can be bounded by a doubly exponential function. 

What we use here is another characterization which is proved in \cite[Sec.~2.4.5]{DiekertW17crm}. It follows rather easily from Bass-Serre theory \cite{serre80} and the 
representation of a f.g.~virtually free group as a fundamental group of finite graphs of finite groups. The characterization says that a f.g.~group $G$ is virtually free \IFF it has an effective embedding 
into a semi-direct product of a free group $F$ with basis $E_+$ by a finite group $H$ which acts by permutations on the symmetric set $E= E_+\cup \oi E_+$. 
(The precise statement is in \prref{prop:DW17}.) Taking this characterization as a black box, no knowledge in Bass-Serre theory is required to understand our results. 
 
An extended abstract of a {preliminary version of} this paper was presented at the conference ICALP 2017, Warsaw (Poland), 10-14 July 2017 \cite{DiekertE17icalp}.  
Ciobanu and the second author have now extended the results of the present paper to show solutions to equations in any hyperbolic group are EDT0L with description in  
$\PSPACE$
\cite{CiobanuEicalp2019}.

\section{Organization of the paper}\label{sec:org}
\subsection{The overall structure}\label{sec:overallorg}

The paper has two main and separate parts. In a first part we deal the following algorithmic problem. The input is  
a system $\cS$ of twisted word equations with regular constraints over a free monoid (with \invol) $A^*$. The ``twist'' comes from a finite group $H$ acting on 
$A^*$. 
We present a $\PSPACE$ algorithm which constructs an NFA $\cAcS$ which gives a description of the set of all solutions as an \edtol language.
Structural properties of the NFA $\cAcS$ tell us whether the set of all solutions is empty, finite, or infinite. Precise complexity bounds are discussed in \prref{sec:mtw}, and under certain assumptions on the size of the regular constraints, we prove the entire algorithm can be done in  $\NSPACE(|H|n^2 \log |A|\log n)$  where $n$ denotes the input size of $\cS$, and when 
 $|A|$ and $|H|$ are constants, this becomes 
quasi-quadratic nondeterministic space $\NSPACE(n^2 \log n)$.

In a second part we apply our results on twisted word equations to the existential theory of equations with rational constraints over finitely generated virtually free groups. {}From the algorithmic viewpoint we deal with a non-uniform complexity where the virtually free group $G$ is not part of the input. (This allows us to assume that  $G$ is embedded into a 
semi-direct product of a free group $F$ by a finite group $H$, where 
the rank of $F$ and the size of $H$ are constants.) The input is a Boolean  formula $\Phi$ in free variables over equations and rational constraints for the \solu specified by NFAs which accepts subsets of $G$.  
The output is a specification of the set of all \solu{s} in standard normal  forms for $\Phi$ as an \edtol language. The proof is a reduction to 
the setting in the first part. The result is in the same overall complexity as in the first part when taking into account that $H$, $F$ and $A$ are not part of the input. 

In the final section 
 we perform this reduction explicitly for the special linear group $\slz$ (without relying on any knowledge of Bass-Serre theory) starting with the well-known classical fact that
 $\slz$ can be embedded in semi-direct product of a free group of rank $2$ (its commutator subgroup) by the finite cyclic group $\Z/12\Z$. 
{\it A priori}, there could be an exponential blow-up in the complexity 
due to fact that we use matrices and not a word representation when describing equations over $\slz$. However, there is no such blow-up thanks to work of Gurevich and Schupp \cite{GurevichS07}.

\subsection{Technical details}\label{sec:tecdet}
We assume that the reader is familiar with some basic facts in combinatorics on words, formal languages and finite automata, and complexity theory. Apart from that (and the promise that a finitely generated (\fg for short) 
virtually free group admits an embedding into a certain 
semi-direct product of a free group $F$ by a finite group $H$) the paper is self-contained.  In principle, it is not necessary that the reader has ever heard of Makanin,  word equations, or any method to solve them before.  
The paper uses various technical tools where the authors would have preferred to give references in the literature rather than lengthy and 
somewhat pedestrian 
constructions, but failed to find the appropriate references. 

The heart of the paper is Je\.z's compression method in the framework of twisted equations: \prref{sec:compr} and  \prref{sec:paircomp}.
The adaption to the twisted setting is far from  trivial and quite different from the original method in \cite{jez16jacm} or its extension to free groups as in \cite{DiekertJP16} or \cite{CiobanuDiekertElder2016ijac}. Therefore to understand Sections~\ref{sec:compr} and~\ref{sec:paircomp} is the most demanding part when reading the paper. 
 
 Many of the technicalities surrounding $\NSPACE$ complexity  can be overlooked if the reader is 
 happy enough
  to replace the explicit 
 complexity bounds  by $\PSPACE$. 
 

\section{Preliminaries}\label{sec:prel}
We use standard notation. 
If $A$ and $B$ are sets, then $A\sse B$ means set inclusion, while
$A\ssneq B$ means $A\sse B\wedge A\neq B$. 
By 
$A\sm B$ we denote the set of $a\in A$ which are not in $B$. 
 By $B^A$ we mean the set of mappings from $A$ to $B$, and  $2^A$ denotes the power set of $A$, that is,
$2^A=\set{B}{B\sse A}$. We also view $2^A$ as a commutative and idempotent monoid. 

$\B=(\{0,1\},\max,{\cdot},0,1),$ denotes the Boolean semi-ring, $\N$ (resp.~$\Z$) denotes the semi-ring of natural numbers, (resp.~the ring of integers). $\N$ is also the free monoid and $\Z$ is the free group in one generator. 

Monoids (resp.~groups) will  typically be denoted by $M$  and $N$ (resp.~by $G$ and $H$). 
If the focus is on finite monoids (resp.~finite groups), then we use the notation $N$ (resp.~$H$). With a few exceptions (like $\N$ or $\Z$) we denote the identity element in monoids by $1$. A \emph{zero} in a monoid $M$ is an element $0\in M$ such that
$0x=x0=0$ for all $x\in M$. If a zero exists, it is unique. Nontrivial groups cannot have a zero.

Let $M$ be a monoid and $u,v\in M$. We say that $u$ is a \emph{factor} of $v$ if we can write
$v= xuy$ for some $x,y \in M$. 
If we can write $v= uy$ (resp.~$v= xu$), then we say that $u$ is a \emph{prefix} (resp.~\emph{suffix}) of $v$. If $u$ is a prefix of $v$, then we also write $u\leq v$. 

\subsection{Complexity}\label{sec:compl}
The $\Oh$-notation and complexity classes like $\P$, $\NP$,  $\PSPACE$, $\NSPACE(s(n))$ are defined  as in standard textbooks (\cite{HU,pap94}), see also page~\pageref{nspace-explanation}. 
We  use a convention that   $\log(m) 
=\max\{1,\log_2(m)\}$.
Throughout  we use the well-known fact due to  Immerman-Szelepcs{\'e}nyi that $\NSPACE(s(n))$ is closed under complementation.
Note that any statement about complexity depends on how the input is given. A statement like ``factorization is not known to be in~$\P$'' makes sense only if 
the encoding of the problem and a notion of \emph{input size} has been defined. 
If integers are encoded in unary, then, trivially,  ``factorization is in~$\P$'' is true. 
Typically, inputs have various parameters. If certain parameters of the input are fixed, then  \wrt the input size these parameters behave as  constants. Still, for many problems $\cP$ it is more accurate to use parametrized inputs, where the input size is tuple of non-negative numbers:  $\Abs{\cP} = (p_1\lds p_k)$ 
with $k\geq 1$. If $\cP$ is such a problem, then \wrt  polynomial resource bounds we view $\cP$ as a one-parameter 
problem of input size $n=1+p_1+\cdots +p_k$. The notation is robust: every polynomial in $(1+p_1)^\ell\cdots (1+p_k)^\ell$ is also a polynomial in $n$
for all $\ell \geq 1$. 
Throughout, we take care to define our input sizes (of systems of equations, or Boolean formulae, with regular constraints) in a natural way.

\subsection{Sets and monoids with \invol}\label{sec:mi}
An \emph{involution} of a set is a bijection $x \mapsto \ov x$ such that 
$\overline{\overline{x}} = x$ for all $x$ in the set. The identity map
 is an \invol. 
A \emph{monoid with \invol} additionally has to satisfy $\overline{xy}=\overline{y}\,\overline{x}$. This implies $\ov 1= 1$ and $\ov 0= 0$ (in case there is a zero).
If $G$ is a group, then it is a monoid with \invol by taking $\ov g = \oi g$ for all $g \in G$. By default, we choose $\ov g$ to be $\oi g$ in groups. 

A \emph{\morph} between sets with involution is a mapping  respecting the involution. 
A \emph{\morph} between monoids with \invol is a \hom $\phi\colon M \to M'$ such that 
$\phi(\ov x) = \ov{\phi(x)}$. Note that every group \hom is a \morph  of monoids with involution. 
The set of \autos 
on a set (or monoid) $M$ forms the group $\Aut(M)$. 
For $\Del \sse M \cap M'$ we say that $\phi\colon M\to M'$ is 
a \emph{$\Del$-\morph} if $\phi(x) = x$ for all $x \in \Del$.

\subsection{Group actions and $H$-monoids}\label{sec:ga}
Recall that a group $H$ acts on a set $\Sig$ (with \invol)  
via a \hom $\psi\colon  H \to \Aut(\Sig)$. That is, $\psi$ defines a permutation
$x \mapsto g \cdot x$ with $1 \cdot x= x$, $f\cdot(g \cdot x) = (fg) \cdot x$ 
(and $f\cdot \ov x = \ov{f \cdot x}$) for all $f,g \in {H}$ and $x \in \Sig$. 
Thus, every $g\in {H}$ defines a permutation of 
$\Sig$ (which respects the \invol).  The \emph{stabilizer} of $x\in \Sig$ is the 
subgroup $H_x=\set{g\in H}{g\cdot x = x}$. 
Frequently, we also write $g(x)$ as a synonym for $g\cdot x$.  
If $H$ acts on a monoid  $M$, then we additionally demand that 
every element of $H$ acts as an \auto: $g(xy) =g(x)g(y)$. If $M$ is equipped with an \invol, then we have $g(\ov{xy})=  g(\ov{y}){g(\ov x)}= \ov{g(y)}\;\ov{g(x)}$. In the following we  say that $M$ is an \emph{$H$-monoid}
if it is a monoid with \invol on which $H$ acts.
A \morph between $H$-monoids $M$ and $M'$ is given by an \emph{$H$-\Hmorph}
which is  a \hom $\phi\colon M\to M'$ respecting for all $g\in H$ and $x\in M$ the 
action $g\cdot \phi(x) =\phi(g\cdot x)$, and the \invol, $\phi(\ov x)= \ov{\phi(x)}$.

\subsection{Free monoids with \invol and an $H$-action}\label{sec:fmi}
By an \emph{alphabet} we mean a finite set $\Sig$ with \invol. (Since the identity is an \invol of $\Sig$ this covers the case of  monoids without a predefined involution.) 
The elements of $\Sig$ are called 
\emph{letters} (or \emph{symbols}). By $\Sig^*$ (by $\Sig^+$ resp.) we denote the free monoid (free semigroup resp.) over 
$\Sig$.  The elements of a free monoid are called 
\emph{words}. The empty word in a free monoid is also denoted by $1$ as in other monoids.  We have $\Sig^+= \Sig^*\sm \os 1$.
The \invol extends 
to 
$\Sig^*$: for a word $w = a_1 \cdots a_m$ with $a_i\in \Sig$
 we let  $\ov{w} = \ov{a_m} \cdots \ov{a_1}$.  
The monoid $\Sig^*$ is called the \emph{free monoid with involution over $\Sig$}. 
If $\ov a = a$ for all $a \in \Sig$ then $\ov{w}$ is simply the word $w$ read from right-to-left.
The length of word $w$ is denoted by $\abs w$, and ${\abs w}_{a}$
counts how often a letter $a$ appears in $w$.

If a group $H$ acts on $\Sig$,
then $g\in H$ acts on $w= a_1\cdots a_m$ with $a_i\in \Sig$ by
$$g(w) = g(a_1)\cdots g(a_m).$$
A letter $a\in\Sig $ is 
\emph{$H$-visible} in $w$ if $g(a)$ is a factor of $w$ for some $g\in {H}$. 

Sometimes it is useful to view a word of $w=a_1\cdots a_m$ with $a_i\in \Sig$  as a labeled linear order as follows. We let $\os{1\lds m}$ be set of \emph{positions}; and we label a position $1\leq p \leq m$ with the letter $w[p]= a_p$. 
For $1\leq i,j \leq m$ we denote by $[i,j]$ the \emph{interval}  $\os{i\lds j}$. The 
labels of the  interval define a factor  
$$w[i,j]= a_i\cdots a_j.$$ 
An \emph{occurrence} of a  factor $u$ in $w$ is an interval  $[i,j]$ such that $u =w[i,j]$.  Typically, a factor has several occurrences. 
For a position $1\leq p\leq |w|$ we  define its \emph{dual} position $\ov p$ by 
$\ov p = m+1-p$. The notion of duality extends to intervals $[i,j]$ with $1\leq i,j\leq |w|$ by  $\ov{[i,j]}= [\ov j,\ov i]$. Thus, the set of intervals is a set with \invol.

 A word $w$ such that $w=\ov w$ is called \emph{self-involuting}, and for such $w$ we have
$\ov{w[i,j]} = w[\ov j,\ov i]$.

\subsection{Automata, rational and recognizable subsets in a monoid}\label{sec:naratrec}
For notation and results in the this subsection we refer to the classical textbook~\cite{eil74}. A 
\emph{regular language} in finitely generated free monoids
 can be defined via a \emph{nondeterministic finite automaton} or via recognizability using a \hom to a finite monoid, to mention just two possible definitions. We need the corresponding notions subsets for other monoids, too. 

Let $M$ be any monoid (not necessarily equipped with an \invol). 
A \emph{nondeterministic automaton} over $M$ is a directed arc-labeled graph $\cA$ {denoted as a tuple $\cA=(Q,M,\del,\inI,\finF)$.
The vertices of $\cA$ form the set $Q$ of \emph{states}, with subsets  $\inI$ of
\emph{initial} and $\finF$  \emph{final}  states. We write $\cA=\es$ if there are no states. 
The arcs are called \emph{\tra{s}} and they  are  labeled with elements of the monoid $M$. We represent the set of transitions  $\del$ as a subset of $Q\times M \times Q$.} 
A \tra labeled by $1\in M$ is called an \emph{$\eps$-transition}. 
In pictures we draw a \tra $(p,h,q)$ as $p \arc h q$.
We say that 
$m\in M$ is \emph{accepted} by the automaton $\cA$ if
there exists a path from some initial to some final state such that multiplying the  labels together yields $m$. 
This defines the \emph{accepted 
language} $L(\cA) = \set{m \in M}{m \text{ is accepted by } \cA}$.   

Often we specify $M$ together with a set $\Sig$ of generators or, more generally, together with \hom
$\pi$ from the free monoid $\Sig^*$ to $M$. In that case, we may denote $\cA$ alternatively 
as $\cA=(Q,\Sig,\del,\inI,\finF)$ 
where $\del\sse Q\times \Sig^* \times Q$. 
This allows two natural interpretations 
of $L(\cA)$:   first  as the set of words $L(\cA)\sse \Sig^*$ obtained by reading  $\cA$ as a shorthand for $(Q,\Sig^*,\del,\inI,\finF)$;  second 
as $L(\cA)\sse M$ by identifying $L(\cA)\sse \Sig^*$
with $\pi(L(\cA))$. If the distinction is crucial, we write $L(\cA)$ and $\pi(L(\cA))$. However, sometimes
a sloppy notation $L(\cA)\sse M$ is used. There will be however no risk of confusion. 

A \emph{\subauto} $\cA'$ of $\cA= (Q,M,\inI,\finF)$ is an automaton 
$\cA'= (Q',M,\inI',\finF')$ such that $Q'\sse Q$, $\del'\sse \del$, $\inI'\sse \inI$, and $\finF'\sse \finF$.

An automaton is called \emph{trim} if every {state} is on some path from an initial to a final {state}. For a trim automaton $\cA$ we have 
$L(\cA)\neq \es$ \IFF $\cA \neq \es$. Clearly, every automaton $\cA$ contains a trim \subauto $\cA'$ such that $L(\cA') =L(\cA)$.

If the set of transitions is finite, then we call $\cA$ a \emph{nondeterministic finite automaton} (or \emph{NFA} for short). A subset $L\sse M$ is  called \emph{rational} if 
$L$ is accepted by some NFA over $M$.

A subset 
$L\sse M$ is called \emph{recognizable} if there is a \hom $h\colon M\to N$ to a finite monoid $N$ such that $\oi h (h(L)) =L$.   
In case that $M$ is a finitely generated free monoid, the notion of rational and recognizable subsets coincide; so these subsets  are also called \emph{regular}. 
It follows rather easily  that a monoid $M$ is finitely generated \IFF all recognizable 
subsets are rational \cite{mck64}. 
Finite subsets are always rational, but  finite subsets in a group  are recognizable \IFF the group is finite. 

\subsection{From NFAs to Boolean matrices}\label{sec:NFA2Bo}
Nondeterministic finite automata encode regular languages in a concise and natural way. It is convenient to work in an algebraic framework with recognizing \morph{s}, too. Let us recall a well-known and classical construction.

Let $\cA=(Q,{\Sig},\del,\inI,\finF)$ be any NFA with 
$m$ states. Then we can assume that $Q=\os{1\lds m}$, and we represent \tra{s} as a mapping to Boolean $m\times m$ matrices as follows. For each letter $a\in\Sig$ we define a matrix
$\mu_\cA(a)\in\B^{m\times m}$ by 
\begin{equation}\label{eq:nfa2mat}
(\mu_\cA(a))_{s,t} =1\iff a \in L(Q,{\Sig},\del,\os{s},\os t).
\end{equation}
We obtain a \hom 
$\mu_\cA\colon {\Sig}^*\to \B^{m\times m}$
such that for all  $w \in \Sig^*$ we have
$$w\in L(\cA)\iff \mu_\cA(w)\in \mu_\cA(L(\cA) )\iff \exists s\in \inI \;\exists t\in \finF:(\mu_\cA(w))_{s,t}=1).$$

\begin{example}\label{ex:redword}
Let ${\Sig}$ be an alphabet (with \invol)  and $H\leq \Aut({\Sig})$ be a subgroup of automorphisms. 
The set $R$ of words having a factor 
$e\ov e$ for some $e\in {\Sig}$ is regular and $R$ is invariant under the action of $H$. Let  $\F= {\Sig}^*\sm R$ be the complement: it is the set of \emph{reduced} words. The set  $\F$ is in canonical bijection with the 
group ${\Sig}^*/\set{e\ov e = 1}{e\in {\Sig}}$. 
 The language $R$ is accepted by an NFA (actually a DFA) with 
$1+\abs {\Sig}$ states. Hence, $\B^{m\times m}$ recognizes them where $m=1+\abs {\Sig}$.

The size of $\B^{m\times m}$ is $2^{m^2}$, but there is a much smaller monoid $N$ which recognizes $R=\bigcup\set{\Sig^*e\ov e\Sig^*}{e\in \Sig}$ and hence $\F$, too. 
 The elements of $N$ are $1$, $0$, and the pairs $(a,b)$ in ${\Sig}\times {\Sig}$. The elements $1$ and $0$ act as the neutral element and a zero, respectively. 
The multiplication for the other elements is given by $(a,b)\cdot (c,d)= (a,d)$ if
$b\neq \ov c$ and $(a,b)\cdot (c,d)= 0$, otherwise.
The involution  is given by $\overline{(a, b)} = (\ov b, \ov a)$. 
In effect, $N$ ``remembers" the first and last letters of elements in $\F$.  A pair 
$(a,b)$ switches to $0$ once  a factor $e\ov e$ is recorded.

 It is an $H$-monoid by the natural action of $H$ induced by the action of $H$ on ${\Sig}$.
Consider the \morph of $H$-monoids 
$\mu\colon {\Sig}^*\to N$ which is defined by $\mu(a) = (a,a)$. Then we have
$R=\oi{\mu}(0)$ and $\F=\oi{\mu}(N\sm \os 0)$. 
The size of $N$ is therefore $2+{\abs \Sig}^2 - \abs \Sig$. Therefore each element in $N$ can be specified by at most $1+\log_2 (1+\abs \Sig)$ bits. 

\end{example}

\section{Regular languages in presence of an \invol and an $H$-action}\label{sec:regiH}
The application of this section is stated precisely in \prref{prop:iGimons}.
We give a  construction which allows us to handle regular constraints for twisted word equations using \morph{s} to finite $H$-monoids.

We proceed in two steps. The first step 
forces \homs to respect the \invol. This part is from the arXiv version of \cite{CiobanuDiekertElder2016ijac}  which was inspired by \cite{dgh05IC}.  
We repeat that construction. 
Let $N$ be any monoid.  We define its dual monoid $N^{\mathrm{op}}$ to use  the same set 
$N^{\mathrm{op}}=N$, but $N^{\mathrm{op}}$ is equipped with a new multiplication $x \circ y= yx$. 
In order to indicate whether we view an element in the monoid $N$ or $N^{\mathrm{op}}$, we use a flag: 
for $x \in N$ we write $x^{\mathrm{op}}$ to indicate the same element in $N^{\mathrm{op}}$. Thus, we can suppress the symbol $\circ$ and we simply write 
$x^{\mathrm{op}}  y^{\mathrm{op}}= (yx)^{\mathrm{op}}$. The notation is intended to mimic transposition in matrix calculus, where the dual operation is just the transpose.
Similarly, we  write $1$ instead of $1^{\mathrm{op}}$ which is true for the identity matrix as well. The 
direct product $N\times N^{\mathrm{op}}$ becomes a monoid with involution 
by letting $\ov {(x,y^{\mathrm{op}})} = (y,x^{\mathrm{op}})$. Indeed, 
$$\ov {(x_1,y_1^{\mathrm{op}}) \cdot (x_2,y_2^{\mathrm{op}})} 
= (y_2 y_1, (x_1x_2)^{\mathrm{op}}) 
= \ov {(x_2,y_2^{\mathrm{op}})}\cdot \ov {(x_1,y_1^{\mathrm{op}})}.$$
The following observations are immediate.
\begin{itemize}
\item If $N$ is finite then $N\times N^{\mathrm{op}}$ is finite, too. 
\item We can embed $N$ into  $N\times N^{\mathrm{op}}$ by a \hom $\iota\colon  N \to N\times N^{\mathrm{op}}$ defined by $\iota(x) = (x,1)$. Note that
if $\eta\colon  N\times N^{\mathrm{op}}\to N$ denotes the projection onto the first component, then $\eta \iota = \id{N}$.
\item If $M$ is a monoid with involution and $\rho\colon  M \to N$ is a \hom of monoids, then 
we can lift  $\rho$ uniquely to a \morph  $\opphi\colon  M \to N\times N^{\mathrm{op}}$ of  monoids with \invol such that 
we have $\rho = \eta \opphi$. Indeed, it is sufficient and necessary to define 
$\opphi(x) = (\rho(x),\rho(\ov x)^{\mathrm{op}})$.
\end{itemize}

\begin{example}[\cite{dgh05IC}]\label{ex:dgh05IC}
Let $M= \Bn$. Then  $M\times M^{\mathrm{op}}
$ 
is a submonoid of the set of $2n \times 2n$-Boolean matrices: 
$$\Bn\times ({\Bn})^{\mathrm{op}}= \set{\vdmatrix P00{Q^T}}{ P,Q \in \Bn} \text { with }  
\overline{\vdmatrix  {P}00 {Q^T}}  =
\vdmatrix {Q}00{P^T}. 
$$
In the line above $P^T$ and $Q^T$ are the transposed matrices. \end{example}

Now, we switch to the new part of our construction.
 For readers familiar with wreath products it might be helpful to say that the following is  a wreath product construction.
 Let $N$  be a monoid with \invol. Consider the direct product $N^H$, which is the set of maps from $H$ to $N$.   
We denote the elements of $N^H$ 
by tuples $(n_g)_{g}$ with the interpretation that $g\in H$ is mapped to $n_g\in N$. It is a monoid by pointwise multiplication with \invol $\ov{(n_g)_{g}} = (\ov{n_g})_{g}$. The monoid $N$ embeds into $N^H$ by sending $n$ to the constant map $(n)_g$.   
We let act $H$ on $N^H$ by
$$f \cdot (n_g)_g= 
(n_{gf})_{g}.$$
Now, let  $M$ be an $H$-monoid with \invol and let $\psi\colon M\to N$ be a \morph of monoids with \invol, then we extend it to 
$\wt \psi\colon  M\to N^H$ by 
$$\wt \psi(x) = (\psi(gx))_g.$$
The \hom $\wt \psi$ respects the \invol since 
$$\wt \psi(\ov x) = (\psi(g\ov{x}))_g = \ov{(\psi(g x))_g};$$
and it respects the action of $H$ since 
$$\wt \psi(fx) = (\psi(gfx))_g = f \cdot \wt \psi(x).$$
Moreover, $\psi$ factorizes though $\wt \psi$ because $\wt \psi(x) = (\psi(gx))_g$ 
implies $\psi= \eta_1 \wt \psi$ where $\eta_1((n_g)_g) = n_1$.

\begin{center}
\begin{figure}\begin{center}
\begin{tikzpicture}[
	xscale=5, yscale=1.2			
]
\path (0.3,-2) node (0) {$M$};
\path (1,0) node (1) {$(N\times N^{\mathrm{op}})^H$};
\path (1,-1) node (2) {$N\times N^{\mathrm{op}}$};
\path (1,-2) node (3) {$N$};

\draw [->, >=latex] (0) -- (1) node[midway, above] {$\wt \phi$};
\draw [->, >=latex] (0) -- (2) node[midway, above] {$\opphi$};
\draw [->, >=latex] (1) -- (2) node[midway, right] {$\eta_1$};
\draw [->, >=latex] (2) -- (3) node[midway, right] {$\eta$};
\draw [->, >=latex] (0) -- (3) node[midway, above] {$\phi$};
\end{tikzpicture}
\caption{A lifting of a \hom $\phi$ to an $H$-\Hmorph $\wt \phi$. }

\label{fig:recHNi}    \end{center}
   \end{figure}
    \end{center}

If we start with a \hom $\phi$  {}from an $H$-monoid with \invol $M$ to a monoid $N$ 
without \invol, then $\psi$ means the \morph $\opphi\colon  M \to N\times N^{\mathrm{op}}$; and $\wt \phi$ is a shorthand for  $\wt \opphi$. Thus, the constructions above yield the commutative diagram as in \prref{fig:recHNi}. 
In that figure $\phi=\eta \eta_1 \wt\phi  $ is a \hom, $\opphi=\eta_1\wt \phi$ is a \morph
of monoids with \invol, and $\wt\phi$ is an $H$-\Hmorph. As a direct consequence we obtain the following proposition. Recall that an $H$-monoid is, by definition, a monoid with \invol.
\begin{proposition}\label{prop:iGimons}
Let $H$ be a finite group that acts on a finite alphabet $A$; and hence 
via a length and involution preserving action on $A^*$. Then all recognizable subsets of $A^*$ can be recognized by some $H$-\Hmorph to a finite $H$-monoid. 
More precisely, let $\phi\colon A^*\to N$ be a \hom to a finite monoid; and $L=\oi\phi(F)$
for some $F\sse N$. Then we have $L=\oi{\wt\phi}(\wt F)$ where 
$\wt F= \oi {\wt \eta}(F)$   and $\wt \eta = \eta\eta_1$.
\end{proposition}

\subsection{Stabilizers}\label{sec:ssgm} Let $H$ be a finite group acting on an alphabet  ${A}$  via a \hom 
$\psi\colon H\to \Aut({A})$. We assume that $H$ is given 
by its multiplication table. The table can be  stored with $\Oh(|H|^2\log |H|)$ bits.
We also need a way 
to represent the action of $H$ and the stabilizer subgroups of $H$. The action is recorded  by 
writing down for each $f\in H$ the element  $\psi(f)$ as a permutation 
of ${A}$. To do this, we write $\psi(f)$ as a set of 
pairs $\psi(f)=\set{(a,f(a))}{a\in {A}}$. Thus, the action of $H$ on ${A}$ can be stored with $\Oh(|H| \abs {A} \log |{A}|)$ bits.

For a word $w\in {{A}}^*$  we denote by  $H_w$ its \emph{stabilizer}: 
$$H_w=\set{f\in H}{f(w) =w}.$$ 
stabilizer are subgroups; and the set of subgroups of the form $H_w$ form a commutative monoid $\ST(H)$ where the operation is intersection, the identity element is $H$, and the involution is the identity. Indeed, we have $$H_{uv} =\{g\in H\mid g(uv)=g(u)g(v)=uv\}= 
H_u\cap H_v,$$ $$H_{\ov u}= \{g\in H \mid g(\ov u)=\ov{g(u)}=\ov u\}=H_u$$ and $$H_{g(u)}=\{f\in H\mid f(g(u))=g(u)\}=\{f\in H\mid \oi g (f(g(u)))=u\}=\oi g H_u g$$ for all $u,v\in {{A}}^*$,  $g\in H$.
\Ip $H$ acts on $\ST(H)$ by conjugation, and 
 $\ST(H)$ is therefore an $H$-monoid. Let $\SG(H)$ denote the set of all subgroups of $H$, then  $\nu(u)=H_u$ yields a canonical surjective \morph
$$\nu\colon {{A}}^* \to \ST(H) = \set{H_w\in \SG(H)}{w\in {{A}}^*}.$$ 

A basic test is to answer ``$f\in H_w$?''.
This is easy: for $w=a_1\cdots a_m$ with $a_i\in {{A}}$ we check
one after another that $f(a_i) = a_i$ for $1\leq i \leq |w|$. 
This enables an efficient test to decide whether or not 
$H_u\sse H_v$. For each $f\in H$ one after another 
we test whether $f\in H_u$ implies $f\in H_v$. 
In particular, we can answer ``$H_u=H_v$?''.
 \begin{lemma}\label{lem:sizeSTH}
 We have $\ST(H) = \set{H_w}{w\in {{A}}^* \wedge |w| \leq \log_2 |H|}.$
 \end{lemma}
 
\begin{proof}
 For each $w\in {{A}}^*$ and $b\in {{A}}$ we either have 
$H_w= H_{wb}$ or $|H_w|\geq 2 |H_{wb}|$ because $H_{wb}=H_w\cap H_b$. So, if they are not equal, then their intersection is a subgroup which has index at least $2$ in $H_w$.
\end{proof}

The idea is therefore to use  words of length at most $\log_2|H|$ to represent stabilizers and to perform the calculations for stabilizers on these words. 
The representation is not unique, but this does not matter for our application. 
 
 A main task is to compute a word $w$ of length at most $\log_2|H|$ such that 
$H_w= H_u \cap H_v$ (when $u$ and $v$ satisfy $|u|,|v| \leq \log_2|H|$). This can be done efficiently according to the following lemma. 
 \begin{lemma}\label{lem:smallSTH}
Every element in the commutative monoid $\ST(H)$ of stabilizers can be
 represented by  a word in $A^*$ of length at most $
\log |H|$, thus with at most 
$\Oh(\log \abs H  \cdot\log |{A}|)$ bits.
Using this representation, multiplication (that is, intersection) and computing the $H$-action (that is, conjugation), can be done in  space $\Oh(\log \abs H  \cdot\log |{A}|)$. 
\end{lemma}

\begin{proof}
Let $uv = a_1 \cdots a_m$. 
We have to compute a word $w$ of length at most $\log_2|H|$ 
such that $H_w= H_{a_1\cdots a_m}$.

We run a loop for $i=1$ to $m$. At each step 
we have computed a word $u_{i-1}$  such that $H_{u_{i-1}}= H_{a_1\cdots a_{i-1}}$ with the invariant $2^{|u_{i-1}|}|H_{u_{i-1}}| \leq |H|$
(initially we let $u_0=1$). There are exactly three mutually disjoint cases.
\begin{enumerate}
\item If $H_{u_{i-1}}\sse H_{a_i}$ then we let $u_i = u_{i-1}$.
\item If $H_{a_i}\ssneq H_{u_{i-1}}$, then we let $u_i=a_i$. 
\item If $H_{a_i}$ and $H_{u_{i-1}}$ are incomparable with respect to containment, then we let 
$u_i = u_{i-1}a_i$.
\end{enumerate}
Each case  keeps the invariant because, by induction, $2^{|u_{i}|}|H_{u_{i}}| \leq |H|$.
\end{proof}

\begin{remark}\label{rem:add}
The reader can easily check that our computation of a word 
$w$ with $H_w= H_{a_1\cdots a_m}$ yields a word $w$ of pairwise different letters. 
So, we could actually put a bound  $|w| \leq \min\os{\log_2|H|, |A|}$ on its length. 
\end{remark}
\subsection{ \HNs}\label{sec:ngi}
In the following  $H$ denotes a finite group and $N$ denotes  a finite $H$-monoid. 
Let $M$ be a set (resp.~be a monoid) with \invol and  $$\mu\colon M \to N$$ be a \morph. 
We say $M$ (together with $\mu$) is an \emph{\HNalp} (resp.~an \emph{\HN}) if $H$ acts on $M$
such that $\mu(g\cdot x) = g \cdot \mu(x)$. 
For example, the identity map on $N$ makes $N$ itself to an \HN.  

A \morph between \HNs is an 
$H$-\Hmorph $\phi\colon M'\to M$ such that $\mu\phi = \mu'$. Thus, if $M$ is an \HN and $M'$ is a monoid with \invol where $H$ acts, then every
$H$-\Hmorph  $\phi\colon M'\to M$ turns $M'$ into an \HN where $\mu'$ is uniquely defined by the equation $\mu\phi = \mu'$.
The use of  \HNs is natural in our setting: the $H$-action is due to a group action on letters, and the finite monoid $N$ is used for the specification of rational constraints. It is clear that the specification of constraints has to be compatible with the group action.

\subsection{Free \HNs and types}\label{sec:twngi}
Let $B$ and $\cY$ be two disjoint \HNalps. 
 We call $B$ the alphabet of \emph{constants} and 
$\cY$  the set of twisted variables.  The free monoid with \invol $(B\cup \cY)^*$ becomes an \HN where $\mu\colon (B\cup \cY)^*\to N$ is induced by $B\cup \cY$. \begin{align*}
\ov{x_1\cdots x_m}&= \ov{x_m}\cdots \ov{x_1},\\
g\cdot({x_1\cdots x_m})&= g\cdot({x_1})\cdots g\cdot(x_m).
\end{align*}
By $\theta\sse (B\cup \cY)^*\times (B\cup \cY)^*$ we denote a finite  homogeneous  relation. Here as usual, a relation is called \emph{homogeneous} if  $(x,y)\in \theta$ implies 
$|x|=|y|$. If $(x,y)\in \theta$ then we also say that $(x,y)$ is a \emph{defining relation} 
because the algebraic object we are interested in is the 
quotient monoid 
$$(B\cup \cY)^*/\set{x=y}{(x,y)\in \theta}.$$
We need more structure of this quotient monoid; \ip $\mu\colon (B\cup \cY)^*\to N$
should  induce a \morph of \HN{s}. Actually we wish more, therefore we impose the following technical restrictions on 
$\theta$; and then we call $\theta$ a \emph{type} (and for a variable $X$ we also define the   \emph{type of $X$} denoted $\theta(X)$ below).
\begin{enumerate}
\item $(x,y)\in \theta$ implies $\mu(x) = \mu(y)$,  
$(\ov y,\ov x)\in \theta$, and 
$(f(x),f(y))\in \theta$ for all $f\in H$, even if these relations are not listed in the specification of $\theta$. 
\item If a (twisted) variable $X$ appears in $\theta$ (that is $|xy|_X\geq 1$ for some $(x,y)\in \theta$), then we call $X$ \emph{typed}. 
For a typed variable $X$ we require that there is  a unique primitive  word\footnote{Recall that $p$ is primitive \IFF it cannot be written as $p = r^e$ with $e \geq 2$.}
 $p\in B^*$ such that  $(Xp,\, pX) \in \theta$.   
We define $\theta(X)=p$, and 
 say that $\theta(X)$ is the \emph{type of $X$}. 
\item 
For $(x,y)\in \theta$ we allow exactly three possibilities:
\begin{enumerate}
\item[(i)] $(x,y) = (ab,ca)$ with $a,b,c\in B$.
\item[(ii)] $(x,y) = (X\,\theta(X),\,\theta(X)\,X)$ for variables $X$.
\item[(iii)] $(x,y) = (Xa,aY)$ where $a\in B$ and $X,Y$ are typed variables such that $X\neq Y$.
\end{enumerate}
\end{enumerate}

It is convenient to choose a subset $\cX\sse \cY$ 
which is closed under the \invol such that every $Y\in \cY$ has the form 
$Y = f(X)$ for some $X\in \cX$ and $f\in H$. 
In the following, 
by a \emph{variable} we typically mean $X\in \cX$ and thus, every twisted variable $Y\in \cY$ can be written as $f\cdot X$ for some $f\in H$ and $X\in \cX$. We assume $X\neq \ov X$ for all variables. 
Having chosen $\theta$ and $\cX$ 
we denote by $M(B,\cX,\theta,\mu)$ the following quotient monoid (and an \emph{\HN with type $\theta$}): 
$$M(B,\cX,\theta,\mu) = (B\cup \cY)^*/ \set{x=y}{(x,y) \in \theta}.$$ 
Point (1) from above makes sure  that one can extend the involution, the morphism $\mu$ and the action of $H$ to the quotient $M(B,\cX,\theta,\mu)$.
The homogeneity condition for $\theta$ makes it possible to solve the uniform word problem in $M(B,\cX,\theta,\mu)$ in nondeterministic quasi-linear space: 
\begin{lemma}\label{lem:nonuni}
There is an $\NSPACE(n \log n)$ algorithm which performs the following task.
The input is an alphabet $B$, a homogeneous relation $\theta\sse B^*\times B^*$, and two words $u,v\in B^*$ such that 
$$ |uv| + \abs B +\sum_{(x,y)\in \theta}|xy| \leq n.$$
The output is ``yes'' if $u=v$ in the quotient monoid 
$B^*/\set{x=y}{(x,y)\in \theta}$ and ``no'' otherwise. 
\end{lemma}

\begin{proof}
If $u=v$ in 
$B^*/\set{x=y}{(x,y)\in \theta}$, then nondeterministically we can apply 
rewriting rules from $\theta$ (which preserve length) to $u$ until we see $u=v$ in the free monoid $B^*$. We get the ``no'' answer because $\NSPACE(n\log n)$ is closed under complementation by the theorem of  Immerman-Szelepcs{\'e}nyi, see for example \cite{pap94}. 
\end{proof}

Using \prref{lem:nonuni} we represent elements
in $M(B,\cX,\theta,\mu)$ by words over $(B\cup \cY)^*$.
For $\theta=\es$ we obtain 
$(B\cup \cY)^*= M(B,\cX,\es,\mu)$. By $M(B,\theta,\mu)$ we denote the \HN submonoid with type $\theta$ which is generated by $B$. \Ip $B^*=M(B,\es,\mu)$.
If $\theta\cap B^*\times B^*=\es$, then $M(B,\es,\mu)= B^*$ is a free submonoid of 
$M(B,\cX,\theta,\mu)$.
If $H$ acts without fixed points on $\cY$, then we identify $\cY= H\times \cX$ and the action becomes $g\cdot (f,X) = (gf,X)$. 
Later we will write typed variables using a special bracket notation $[X,p]$.
For complexity issues we  will only allow $\theta$ which satisfy
$\abs\theta \in \Oh({\abs H}\Abs{\cS}^2)$ where $\Abs{\cS}$ is specified in \prref{eq:AbscS} below.

\subsection{\edtol languages and relations}\label{sec:Languages}
The acronym {EDT0L} refers to \emph{{\textbf{E}xtended, \textbf{D}eterministic, \textbf{T}able, 
\textbf{0} interaction, and \textbf{L}indenmayer}}. See the handbook  \cite{rs97vol1} for the many  results about \textbf{L}-systems.
Let $A$ be  an alphabet. 
A subset $L$ in a $k$-fold direct product $A^* \times \cdots \times A^*$ is called a {\emph{EDT0L relation}}
if there is some (extended) alphabet $C$ with $d_1\lds d_k \in C$ such that $A\sse C$ and a rational set 
$\cR \sse \End(C^*)$ of \Endos over $C^*$ such that 
$$L = \set{(h(d_1) \lds h(d_k))}{h \in \cR}.$$
The classical situation refers to $k=1$. In that case we speak about an \edtol language; and our definition uses a characterization of \edtol languages due to 
\cite{Asveld1977}. The connection is as follows. 
Let $\$$ be a symbol which is not in $A$ and  $L\sse A^* \times \cdots \times A^*$
be a \edtol relation, then 
$\set{w_1\$w_2\cdots \$w_k}{(w_1\lds w_k)\in L}$ is an \edtol language in the usual sense over the alphabet $A\cup \os \$$. 
It should also be noted that the class of  \edtol languages coincides with the class
of HDT$0$L languages (\cite[Thm.~2.6]{rs97vol1}).   

We say 
$L$ is an \emph{effective} \edtol relation if  there is an effective description of an NFA $\cA$ with  transitions labeled by  ``deterministic tables'' of pairs $(c,u_c) \in C \times C^*$ (encoding the \Endo which maps $c$ to $u_c$ (and $\ov c$ to $\ov{u_c}$))\footnote{Without restriction we can assume each transition is labeled by an endomorphism which changes at most one pair of letters $c,\overline c$.} and  letters $d_1\lds d_k \in C$ such that 
$(w_1\lds w_k) \in L$ \IFF there is some $h\in L(\cA)\sse \End(C^*)$ such that 
$(w_1\lds w_k)= (h(d_1) \lds h(d_k))$.

\section{Twisted word equations}\label{sec:tweqs}

\subsection{The initial setting}\label{sec:inset}
We begin with a nonempty alphabet of constants $A$, and a list of $2k$ variables ${\cV_0}$ (as always, both with \invol) and a finite group $H$ where $H$ acts on $A$ via a \hom $\psi\colon H\to \Aut(A)$. \Ip $|\psi(H)|\leq |A|!\leq |A|^{|A|}$. 
As above, $H$ acts on $H\times {\cV_0}$ by $f\cdot (g,X) = (fg,X)$. 
For $w\in A^*$
and $f\in H$ we also use the notation $f(w) = (f,w)$. Hence, we may represent elements in  $(A\cup (H\times {\cV_0}))^*$ by words in  $(H\times (A^* \cup {\cV_0}))^*$. 
We abbreviate $(1,x)$ as $x$ for $x\in A^* \cup {\cV_0}$.  By $\mu_0\colon A^* \to N$ we mean a \hom which respects the \invol and the action of $H$. Thus $A^*$ is, via  $\mu_0$, an \HN. Assume that 
$\mu_0$ has been extended to a mapping $\mu_0\colon A^*\cup {\cV_0} \to N$ such that
$\mu_0(\ov X) = \ov{\mu_0(X)}$, then $\mu_0$ extends to a \morph
$\mu_0\colon (A\cup (H\times {\cV_0}))^* \to N$ of \HNs by $\mu_0(f,X)= f\cdot \mu_0(X)$. 
Initially we work over free monoids. 
\begin{definition}\label{def:initsysS}
A system $\cS$ of \emph{twisted word equations with regular constraints} over $A$ and 
${\cV_0}$ is given by the following data: 
\begin{itemize}
\item A list of $2k$ variables such that ${\cV_0}= \os{X_1, \ov{X_1} \lds X_{k}, \ov{X_{k}}}$.
\item The set of twisted variables becomes $\cY =  H\times {\cV_0}$. 
\item A set of pairs $\set{(U_i,V_i)}{1 \leq i \leq s}$ where 
$U_i, V_i \in (A\cup \cY)^*$ .
\item A \morph $\mu_0\colon (A\cup  \cY)^*\to N$ of $H$-monoids, with $N$ finite.
\end{itemize}
A \emph{solution} of $\cS$ is  a \morph of sets with \invol 
 $\sig\colon {\cV_0} \to A^*$ which is (uniquely) extended
to an $A$-\morph of \HNs $\sig\colon (A\cup \cY)^* \to A^*$ such that 
\begin{itemize}
\item $\sig(U_i)=\sig(V_i)$ for all pairs $(U_i,V_i)$.
\item $\mu_0 \sig(X) = \mu_0(X)$ for all variables. Hence,  $\mu_0 \sig=\mu_0$. 
\end{itemize}
\end{definition}
As usual, a pair  $(U_i,V_i)$ representing a twisted equation
is simply written as $U_i=V_i$.

\begin{example}\label{ex:twereg}
Let $A=\{a,\ov a,b,\ov b\}, {\cV_0}=\{X,\ov X, Y, \ov Y, Z, \ov Z\}$,  $f,g\in H$ defined by $f(a) = b, g(a) = \ov a, g(b) = b$, 
$ U_1  =(f,X)a(g,\ov Y)$,  $V_1 =Z$,
 $U_2  =(f,Y)b$,  $V_2  =\ov ab(g,X)$, 
 $U_3 =Xa$,   $V_3 =b(f,X)$
  and  (for simplicity) $\mu_0(x)=1$ for all $x\in A\cup {\cV_0}$.  A pictorial representation of the example is shown in \prref{fig:extwereg}.
 The reader is invited to verify that one possible solution is
  $\sig(X)= bab, \sig(Y)= \ov baa\ov b, \sig(Z)= abaabaab$, and a second solution is $\sig(X)= b, \sig(Y)= \ov ba, \sig(Z)= aaab$. 
\end{example}

\begin{figure}[h!]
\begin{center}
\begin{tikzpicture}[scale=.6]

\draw  (0,4.5) -- (8,4.5) -- (8,6.5) -- (0,6.5) -- cycle;
\draw  (0,5.5) -- (8,5.5);

\draw  (9.5,4.5) -- (14.5,4.5) -- (14.5,6.5) -- (9.5,6.5) -- cycle;
\draw  (9.5,5.5) -- (14.5,5.5);

\draw  (16,4.5) -- (20,4.5) -- (20,6.5) -- (16,6.5) -- cycle;
\draw  (16,5.5) -- (20,5.5);

\draw  (10.5,4.5) -- (10.5,5.5);
\draw  (11.5,4.5) -- (11.5,5.5);
\draw  (17,4.5) -- (17,5.5);
\draw  (3,6.5) -- (3,5.5);
\draw  (4,6.5) -- (4,5.5);
\draw  (13.5,6.5) -- (13.5,5.5);
\draw  (19,6.5) -- (19,5.5);

\node[align=left, above] at (1.5,6.5) {$X$};
\node[align=left, above] at (3.5,6.5) {$a$};
\node[align=left, above] at (6,6.5) {$\ov Y$};
\node[align=left, below] at (4,4.5) {$Z$};
\footnotesize
\node[align=left, below] at (6,6.5) {$\downarrow g$};
\node[align=left, below] at (1.5,6.5) {$\downarrow f$};

\normalsize
\node[align=left, above] at (11.5,6.5) {$Y$};
\node[align=left, above] at (14,6.5) {$b$};
\node[align=left, below] at (10,4.5) {$\ov a$};
\node[align=left, below] at (11,4.5) {$b$};
\node[align=left, below] at (13,4.5) {$X$};
\footnotesize
\node[align=left, below] at (11.5,6.5) {$\downarrow f$};
\node[align=left, above] at (13,4.5) {$\uparrow g$};

\normalsize
\node[align=left, above] at (17.5,6.5) {$X$};
\node[align=left, above] at (19.5,6.5) {$a$};
\node[align=left, below] at (16.5,4.5) {$b$};
\node[align=left, below] at (18.5,4.5) {$X$};
\footnotesize
\node[align=left, above] at (18.5,4.5) {$\uparrow f$};

\end{tikzpicture}

\caption{Pictorial representation of \prref{ex:twereg}.}
\label{fig:extwereg}\end{center}
\end{figure}

If $\sig$ is a solution of $\cS$ we also say that $\sig$ \emph{solves} $\cS$. For $\os{X_1, \ov{X_1} \lds X_k, \ov{X_k}}$ the \emph{full solution set} $\cSol(\cS)$ of $\cS$ is defined as 
$$\cSol(\cS)= \set{(\sig(X_1)\lds\sig(X_{k})) \in A^* \times \cdots \times A^* }{\sig \text { solves $\cS$}}.$$

\subsection{The main result on twisted word equations}\label{sec:mtw}
Our main result shows that $\cSol(\cS)$ is an \edtol language, for which we can compute an effective 
description in polynomial space. 
In order to measure complexities accurately, we need a precise notion of input size. Let $\cS$ be  a system of twisted word equations with regular constraints over $A$ and 
${\cV_0}$ according to \prref{def:initsysS}.

We define the size of  $\cS$ using two parameters $\Abs{\cS}$ and 
 $m(\cS)$. Thus, the size is the pair $(\Abs{\cS},m(\cS))$. The first parameter  $\Abs{\cS}$  ignores the 
 size of the finite monoid $N$. It is the main parameter as we don't want that the complexity due to constraints dominates the overall complexity. The second parameter measures
 separately  the number of additional bits for handling the constraints.  We begin by defining  $\Abs{\cS}$. Let 
 \begin{equation}\label{eq:AbscS}
\Abs{\cS} = \abs H+\abs A + {k} + s  + \sum_{1\leq i \leq s} \abs{U_iV_i}.
\end{equation}
Recall that $2k$ is the number of variables, $s$ is the number of equations $U_i=V_i$, and $H$ denotes a finite group acting on $A$ and hence on $A^*$, too. 
We are interested in a situation only where
$A\neq \es$ and $\Abs{\cS} >4$. 

The finite monoid $N$ is also part of the input. We measure the size relative 
to $\Abs{\cS}$. Let $N$ be any finite $H$-monoid and $m_\cS(N)\in \N$ be a number such that elements of $N$ can be encoded by a number of bits which is at most 
\begin{equation}\label{eq:bitMon}
m_\cS(N) (2+\log|A|)\cdot\log \Abs{\cS}.
\end{equation}
Moreover, using this specification, monoid computations,
like computing the involution, the multiplication of two elements, and the action by $H$,
can be performed on a Turing machine in space $m_\cS(N)\Abs \cS \log \Abs \cS$. 
We let 
 \begin{equation}\label{eq:moncS}
m(\cS) = m_\cS(N) \text{ where $N$ is the monoid which appears in $\cS$}
\end{equation}
There are examples where the monoid  $N$ (which appears in $\cS$) is polynomially bounded in  $\Abs{\cS}$ and 
still $m(\cS)\in \Oh(1)$. However,  if $m(\cS)$ becomes 
 a polynomial in $\Abs{\cS}$, then we need to consider $m(\cS)$ separately for a finer analysis below $\PSPACE$.

An $H$-monoid $N$ is called \emph{small with respect to $\cS$} 
if $m_\cS(N)\in\Oh(1)$. The finite monoid recognizing 
reduced words (those without factors $a\ov a$ for $a\in A$) is small with respect to $\cS$, see \prref{ex:redword}. 
Being small is no restriction for the computability of the representation of the full \solu set as an \edtol relation -- we can always add trivial equations until $N$ becomes small with respect to $\cS$. 
Another example of a small monoid is the finite monoid $\ST(H)$ of stabilizers.
Its size depends on $A$ and $H$, but still it is small due to \prref{lem:smallSTH}.

Now, during the process we might wish to use direct products of (small) monoids.
For that the parameter $m_\cS$ behaves nicely: 
$$m_\cS(N_1 \times N_2) \leq m_\cS(N_1) +m_\cS(N_2).$$
Indeed, given $(n_1,n_2)\in N_1 \times N_2$ we can use the first $m_\cS(N_1)$
bits to encode $n_1$ and the last $m_\cS(N_2)$ bits to encode $n_2$. 
The operations on $N_1 \times N_2$ can be done component wise. 
\Ip a direct product of small monoids remains small.  

We are ready to state our main result which gives $\PSPACE$ as an upper bound for the complexity and a quasi-quadratic space bound if $N$ is small. 
\begin{theorem}\label{thm:central}
There is an  $\NSPACE(|H|  \cdot \Abs{\cS}^2\cdot m(\cS)\cdot  \log |A|\cdot  \log \Abs{\cS})$
algorithm which performs the following task. 
It takes as input 
 a system of twisted word equations $\cS$ with regular constraints. The system  
use  a set of constants $A$, a set of variables $\cV_0=\os{X_1, \ov{X_1} \lds X_{k}, \ov{X_{k}}}$, and the regular constraint is given by 
a \morph ${\mu_0}\colon A\cup\cV_0\to {N}$. The output is:
\begin{itemize}\item 
an extended alphabet $C$ of size $\Oh(|H|^2\Abs{\cS}^2)$;
\item distinguished letters $d_i\in C$ for each 
variable $X_i$;
\item
a trim NFA $\cAcS$ accepting a rational set of 
$A$-morphisms $L(\cAcS) \sse \End(C^*)$ such that 
\begin{equation}\label{eq:goodA}
\cSol(\cS) =\set{(h(d_{1})\lds h(d_{{k}})) \in C^* \times \cdots \times C^* }{h \in L(\cAcS)}.
\end{equation}
\end{itemize}
The algorithm stores intermediate equations with a length bound in $\Oh(|H|\Abs{\cS}^2)$.
Moreover, $\cSol(\cS)=\es$ \IFF $\cAcS= \es$;  and $\abs{\cSol(\cS)}<\infty$
\IFF $\cAcS$ doesn't contain any directed cycle. 
\end{theorem}

Let us comment on the rather complicated space bound 
$$|H|  \cdot \Abs{\cS}^2\cdot m(\cS)\cdot  \log |A|\cdot  \log \Abs{\cS}$$ which appears in the statement of the theorem. First, since $\abs H \leq \Abs \cS$ and $\abs C\in \Oh(|H|^2\Abs{\cS}^2)$, we can encode all letters by $\Oh(\log \Abs \cS)$ bits. Second, the $\mu$-value for the constraints changes dynamically: it is a priori not fixed for the extended alphabet $C$. 
So, it is enough 
to store the  $\mu$-value for each symbol which appears in intermediate equations. The  length bound on 
intermediate equations is in  $\Oh(|H|\Abs{\cS}^2)$. Each 
$\mu$-value is an element in $N$, which requires, by definition,  $m(\cS)\cdot  \log |A|\cdot  \log \Abs{\cS}$ bits for the encoding. 
Together, we need $\Oh(|H|  \cdot \Abs{\cS}^2\cdot m(\cS)\cdot  \log |A|\cdot  \log \Abs{\cS})$ bits to encode intermediate equations,

\begin{corollary}\label{cor:central}
Let $\cS$ be a  system of twisted word equations with regular constraints in variables $\os{X_1, \ov{X_1} \lds X_k, \ov{X_k}}$. Then   $\cSol(\cS)\sse A^*\times \cdots \times A^*$ is an effective  \edtol relation.
\end{corollary}

\begin{proof}
This is a formal consequence of \prref{thm:central}.
\end{proof}
Sections~\ref{sec:pretria} through \ref{sec:cm} are devoted to the proof of \prref{thm:central}.
The theorem  implies that we can decide in  \PSPACE\ (hence, in deterministic singly exponential time) whether $\cS$ is solvable and whether or not there are only finitely many solutions. 
The decision problem of whether a word equation  with regular constraints has a \solu is known to be \PSPACE-hard by \cite{koz77} because the intersection problem of regular languages is a special case. In our setting, if the finite monoid $N$ is small, then  the best known lower bound to date is $\NP$-hardness: it is the lower bound for deciding whether or not a linear Diophantine system over $\N$ has a \solu \cite{HU}.

\section{Preparation}\label{sec:pretria}
We begin the proof of \prref{thm:central} with some   technical preparations.
Sections~\ref{sec:sharks}--\ref{sec:sharks-zero}
 concern some reductions, and could easily
  be skipped in a first reading of the paper. These sections yield a reduction to the situation as stated in beginning of
\prref{sec:trisys}. We invite the reader to jump  directly to 
\prref{sec:trisys} and to read the parts in between only when necessary. 
\subsection{Reducing to faithful actions}\label{sec:sharks}
Recall that the action of $H$ on $A$ is given by a \hom $\psi\colon H\to  \Aut(A)$. We don't require that $\psi$ is injective because in some natural examples this is not the case, see \prref{sec:sl2z}. 
On the other hand  it is enough to prove \prref{thm:central} in the case where $H$ is actually a subgroup of $\Aut(A)$. Let us show how the reduction works.
The principal idea is to replace $H$ by $H/K$ where 
$K= \ker(\psi)$ is the kernel of $\psi$. 

If $M$ is any $H$-monoid, then the action of $H$ induces an action of $H/K$ on $M$ only if for all $f\in K$ and all $m\in M$ we have $f(m) = m$. In this case $M$ becomes 
an $H/K$ monoid: the action $g\cdot K(m) = g(m)$ is well-defined for all $g\cdot K\in H/K$. 
By definition of $K$, the free monoid $A^*$ is therefore an $H/K$-monoid. 
Inspecting the statement in \prref{thm:central}, there are two problems: the induced action of $K$ 
on the finite monoid is not trivial, in general. Moreover, the group acts $H$ freely 
on the set of variables $H\times \cX_0$, so there is no induced action of $H/K$ on this set unless $K$ is trivial. 
We address and solve both problems. 

Let us begin with the $H$-monoid $N$, then it has a largest $H$-invariant submonoid 
$N'$ where $K$ acts trivially. It is the submonoid  of $K$-invariant elements:
$$N'=\set{m\in N}{\forall f\in K:\, f(m)=m}.$$
The image $\mu_0(A^*)$ is a submonoid of $N'$, since 
 for all $f\in K$ and all  $w\in A^*$ we have 
 $f(\mu_0(w)) = \mu_0(f(w))= \mu_0(w)$. 
 However, the statement in \prref{thm:central}  doesn't  require that $\mu_0(X)$ takes values in $N'$. Let us show that $\cS$ is not solvable if there is some variable $X$ such that $\mu_0(X)\notin N'$.  Indeed, assume the contrary that there is a \solu $\sig\colon \cX_0\to A^*$  such that $\mu_0(X)\notin N'$.
 Then  there is 
some $f\in K$ such that 
 $$\mu_0(X) \neq f(\mu_0(X)) = f(\mu_0\sig(X))=  \mu_0(f(\sig(X))) = \mu_0\sig(X) = 
 \mu_0(X),$$
 which is a contradiction. We have $f(\sig(X))= \sig(X)$ because $K$ acts trivially on $A^*$. 
 
 Thus, as a first procedure in the proof of \prref{thm:central} we check that 
 $\mu_0(X)\in N'$ for all $X\in \cX_0$ (and therefore $\mu_0(f,X)\in N'$ for all $(f,X)\in H\times \cX_0$).
 The test
 runs over all 
 $f\in H$ and for each $f$ checks the following implication:
 \begin{equation}\label{eq:ceckK}
(\forall a\in A:\, f(a)=a) \implies f(\mu_0(X)) = \mu_0(X).
\end{equation}
 If the check is positive then $\mu_0(X)\in N'$ for all $X\in \cX_0$.
 If the check fails then we output 
 that $\cS$ is not solvable by defining $\cAcS=\es$. 
 We can perform the test within our given space bound by definition of $m(\cS)$. 

 After the test, we may assume that $\mu_0$ maps $A\cup (H\times \cX_0)$ to $N'$, and we replace $N$ by $N'$. We can use the same bit encoding for elements in $N'$ as we did for $N$, but if we have to guess an element in $N'$, we perform the test (\ref{eq:ceckK}). Thus, the parameter $m(\Phi)$ is still valid.

 In the second step we replace each variable
 $(f,X) \in H\times \cX_0$ by a fresh variable  $(f\cdot K,X) \in (H/K)\times \cX_0$.
 Again, this doesn't change $\cSol(\cS)$ since for every $H$-\Hmorph 
 $\sig\colon  H\times \cX_0 \to A^*$ and all $f\in K$ we have 
 $$\sig(f,X) = f(\sig(1,X)) = f(\sig(X)) = \sig(X)  \text{ and } \mu_0(f,X) = f(\mu_0(X))= \mu_0(X).$$
 
We are done.  We have shown the following statement. 
\begin{lemma}\label{lem:}
It is enough to prove \prref{thm:central} under the additional assumption that 
$\psi$ is injective. This means we can assume that $H$ is a subgroup of $\Aut(A)$. 
\end{lemma}

\subsection{Making the finite monoid $N$ larger}\label{sec:NdivN}
The aim in this section  is to replace $N$ by a larger monoid, which additionally encodes information about stabilizers
for all $x\in A \cup \cX_0$.  Up to a constant factor we don't change $m(\cS)$. 

Let 
$\ST(H)$ be the monoid of stabilizers, see \prref{sec:ssgm}. 
We define a 
\morph   $\mu_1\colon A\cup \cX_0\to N\times \ST(H)$ which maps 
a letter $a\in A$ to $(\mu_0(a),H_a)$ and 
each variable $X$ to some $(\mu_0(X),H_u)$ where $u\in A^*$ is any word of length at most $\log_2|H|$ by guessing $u$. The $H$-action on $N \times \ST(H)$ is inherited from the action on $N$ and the action on $\ST(H)$ by conjugation. Moreover, guessing is equivalent to taking the union over finitely many cases, see (\ref{eq:choicemus}). The union will give the same solutions we had before, and it will not introduce any new solutions.

The projection to the first component turns $N\times \ST(H)$ into an \HN.  
 Using $\mu_1$  we achieve the following:
\begin{itemize}
\item  for all $x\in (A\cup \cX_0)^*$, $\mu_1(x)\in N\times \ST(H)$ is a pair where the second component is $H_x$ which is represented by a word $u\in A^*$  of length at most $\log |H|$ such that $H_x=H_u$.
\end{itemize}

The switch to $\mu_1$ has  a price. By defining $\mu_1(X)$ we restrict the set of possible \solu{s}. The value $\mu_1(X)=(\mu_0(X),H_u)$ fixes the stabilizer $H_{\sig(X)}$ for a \solu $\sig$ to be the subgroup $H_u$. The number of choices (= nondeterministic guesses) to extend $\mu_0$ to $\mu_1$
is bounded by
\begin{align}\label{eq:choicemus}
(\abs{\ST(H)})^k\leq{\abs A}^{|\cX_0|\log |H|}.
\end{align}
These choices result in a splitting the original system into that many subsystems. Splitting is fortunately no problem since \edtol languages are  closed under finite union by taking the unions of the corresponding NFAs. 

At the end of this we rename $N\times \ST(H)$ as $N$ and  $\mu_1$ as $\mu_0\colon  A\cup \cX_0\to N$.

\subsection{Introducing a zero to $N$ and a marker symbol to $A$.}\label{sec:sharks-zero}
In the following it is convenient to have a special symbol $\#$, but we want to make sure no variable uses it, so we add $0$ to our constraint monoid. 
We next embed  our current $N$ into $N\cup \os0$ where $0$ is a fresh symbol not included in $N$ and $0$ acts as a zero in $N\cup \os0$. We turn it  into an $H$-monoid by defining $f(0)=0$ for all $f\in H$.

The monoid $N$ is an $H$-submonoid of $N\cup \os0$ and, by a slight abuse of language, we denote by
$\mu_0$ the induced mapping to the larger monoid  $N\cup \os0$ as well:xt
$$\mu_0\colon A\cup (H\times \cX_0) \arc{\mu_0} N \hookrightarrow N\cup \os0.$$
Without restriction (by adjusting constants if necessary) we may assume that $N\cup \os 0$ doesn't change the parameter $m(\cS)$.  Using $N\cup \os0$ instead of $N$ doesn't change $\cSol(\cS)$ because  $\mu_0(A\cup (H\times \cX_0)) \sse N$. Phrased differently, without restriction $N$ has a zero $0$; and 
$\mu_0(A\cup (H\times \cX_0)) \sse N\sm \os 0$.

At this point  we to add the special symbol $\#$ to $A$. We let 
$\mu_0(\#) = 0$, $\ov \#=\#$ and $f(\#)= \#$ for all $f\in H$. So, from now on we assume that $\#\in A$.
Since we did not change $\mu_0(X)$ for any variable $X$ we are sure 
that for every \solu $\sig$ to $\cS$ and every variable $X$ we have $|\sig(X)|_\#=0$: the marker cannot appear in any \solu.

\subsection{Triangular systems
}\label{sec:trisys}

Due to the preceding subsections we henceforth make the following assumptions:
\begin{itemize}
\item $H$ is a subgroup of $\Aut(A)$.
 \item There is some $\#\in A$ with $\ov \#=\#$ and $f(\#)=\#$ for all $f\in H$.
 \item The $H$-monoid $N$ contains a zero $0$ and for all $x\in A\cup \cX_0$ we have 
$$ \mu_0(x)=0\iff x=\#.$$
\item $\mu_0(x)$ is a pair where the second component is the stabilizer $H_x$ which is represented by a word $u\in A^*$  of length at most $\log |H|$ such that $H_x=H_u$ for all $x\in (A\cup \cX_0)^*$. 
Since $H_x=H_u$ and $u$ is in $A^*$  we have for all $f\in H$:
\begin{align}\label{eq:teststab}
f\in H_x\iff f(u)=u.
\end{align}
\end{itemize}

\begin{definition}\label{def:triaeqs}
A twisted word equation $U=V$ is called \emph{triangular} if $U$ contains at most two variables and $V$ at most one variable.
\end{definition}
 Following well-known methods (see for example \cite{die98lothaire})
we enlarge the set of variables $\cX_0$ to a larger set $\cX\supset \cX_0$
 (using at most $2\Abs \cS$ more variables) such that every equation becomes triangular in a more specific form:
 every equation has the form either $Z=1$ or $U=Z$ where  $\abs{U}=2$ and in both cases $Z$ is a variable.

  It therefore is enough to show \prref{thm:central} in the case where each each equation 
$U_i=V_i$ equals $(f,x)(g,y)=(h,Z)$ where $x,y \in A \cup \cX$ and $Z\in \cX$.
Moreover, since $(f,x)(g,y)=(h,Z)$ 
is equivalent to $(\oi h f,x)(\oi h g,y)=Z$ we can restrict ourselves to the case 
that each  
$U_i=V_i$ is of the form $(f,x)(g,y)=Z$.
Due to additional variables, we work over a set of variables
$$\cX = \os{X_1, \ov{X_1} \lds X_k, \ov{X_k}, X_{k+1}, \ov{X_{k+1}}\lds X_{k'}, \ov{X_{k'} }}$$
where $k\leq k'\leq 2 \Abs{\cS}$ and the first $2k$ variables  belong to the original system.

 Hence, the starting point is 
a system of equations $(f,x)(g,y)=Z$. The number of these triangular equations
is at most $2\Abs{\cS}$, so we can ignore this blow-up. 
During the process we need a more general 
form, nontrivial triangular equations appear as $u(f,x)w(g,y)v=u'Zv'$
where $u,w,v,u',v'$ are words over constants. Whenever such an equation with $\abs u = \abs {u'} = \abs v = \abs {v'}$ appears, then necessarily $u={u'}$ and $v =v'$; otherwise the equation is ``unsolvable''. That is, in a nondeterministic implementation of our process, this branch never leads to an accepting state. In an implementation of the algorithm we would reject the branch immediately. 

Finally, it is somewhat convenient to assume 
$\abs A + \abs \cX \leq |UV|$. We may achieve this for example 
by adding some dummy equations, and then 
  $\Abs \cS\in \Oh(\abs H + \abs{UV})$.

\subsection{Fixing more notation}\label{sec:fm}

During the process we  enlarge the sets of constants and variables. 
We begin with  two disjoint infinite alphabets with 
\invol $C$ and $\OO$ and 
$\Sig = C \cup \OO$. All constants are drawn from $C$ and all 
variables are drawn from $\OO$.  We never write down all elements from $C$ or $\OO$, just certain subsets which are needed in a specific situation.   Later we will  choose $\Sig$ such that $ \abs \Sig \in \Oh(\abs H^2 {\Abs \cS}^2)$, but initially for our infinite automata $\cT$ and $\cF$ we do not impose any size restrictions.

Throughout  we use following conventions and notation.
\begin{itemize}
\item There are $2k$ \emph{distinguished letters} $\os{d_1,\ov{d_1}\lds d_k,\ov{d_k} }$ which appear in \prref{thm:central}. 
\item $\#\in A \sse B= \ov B = \set{f(b)}{b\in  B}\sse C$.
\item $B\cap \os{d_1,\ov{d_1}\lds d_k,\ov{d_k}}=\es$ unless we are at a {\em final state} (to be defined below, see \prref{sec:finalstatedefn}).
\item  $\cY=  \ov \cY= \set{f(X)}{X\in  \cX, f\in H} \sse \OO$, and $X\neq \ov X$ for all $X\in \OO$. If $H$ acts freely on $\cY$, then we write $\cY= H \times \cX$, too. We view $\cX\sse \cY$.
\item The action of $H$ and the \invol on $\Sig$  extend those on $A\cup  \cY$.
\item $\mu\colon (B\cup \cY)^*\to N$ satisfies $\mu(a) = \mu_0(a)$ for $a \in A$.
\item $a,b,c, [p], [r,s,\lam], \ldots$ refer to letters in $C$.
\item $u,v,w, \ldots$ refer to words in $C^*$.
\item $X,Y,Z, [X,p],\ldots$ refer to variables in $\OO$.
\item $x,y,z, \ldots$ refer to words in $\Sig^*$.
\end{itemize}
These conventions hold everywhere unless explicitly stated otherwise.
They also apply to primed symbols such as $B'$,  $\cX'$ etc. 
Throughout we also use the following.
\begin{remark}\label{rem:mus}
If we know $\mu(x)\in N$ for any  $x\in (B\cup \cY)^*$, then we also know, 
by the  representation of the second component in $\mu(x)$, a word $u\in A^*$ of length at most $\log |H|$ such that $H_x=H_u$. This enables  for all $f\in H$ an efficient test 
to check whether $f(x) = x$, see (\ref{eq:teststab}) above. 
Moreover, if we have $z=xy$ with $x,y,z\in (B\cup \cY)^*$ and  we have to calculate 
$\mu(z)$ as the product $\mu(x)\mu(y)$, then we need to find a word $w\in A^*$ of length
at most $\log |H|$ such that $H_z = H_x\cup H_y= H_w$. We may assume that
 $H_x$ and $H_y$ are already given as  $H_x=H_u$ and $H_y=H_v$ where 
$uv\in A^*$ of length
at most $2\log |H|$. In order to compute $w$ we run the algorithm from the proof of \prref{lem:smallSTH}.
\end{remark}

\subsection{The initial word equation $\Winit$}\label{sec:tie}
For technical reasons we encode the initial (triangular) system $\set{(U_i,V_i)}{1 \leq i \leq s}$ of \twequs in variables $\cX$ where
$\set{X_i}{1 \leq i \leq {|\cX|/2}}\sse \cX=\ov \cX$
as a single 
word. 
Let
$U= U_1\#U_2\cdots \#U_s$ 
and 
$V= V_1\#V_2\cdots \#V_s$. 

The \emph{initial equation} $\Winit\in (A \cup (H\times\cX))^*$ is defined 
 as: 
\begin{equation}\label{eq:Winit}
\Winit=  \#X_1\cdots \#X_{|\cX|/2} \# U\# \# \ov V\# \ov{X_{|\cX|/2}}\#\cdots \ov {X_1} \# .
\end{equation}
\Ip  each $X\in \cX$ appears in $\Winit$. Here:
\begin{align*}
{\cX_0}& = \os{X_1,\ov{X_1} \lds X_k, \ov{X_k}}\text{ and}\\
\cX &=\os{X_1,\ov{X_1}\lds X_k, \ov{X_k}, X_{k+1}, \ov{X_{k+1}}\lds {X_{|\cX|/2}}, \ov{X_{|\cX|/2}}}
\end{align*}
Note that $\sig(W) =\sig(\ov W)$ \IFF  $\sig(U_i) =\sig(V_i)$ for all $i$.

\subsection{Fixing the parameters $n$, $\eps$, and $\del$}\label{sec:ned}

Having defined $\Winit$ we fix the following parameters $n,\eps, \del\in \N$ by 
\begin{equation}\label{eq:n}
 n= |\Winit|, \quad  \eps= 30n, \quad \text{and }\del = \abs H \eps = 30 |H| n.
\end{equation}
By our assumptions this implies $n>|A|+|\cX|$. We have 
$\Abs \cS \in \abs{H}+\Theta(n)$. Moreover, since $n\in \Oh(\Abs \cS)$, we have $\del\in \Oh(\abs H \Abs \cS)$ and $\eps\in \Oh(\Abs \cS)$.
As a consequence: 
\begin{equation}\label{eq:Odeln}
\Oh(\del n )\sse \Oh(\abs H \Abs \cS^2).
\end{equation}
Note that the right-hand side in (\ref{eq:Odeln}) coincides with the space bound we allow to store intermediate equations according to \prref{thm:central}.

\subsection{Extended equations and their solutions}\label{sec:exeq}
The NFA over the monoid $\End(C^*)$ we will construct uses extended equations as states. The overall strategy is to remove variables from the equation until  no variables remain. During the process 
we will enter a phase called  \delper compression repeatedly (which is the analogue of  ``block compression'' for solving word equations in free groups). During each call of   \delper compression,  each variable may create temporarily two new variables which will vanish before the end of that call. So at most $3\abs H n$  variables are needed, including these additional (temporary) variables.

\begin{definition}\label{def:exe}
An \emph{extended equation} is a tuple $E= (W,B,\cX,\theta,\mu)$, where $A\sse B$  and 
$M(B,\cX,\theta,\mu)$ is an \HN with type $\theta$. Moreover, we require: 
\begin{enumerate}
\item $W\in M(B,\cX,\theta,\mu)$ which can be written as a word in the form:
 $$W=  \#x_1\#\cdots \#x_{|\cX|/2} \# u_1\cdots \#u_s\# \, \# \ov {v_s}\#\cdots \ov {v_1} \# \ov{x_{|\cX|/2}}\#\cdots \ov {x_1} \#$$ 
with 
$x_i$, $u_j$, $v_k \in (B\cup \cY)^*$ and  $\mu(x_i)\neq 0, \mu(u_j) \neq 0, \mu(v_k)\neq 0$. 
\item Given $W$ as above 
we call $u_i=v_i$ a \emph{local \equ}.
\item $|W|_\# = |\Winit|_\#$.
\item For every $X\in \cX$ there exists some $f\in H$ such that $f(X)$ appears in $W$. 
\item We say $E$   is a \emph{standard state} if  first, $\theta=\es$,
 and second, all local \equs are triangular.
 \item If $E= (W,B,\cX,\es,\mu)$ is a standard state, then 
 $\cX\sse \cX$. 
  Moreover, 
 $$\sum_{Y\in \cY} |W|_Y \leq \sum_{Y\in \cX} |\Winit|_Y.$$
 \item If $E= (W,B,\cX,\theta,\mu)$ is any state, then
 $\sum_{X\in \cX} |W|_X\leq 3n.$ (Thus, we can bound $\OO$ by $\abs\OO\leq 3|H|n$ right away.) 
 \item If variables $X,Y$ are typed with $X\neq Y$ 
 and $(Xa,aY)\in \theta$, then we  have $\theta(X) a = a\theta(Y)$ in the submonoid $M(B,\theta,\mu)$ generated by $B$.
 \end{enumerate}
\end{definition}
Recall  (\prref{sec:twngi}) that if a variable $X$  is typed, then there is a primitive word $p=\theta(X)$ such that 
$(Xp,pX)\in \theta$.

\begin{definition}\label{def:esolu}
Let $E= (W,B,\cX,\theta,\mu)$ be an extended equation. 
\begin{itemize}
\item A \emph{solution} is a $B$-\morph 
$\sig\colon M(B,\cX,\theta,\mu)\to M(B,\theta,\mu)$
such that:  
\begin{itemize}
\item 
$\sig(W)= \sig(\ov W).$
\item $\sig(X) \in p^*$, whenever $X$ is typed and $p=\theta(X)$. 
\end{itemize}
\item An  \emph{entire solution} is a pair $(\alp,\sig)$ where $\alp\colon  M(B,\theta,\mu)\to M(A,\es,\mu_0)$ is an $A$-\morph and $\sig$ is a solution.
\end{itemize}
\end{definition}

\section{Twisted conjugacy and $\del$-periodic words}\label{sec:twconj}

A key step in proving \prref{thm:central} 
is to solve 
 a particular kind of a twisted equation: conjugacy. Let $x,y,z \in A^*$. An easy exercise in combinatorics on words~\cite{har78} 
yields: 
\begin{equation}\label{eq:wconj}
zy=xz \iff \exists r,s\in A^*\,\exists e \in \N: 
 x = rs \wedge y = sr \wedge z=(rs)^e r.
\end{equation}
This fact is crucial in Makanin's classical approach \cite{mak77} to solve (untwisted) word equations. 
Here, we need a variant of (\ref{eq:wconj}) in the twisted environment. 
We say that words $x,y\in {A}^*$ are \emph{twisted conjugate} if there are $f,g,h\in H$ 
and $z\in {A}^*$ such that $zg(y) = h(x)f(z)$. 
We also say that $\abs{x} =\abs{y}$ is the \emph{offset} of the conjugacy. 
A \emph{twisted conjugacy equation} is a  (non-triangular) twisted equation of the form 
\begin{equation}\label{eq:twco}
Z\, (g,Y) = (h,X)\, (f,Z). 
\end{equation}
\begin{proposition}\label{prop:twconj}
Let $\sig$ be a solution of the twisted equation (\ref{eq:twco}) 
such that the offset $\abs{\sig(X)}$ satisfies $1 \leq \abs{\sig(X)} < \abs{\sig(Z)}$. Then there are words 
$r\in {A}^+$, $s\in {A}^*$ and $e,j\in \N$ with $0\leq j < \abs H$
such that $|rs|=\sig(X)$ and 
\begin{equation}\label{eq:twconj}
\sig(Z) = ((rs)f(rs) \cdots f^{\abs H -1}(rs))^e\, f^0(rs) \cdots f^{j-1}(rs)f^{j}(r).
\end{equation}
\end{proposition}

\begin{figure}[h!]
\begin{center}
\begin{tikzpicture}[scale=.6]

\draw  (0,4.5) -- (17,4.5) -- (17,6.5) -- (0,6.5) -- cycle;
\draw  (0,5.5) -- (17,5.5);
\draw  (15,5.5) -- (15,6.5);
\draw  (2,4.5) -- (2,5.5);
\node[align=left, above] at (7,6.5) {$Z$};
\node[align=left, below] at (9,4.5) {$Z$};
\node[align=left, above] at (16,6.5) {$Y$};
\node[align=left, below] at (1,4.5) {$X$};
\footnotesize
\node[align=left, below] at (1,5.5) {$\uparrow h$};
\node[align=left, below] at (9,5.5) {$\uparrow f$};
\node[align=left, below] at (7,6.5) {$\downarrow 1$};
\node[align=left, below] at (16,6.5) {$\downarrow g$};

\draw  (0,0) -- (17,0) -- (17,2) -- (0,2) -- cycle;
\draw  (0,1) -- (17,1);
\draw  (15,1) -- (15,2);
\draw  (2,0) -- (2,1);
\draw[red,dotted]  (15,1) -- (15,0);
\draw[red,dotted]  (16,1) -- (16,2);
\draw[red, dashed]  (16,1) -- (16,0);
\draw[red, dashed]  (14,1) -- (14,0);
\draw[red, dashed]  (14,1) -- (14,2);
\draw[red, dashed]  (2,2) -- (2,1);
\draw[red, dashed]  (12,1) -- (12,0);
\draw[red, dashed]  (4,2) -- (4,1);
\draw[red, dashed]  (10,1) -- (10,0);
\draw[red, dashed]  (6,2) -- (6,1);
\draw[red, dashed]  (8,1) -- (8,0);
\draw[red, dashed]  (8,2) -- (8,1);
\draw[red, dashed]  (6,1) -- (6,0);
\draw[red, dashed]  (10,2) -- (10,1);
\draw[red, dashed]  (4,1) -- (4,0);
\draw[red, dashed]  (12,2) -- (12,1);
\node[align=left, above] at (16,2) {$Y$};

\tiny
\node[align=left, below] at (1,1) {$h(v)$};

\footnotesize
\node[align=left, below] at (1,0) {$v$};
\node[align=left, above] at (1,2) {$u$};
\node[align=left, above] at (3,2) {$f(u)$};
\node[align=left, above] at (5,2) {$f^2(u)$};
\node[align=left, below] at (3,0) {$u$};
\node[align=left, below] at (5,0) {$f(u)$};
\node[align=left, below] at (7,0) {$f^2(u)$};

\tiny
\node[align=left, above] at (14.5,2) {$f^j(r)$};
\node[align=left, below] at (15,0) {$f^{j-1}(u)$};
\node[align=left, below] at (16.5,0) {$f^{j}(r)$};
\end{tikzpicture}

\caption{Twisted conjugacy}
\label{fig:tw}\end{center}
\end{figure}

\begin{proof}
Let $v=\sig(X)$ and $u = h(v)$. Since $1 \leq \abs{\sig(X)} < \abs{\sig(Z)}$ the word $u$ is a proper nonempty prefix of $\sig(Z)$. If  $2 \abs u \leq \abs{\sig(Z)}$, then $uf(u)$ is a prefix of $\sig(Z)$, and so on. 
Thus, $\sig(Z)$ is a prefix of a word  $uf(u)f^2(u)  \cdots f^{k}(u)$ for some $k\in\N$.
 Next, observe that
$f^{\abs H}(u) = f^{0}(u) = u$ for every word $u\in {A}^*$. Thus,
$$\sig(Z)=[uf(u)f^2(u) \cdots f^{\abs H-1}(u)]^euf(u) \cdots f^{j-1}(u)f^j(r)$$ where $0\leq j<\abs H$, $u=rs$ and the $|r|$ suffix of $Z$ is where the pattern runs out, as illustrated in \prref{fig:tw}.  We then have
$\sig(Y)=g^{-1}f^{j}(sf(r))$.
Hence, the nonempty word $u$ and the length $\abs{\sig (Z)}$ define a unique factorization $u = rs$, integers $0 \leq e$ and $0\leq j < \abs H$ such that $\sig(Z)$ has the desired form above. 
\end{proof}

A word $p$ is called \emph{primitive} if it cannot be written as
$p=r^e$ with $e \geq 2$. In particular, the empty word $1$ is not primitive. It is well known (and easy to see) that a nonempty 
word $p$ is \IFF $p^2$ cannot be written as $p^2 = xpy$ with $x\neq 1$ and $y\neq 1$.  

Let $w,p\in {A}^+$ be nonempty words. 
We say that $w$  \emph{has period} $|p|$ if $w$ is a prefix of $p^{|w|}$. 
In other words, if 
$w=a_1 \cdots a_n$ with $a\in {A}$, then $a_i=a_{i+|p|}$ for all $1\leq i\leq n-|p|$.
A word may have several periods, for example $w=aabaabaa$ has periods $3,6,7,8$. If $|p|$ is the least period of $w$, then $|p|\leq |w|$ and we can choose $p$ to be primitive such that $p\leq w$. For example, $aab\leq 
aabaabaa$ is a primitive prefix and $|aab|=3$. 

\begin{corollary}\label{cor:twconj}
Let $\eps\in \N$, $f,g,h \in H$, and $x,y,z \in A^*$ be words with $1\leq |x| \leq \eps$ and $|z| \geq \abs H\eps$. If we have $zg(y) = h(x) f(z)$, then $z$ has a period of at most $\abs{H}\eps$.

Moreover, let $z=\alpha w \beta$ be any factorization  with $|w|=|x|$.  Then every letter $b$ occurring in  $z$ satisfies $b=f(a)$ for some $f\in H$ and some  
letter $a$ occurring in  $w$. 
\end{corollary}
\begin{proof}
By \prref{prop:twconj}  we have \begin{align*}z &= ((rs)f(rs) \cdots f^{\abs H -1}(rs))^e\, f^0(rs) \cdots f^{j-1}(rs)f^{j}(r)\end{align*}
where  $|f^i(rs)|=|x|\leq \eps$ for all $i \geq 0$. Hence, $z$ has a period $$|(rs)f(rs) \cdots f^{\abs H -1}(rs)|\leq \abs H\eps.$$ 

For the second claim, if $z=\alpha w \beta$  
with $|w|= |x|=|rs|$ then $w$ is a factor of $f^i(rs)f^{i+1}(rs)$ for some $i\geq 0$. If we  write $rs=a_1\dots a_{|x|}$, then any letter $b$ in $z$ satisfies $b=f^j(a_\ell)$. 
Let $\iota\in\{i,i+1\}$ so that $f^\iota(a_\ell)$ is a letter in $w$, then 
$b=f^j(a_\ell)=f^j(f^{-\iota}(f^\iota(a_\ell)))=f^{j-\iota}(a)$ for some $j\geq 0$.
\end{proof}

\begin{definition}\label{def:delpervelos}
We say that a word $w$ is  \emph{$\del$-periodic} if
it has some period of length at most  $\del$. 
A \delper word $w$ is called  
\emph{\llong \delper} if $\abs w \geq 3\del$, and \emph{\velo \delper} if $\abs w \geq 10\del$.
\end{definition}
For example, $aabaaabaaab$ is 4-periodic but not long 4-periodic.
An important property of \delper words is the following. 
\begin{lemma}\label{lem:sydney}
Let $w$ be a
\delper word and $w= p^e r=q^f s$
such that $p,q$ are primitive $|p|\leq |q| \leq \del$, $1 \neq r\leq p$, 
$1 \neq s\leq q$, and $\abs w \geq 2\del$. Then 
$p=q$, $e=f\geq 1$, and $r = s$.  
\end{lemma}
\begin{proof}
The assertion is clear for $|p|= |q|$. Hence we may assume that 
$p$ is a proper prefix of $q$. Since $q\leq w$ we conclude $q\leq p^\del$. Since $\abs w \geq 2\del$, and $\abs p \leq \abs q \leq \del$ we see 
$pq \leq w \leq q^{\abs w}$. Thus $q$ occurs as a factor inside $qq$: we have 
$pqs=qq$ for some~$s$. Since $1 \leq \abs p < \abs q$, this contradicts the primitivity of $q$. 
\end{proof}

Let $u$ be a prefix (resp.~factor, resp.~suffix) of some nonempty word $w$. We say that $u$ is a \emph{maximal \delper
prefix (resp.~factor, resp.~suffix)} in $w$ if we cannot extend the occurrence of the factor $u$ inside $w$ by any letter to the right or left, to see a \delper word.

\section{The ambient infinite automaton $\cT$}\label{sec:nfaT}
The states of the NFA $\cAcS$ (we are aiming  for in  \prref{thm:central}) are \exes and transitions are certain labeled arcs between states which modify the extended equations. Before we construct $\cAcS$ let us define an infinite automaton $\cT$. It will contain $\cAcS$ as a finite subautomaton. We show that $\cT$ is \emph{sound}: this means in the notation of \prref{thm:central}
\begin{equation}\label{eq:soundT}
\set{(h(d_{1})\lds h(d_{k})) \in C^* \times \cdots \times C^* }{h \in L(\cT)} \sse \cSol(\cS).
\end{equation}
This implies that all \subauta of $\cT$ are sound, too. 
The set of states in $\cT$ is the set of  \exe{s} according to \prref{def:exe}, see \prref{sec:infinite}. There are two kinds of \tras:
a \emph{\subst \tra} transforms the variables; 
a \emph{compression \tra} affects the constants, but not the variables, see \prref{sec:tritra}.

If $(W,B,\cX,\theta,\mu)\arc h (W',B',\cX',\theta',\mu')$ is a \tra,
then its label $h$ is a \morph $h\colon M(B',\cX',\theta',\mu') \to M(B,\cX,\theta,\mu)$  (in the opposite direction of the arc) which is specified by a mapping $h\colon \Del \to B^*$ where $\Del\sse B'$ is some subset (possibly empty) of constants with $\Del \cap A= \es$. We assume that such a map $h$ extends to a $A\cup\cX'$-\morph $h\colon M(B',\cX',\theta',\mu') \to M(B,\cX,\theta,\mu)$ 
by leaving all letters in $(B'\cup \cX')\sm \set{f(d)}{d \in \Del \cup \ov \Del, f \in H}$ invariant. Since $h(\Del) \sse B'^*$, the restriction of $h$ also defines 
a \morph $h\colon M(B',\theta',\mu') \to M(B,\theta,\mu)$. Note that we use the same 
letter $h$ for both \morph{s}. There will be no risk of confusion. 

Since $B'\sse C$, the  \morph $h$ also induces an \Endo of $C^*$ which respects the involution assuming
$h(c) =c$ for all $c \in C\sm B'$. However, outside $B'$ neither the action of $H$ nor the value of $\mu$ is defined, so $C^*$ is not an \HN. It is simply a free monoid with \invol, and we can read the label always as an \Endo of the free monoid with \invol $C^*$.

New constants appear only by compression. If  a word $w$ is replaced a letter $c$ by specifying $h(c)= w$, then we will automatically set $\mu(c) =\mu(w)$,
$h(\ov c) =\ov w$, $h(f(c))=f(h(c))$, and hence: $f(c) = c\iff f(w)=w$ for all $f\in H$. 
By definition of $N$, the second component in  $\mu(w)$ is a word  $u\in A^*$ of length
at most $\log {\abs H}$ such that the stabilizer $H_w$ satisfies 
$H_c = H_w= H_u$. \Ip we have an efficient test whether  $f(c)=c$ for all
$f\in H$: we just check that $f(a)=a$ for all letters $a$ which appear in the word $u$.
The crucial observation is that whenever 
$$(W_s,B_s,\cX_s,\theta_s,\mu_s)\arc {h_{s+1}} \cdots \arc {h_t} 
(W_{t},B_{t},\cX_{t},\theta_{t},\mu_{t})$$
is a labeled path and $w \in B_{t}^*$ is word, then $h=h_{s+1} \cdots h_t$ can be viewed either as a \morph $h\colon M(B_{t},\theta_{t},\mu_{t}) \to M(B_{s},\theta_{s},\mu_{s})$ or as an \Endo of $C^*$. If we have $w \in B_{t}^*$, then $h$ defines a 
word  $h(w) \in B_{s}^*$ and the corresponding element 
 $h(w) \in M(B_{s},\theta_{s},\mu_{s})$. 
 By $\eps$ we denote the identity endomorphism on $C^*$. Then $\eps$ appears as the label of 
 \tra{s} $(W,B,\cX,\theta,\mu)\arc \eps (W',B',\cX',\theta',\mu')$ where
 $h\colon M(B',\theta',\mu') \to M(B,\theta,\mu)$ is a \morph with $h(a) = a$ for all $a\in B'$.
 For example, the label $\eps$ might appear when $B'\sse B$ or 
 $\theta' \sse \theta$, etc.

\subsection{States}\label{sec:infinite}
We define the states of the $\cT$ as the set of extended equations  according to \prref{def:exe}. Thus, every state $E$ is of the form $E=(W,B,\cX,\theta,\mu)$.

\subsubsection{Initial state.} 
The 
\emph{initial state} is 
$\Einit= (\Winit,A,\cX,\es,\mu_0)$. 

\subsubsection{Final states.}\label{sec:finalstatedefn}
A state 
$(W,B,\es,\es,\mu)$ is \emph{final} if 
\begin{enumerate}
\item $W= \ov W$ and uses  no variables. 
\item The word $W$ has a prefix of the form $\#d_{1}\# \cdots \# d_{k} \#$ where 
$d_i$ are the \emph{distinguished} letters mentioned in \prref{thm:central}. 
\end{enumerate}

\subsection{Transitions}\label{sec:tritra}
We denote a \tra as 
as $E\arc{h}E'$
and for both kinds, \subst{s} and \comp{s}, we put some additional length restrictions on $h$.
For example, we allow $h(c)=1$ for a letter $c$ only if $E'$ is a final state.
Thus labels on paths not ending in a final state are never length decreasing \morph{s}.
Moreover, we require that $h(c)$ is not too long. If $h$ is specified by a set $\Del'$, then we require $\sum_{c\in \Del'} |h(c)|< |W|$ where  $E=(W,B,\cX,\theta,\mu)$. These length restrictions are not used in the proof  of the soundness result \prref{prop:bf}. We need them when proving completeness for a finite subautomaton of $\cT$. 
\subsubsection{Substitution \tra{s}}\label{sec:substtras}
\begin{definition}\label{def:substr}
A \subst \tra is denoted as 
$$(W,B,\cX,\theta,\mu)\arc h (\tau(W),B',\cX',\theta',\mu').$$
We must have  $B\sse B'$ and we require that the \tra is
defined by a $B$-\morph
$\tau\colon M(B,\cX,\theta,\mu)\to M(B',\cX',\theta',\mu')$
 and a $B$-\morph 
$h\colon M(B',\cX',\theta',\mu')\to  M(B,\cX,\theta,\mu)$
such that $|h(b)|=1$ for all $b \in B'$. \Ip $h$ is length preserving.

In the case that  some variable is typed in the source node $(W,B,\cX,\theta,\mu)$, that is 
   $\theta\neq\es$, then we add the following restrictions: 
\begin{itemize}
\item $\cX' \sse \cX$. (Thus, for $\theta\neq\es$ the set of variables cannot increase.)
\item If $X\in \cX'$, then $\theta(X)=\theta'(X)$.
\Ip $\theta(X)$ is defined \IFF $\theta'(X)$ is defined.
\item If $\theta(X)$ is not defined, then $\tau(X)=X$. 
\item If $\theta(X)$ is defined, then $\tau(X) \in\theta(X)^*\cup  \theta(X)^*\, X\,\theta(X)^* $. 
\end{itemize}
\end{definition}

We say that a \subst \tra is \emph{special} if $B=B'$. This implies that 
the label $h$ is is the identity on $M(B,\theta,\mu)$; and therefore the label will be 
$h=\eps=\id{C^*}$. Later it  would be enough to only consider special \subst \tra{s}. However this would not simplify the following proof. 

\begin{lemma}\label{lem:varsub}
Let $E=(W,B,\cX,\theta,\mu)\arc h (\tau(W),B',\cX',\theta',\mu')=E'$ be  a \subst \tra.
If $\sig'$ solves $E'$ 
 then 
 $\sig=h \sig'\tau $ solves $E$. 
 In particular, if $\alp\colon M(B,\theta,\mu)\to M(A,\es,\mu_0)$ is an $A$-\morph, then $(\alp,\sig)$ and $(\alp h,\sig')$ are entire solutions with 
$\alp\sig(W)=\alp h  \sig'(W')$.
\end{lemma}

\begin{proof} Recall by \prref{def:esolu} to prove  $\sig=h \sig'\tau $ solves $E$ we must show two things: $\sig(W)= \sig(\ov W)$ and 
whenever $X$ is typed and $p=\theta(X)$ we must have  $\sig(X) \in p^*$. 

We begin by checking  that $\theta\neq \es$ implies $h\sig'\tau(X)\in \theta(X)^*$ 
for all typed variables.
Consider a typed variable $X\in \cX$. The first case is: 
$\tau(X) \in \theta(X)^*X \theta(X)^*$. Hence $X\in \cX'$. By definition, $\theta'(X)$ is defined and $\theta(X)= \theta'(X)=p\in B^*$.
 Then $\tau(X) \in p^*Xp^*$, too. Hence, $\sig'\tau(X) \in p^*$ because every \solu $\sig'$ has to satisfy $\sig'(X) \in p^*$. Since $p\in B^*$ and $h$ is a $B$-\morph we have $h(p)=p$. Therefore $h\sig'\tau(X)\in \theta(X)^*$ in the first case. 
The second case is $\tau(X) \in p^*$ where $p=\theta(X)\in B^*$. Again, we can conclude 
$h\sig'\tau(X)\in p^*$. Thus, in both cases: whenever $X\in \cX$ is typed, then 
$h\sig'\tau(X)\in \theta(X)^*$.

Since $h$, $\sig'$, and $\tau$ are $B$-\morph{s}, so is their composition 
$h\sig'\tau$. Since $\sig'$ is a \solu of $W'=\tau(W)$, we have
$\sig'(W')=\sig'(\ov{W'})$. Hence, $\sig'\tau(W)=\sig'\tau(\ov{W})$ since $\ov{\tau(W)}=\tau(\ov W)$ because 
$\tau$ respects the \invol. It follows
that $h\sig'\tau(W)=h\sig'\tau(\ov{W})$, so $\sig(W)= \sig(\ov W)$. Thus, $h\sig'\tau$ is a \solu at ${E}$. As a consequence, $(\alp,h\sig'\tau)$ and
  $(\alp h,\sig')$ are both entire \solu{s} because  $h$ is a $B$-\morph and $A\sse B$. 
\begin{center}
\begin{tikzpicture}[
	xscale=6, yscale=2			
]

\path (0,1) node (1) {$M(B',\cX',\theta',\mu')$};
\path (-1,1) node (2) {$M(B,\cX,\theta,\mu)$};
\path (0,0) node (3) {$M(B',\theta',\mu')$};
\path (-1,0) node (4) {$M(B,\theta,\mu)$};

\draw [->, >=latex] (2) -- (1) node[midway, above] {$\tau$};
\draw [->, >=latex] (1) -- (3) node[midway, right] {$\sigma'$};
\draw [->, >=latex] (2) -- (4) node[midway, right] {$\sigma$};
\draw [->, >=latex] (3) -- (4) node[midway, above] {$h$};

\end{tikzpicture}
\end{center}

\end{proof}

\subsubsection{Compression  transitions}\label{sec:cptra}
Compressions are defined only if $\cX=\cX'$. They leave the variables invariant, but we encounter both situations  $B\sse B'$ or $B'\sse B$.
However, in case that $\theta\neq \es$ the situation is more subtle than for \subst{s}, and we need again technical restrictions in order to guarantee soundness. 

\begin{definition}\label{def:comcontra}
A  \emph{\comp \tra} 
$$(W,B,\cX,\theta,\mu)\arc h (W',B',\cX,\theta',\mu')$$ 
 is defined in $\cT$ if 
$h\colon M(B',\cX,\theta',\mu')\to M(B,\cX,\theta,\mu)$ is  an $(A\cup \cX)$-\morph such that the following conditions hold. 
\begin{itemize}
 \item We have $W=h(W')$
\item $h(b')$ can be written as a word in $B^*$ for every $b'\in B'$ and $\abs{h(c)}\geq 1$ for all $c\in B'$ unless ${E}'=(W',B',\cX,\theta',\mu')$ is a final state. 
\item  $h$ is specified by a mapping $h\colon \Del'\to B^*$ with $\Del'\sse B'$ 
such that 
$$\sum_{c\in \Del'} |h(c)|< |W|.$$
\item A variable $X$ is typed using $\theta'$ \IFF it is typed using $\theta$. 
\item There is some $e\geq 1$ such that  for all typed variables we have
 $$h(\theta'(X)) = \theta(X)^e.$$
\end{itemize}
\end{definition}
Note that for a given $(A\cup \cX)$-\morph $h\colon  M(B',\cX,\theta',\mu')\to M(B,\cX,\theta,\mu)$ the conditions 
 to be a \comp \tra are effective.

\begin{lemma}\label{lem:contr}
Let ${{E}}= (h(W'),B,\cX,\theta,\mu) \arc{h} (W',B',\cX,\theta',\mu')= {{E}}'$ be a \comp and  $\sig'$ be a \solu at ${E}'$. Then there exists a  $B$-\morph $$\sig\colon M(B,\cX,\theta,\mu)\to M(B,\theta,\mu)$$ such that $h\sig'(X)= \sig(X)$ for all $X\in \cX$.
The $B$-\morph $\sig$ satisfies the following conditions. 
If $\alp\colon M(B,\theta,\mu)\to M(A,\es,\mu_0)$ 
 is an $A$-\morph, then 
 first, $(\alp,\sig)$ is an entire \solu at ${E}$ and second, 
  $(\alp h,\sig')$ is an entire \solu at ${E}'$. Moreover, $\alp\sig(W) = \alp h\sig'(W').$
\end{lemma}

\begin{proof}
Define $\sig(X) = h\sig'(X)$ for all variables and $\sig(b) = b$ for all $b\in B$. 
This defines a $B$-\morph $\sig\colon M(B,\cX,\es,\mu)\to M(B,\theta,\mu)$ since 
$M(B,\cX,\es,\mu)$ is a free monoid. (There is no type yet on the left.) Let us  show first that
$(x,y)\in \theta$ implies $\sig(x)=\sig(y)$ in the monoid $M(B,\cX,\theta,\mu)$.
That is, $\sig$ induces a $B$-\morph (which we also denote  by $\sig$) 
$\sig\colon M(B,\cX,\theta,\mu)\to M(B,\theta,\mu)$.

For $(x,y)\in \theta$ with $x,y\in B^*$ the assertion 
$\sig(x)= \sig(y)$ is trivial because $\sig$ leaves $B^*$ invariant. 
Thus, it is enough to consider a defining relation of the form 
$(Xp,\, pX)\in \theta$ where $p=\theta(X)\in B^*$. Because of \prref{def:comcontra} we  know that  $X$ is typed on the right hand side $ M(B',\cX,\theta',\mu')$, too. Let 
$q=\theta'(X)\in B'^*$.
Thus, $h(q)= p^{e}$ for some $e\geq 1$ according to the last condition in \prref{def:comcontra}. 
Since $\sig'(X)=q^{\ell}$ for some  $\ell\geq 0$,  
we conclude $h\sig'(X) = p^{e\ell}$. Hence, whenever $(Xq,\, qX)\in \theta'$, then 
$h\sig'(Xp) = p^{1+e\ell}= h\sig'(pX)$ in $M(B,\theta,\mu)$  
since
$(Xp,\, pX)\in \theta$.

So far we have shown that $\sig$ is a well-defined \morph such that $\sig(X)= h \sig'(X)$. This implies $\sig(h(X))= h \sig'(X)$ for all variables. 
 For a constant $b\in B$ we have 
 $\sig(h(b))= h(b) = h(\sig'(b)).$
 Hence $\sig h= h \sig'$ and this means that the diagram in \prref{fig:hsigsigh} commutes.
 \begin{figure}[h!]
 \begin{center}
\begin{tikzpicture}[
	xscale=6, yscale=2			
]

\path (0,1) node (1) {$M(B',\cX,\theta',\mu')$};
\path (-1,1) node (2) {$M(B,\cX,\theta,\mu)$};
\path (0,0) node (3) {$M(B',\theta',\mu')$};
\path (-1,0) node (4) {$M(B,\theta,\mu)$};

\draw [->, >=latex] (1) -- (2) node[midway, above] {$h$};
\draw [->, >=latex] (1) -- (3) node[midway, right] {$\sigma'$};
\draw [->, >=latex] (2) -- (4) node[midway, right] {$\sigma$};
\draw [->, >=latex] (3) -- (4) node[midway, above] {$h$};
\end{tikzpicture}
\caption{$h\sig=\sig' h$.}
\label{fig:hsigsigh}\end{center}
\end{figure}
The morphism $h\colon M(B',\theta',\mu') \to M(B,\theta,\mu)$ in \prref{fig:hsigsigh} denotes the restriction of the morphism $h\colon  M(B',\cX',\theta',\mu') \to M(B,\cX,\theta,\mu)$, too.
Let $W=h(W')$ and hence, $\ov W=h(\ov {W'})$.
In order to see that $(\alp,\sig)$ is an entire \solu at ${E}$ we use $\sig h= h \sig'$ and we content ourselves to consider the following line of equations:
$$\sig(W) = \sig h(W')= h \sig'(W') = h \sig'(\ov{W'})= \sig h(\ov{W'}) = 
\sig(\ov{W}).$$
In particular, $\alp\sig(W) = \alp\sig h(W') =\alp h\sig'(W')$.  
It is also clear that  $(\alp h,\sig')$ is an entire \solu at ${E}'$
since $h$ leaves $A$ invariant.
\end{proof}

\begin{proposition}\label{prop:bf}
Let $E_{0}\arc{h_1} 
\cdots \arc{h_{t}} E_t$
be a path in $\cT$ of length $t$, where $E_0= (\Winit, A, \cX,\es ,\mu_0)$ 
is  an initial and $E_t=(W,B,\es,\es,\mu)$ is a final {state}. 
Then $E_{0}$ has an entire solution $(\id{A^*},\sig)$ with
 $\sig(\Winit)= h_1 \cdots h_t(W)$. 
\Ip for $X\in \cX$ we have 
$
\sig(X) = h_1 \cdots h_t(d_{X})
$; and $\cT$ is sound in the sense of (\ref{eq:soundT}).
\end{proposition}

\begin{proof}
Since $E_t$ is final, it has a unique solution $\sig_t= \id{B^*}$. By the lemmas above, we obtain a solution $\sig$ at $E_0$ such that  $\id{A^*}\sig(\Winit) = \id{A^*}h_1 \cdots h_t\id{B^*}(W)$. Hence, $(\id{A^*},\sig)$ is an entire  solution as desired.
\end{proof}

\section{The intermediate automaton  $\cF$}\label{sec:cF}
\prref{prop:bf} states that the large automaton $\cT$ is sound. This property cannot be destroyed by removing states or transitions. That is, every \subauto of 
$\cT$ is sound, too. 
We define a subautomaton $\cF$ of $\cT$  as follows. All \exes $\EE$ are states of $\cF$, so  the state set is the same infinite set as for $\cT$.
However, for \tra{s} we are more restrictive. To define \tras, let us first  define a \emph{weight} for equations and states. The definition is tailored that
all \comp \tras and certain \subst \tras reduce the weight of the state. 
\begin{definition}\label{def:weightexe}
Let  $E = (W,B,\cX,\theta,\mu)$ be an \exe where (as usual) $W\in (A\cup \cY)^*$ is represented as a word.  The
\emph{weight} of the equation  $\Abs{W}$ is defined by 
\begin{equation}\label{eq:weightW}
\Abs{W} =
\begin{cases}
              \abs W + {30}\del\sum_{Y\in \cY}|W|_Y & \text{ if $E$ is not final,}\\
         0 & \text{ otherwise, that is: $E$ is final.}   
        \end{cases}
\end{equation}
 The
\emph{weight} of the state
$\Abs{E}$ is a pair of natural numbers
 $\Abs{E}= (\Abs W,\abs B)$. 
\end{definition}

For $\ell\in \N$ we  order tuples in $\N^\ell$ lexicographically. For example 
$(0,42) < (1,0)$, but $(1,0,42) > (0,10,100)$; and we use the fact that there are no infinite descending chains in $\N^\ell$.   
Consider any \tra $\EE  \arc h \EE'$ in $\cT$.
Then we always have $\abs\EE< \abs{\EE'}$ unless the transition is a \subst \tra 
where at least one variable that appears in $W$ pops out a constant. 

\begin{remark}\label{rem:weightexe}
The definition of $\Abs{W}$ is invariant under the word representation of $W$. This follows because $\sum_{Y\in \cY}|x|_Y= \sum_{Y\in \cY}|y|_Y$ for all $(x,y)\in \theta$.
Second, the advantage to use the weight $\Abs{W}$ in $\Abs{E}$ (instead of 
using the more straightforward choice  of $(\abs W,\abs B)$ for $\Abs{E}$) is that following a \subst \tra, which does nothing but replace variables $X$ by $\sig(X)$ for $\abs{\sig(X)}\leq {30}\del$, leads to a state of smaller weight. 
\end{remark}

A \tra $(W,B,\cX,\theta,\mu)  \arc h  (W',B',\cX',\theta',\mu')$ in $\cT$ belongs to 
$\cF$ \IFF the following properties are satisfied. 
\begin{itemize}
\item If $(W,B,\cX,\theta,\mu)  \arc h  (W',B',\cX',\theta',\mu')$ is a \subst \tra, then 
$W'= \tau(W)$, $B=B'$, and $h=\eps$. 
\item If $E= (W,B,\cX,\theta,\mu)  \arc h  (W',B',\cX',\theta',\mu')=E'$ is a \comp \tra, then $W=h(W')$ and  $\Abs {E'} < \Abs {E}$. 
\end{itemize}

The focus for the remaining part of the proof is  on completeness. 
A \subauto $\cA$ of $\cT$ is called \emph{complete} if it holds:
\begin{equation}\label{eq:compA}
\cSol(\cS) \sse \set{(h(d_1)\lds h(d_k)) \in C^* \times \cdots \times C^* }{h \in L(\cA)}.
\end{equation}
Since every \subauto of $\cT$ is also sound in the sense (\ref{eq:soundT}) we see that every complete \subauto  is sound and complete. 

\begin{proposition}\label{prop:cAsounder}
Let $\cA$ be a trim, finite \subauto of $\cF$.
If $\cA\neq \es$, then $\cS$ has at least one solution. If $\cA$ contains a directed cycle, then $\cS$ has infinitely many solutions. 
Moreover, if $\cA$ is complete, 
then 
the converse of both assertions is true.
\end{proposition}

\begin{proof}
If $\cA\neq \es$, then $\cS$ has at least one solution by \prref{prop:bf}.
Now assume that $\cA$ contains a directed cycle. By hypothesis $\cA$ is trim. Hence, there is an accepting path with a directed cycle and this cycle doesn't involve any final state as final states are without outgoing arcs. Let $E_{s}\arc{h_s} 
\cdots \arc{h_{t}} E_t = E_s$ be this cycle. Without restriction we have 
$t > s$  and 
$\Abs {E_{s}} = \Abs {E_{s+1}}$ because $\N^\ell$ admits no infinite strictly descending chains. This means $E_{s}\arc{\eps} 
E_{s+1}$ must be a \subst \tra which is defined by some 
$\tau$ with $\abs{\tau(X)}_a\geq 1$ for some $X$ where $X$ appears in the equation belonging to ${E}_s$ and $a$ is a constant. Hence, on some accepting path we can pop out an arbitrary number of letters of $X$. Since on paths 
from an initial state to $E_{s}$ the labels are non-erasing \Endos, 
we see that we can make $\sig(X)\in A^*$ at the initial state $E_\init$ larger and larger. Thus, there are infinitely many solutions. 
The converse, under the assumption that (\ref{eq:compA}) holds, is trivial. 
\end{proof}

\section{Towards completeness}\label{sec:compcA}
During the completeness proof we always work with a state $E=(W,B,\cX,\theta,\mu)$ and a given \esolu $(\alp,\sig)$. Starting at a triple $({E},\alp,\sig)$   where 
${E}$ is a standard state, we describe a deterministic process which yields a path $(h_1, \ldots, h_t)$ inside the (infinite) automaton $\cF$ from ${E}$ to some final state $E_t=(W_t,B_t,\es,\es, \mu_t)$ so that $\alp\sig(W)= h_1 \cdots h_t(W_t)$.
Thus, $\cF$ is complete.  
The crucial property is that we are able to control the lengths of all intermediate equations $W_i$ for $1\leq s \leq t$ by 
\begin{equation}\label{eq:intequs}
\abs{W_s} \leq \abs{W_s} +\Oh(\del n).
\end{equation}
We make sure that whenever we see  an intermediate state $E_s$ 
where $\theta_s\neq \es$, then $\theta_s$ has a special structure. Moreover, when we follow a \comp \tra then we make sure the soundness condition holds for the  corresponding label according \prref{def:comcontra}.

We then can deduce \prref{thm:central} because for defining the complete 
NFA $\cAcS$ mentioned in the  theorem it is enough to consider the starting point  $$E=E_\init = (\Winit,A,\cX,\es,\mu_0).$$
Since $\abs\Winit \leq n$  we can ensure $\cAcS$ is finite by allowing \exes in $\cAcS$ only if the corresponding equation satisfies a concrete length bound in $\Oh(\del n)= \Oh(\abs H  n^2)$.
Moreover, we impose that $\cAcS$ is trim. We will come back later to these issues. For the moment we work in the infinite automaton $\cF$ and there is no length bound for the equation $W$.

\subsection{Dummy variables denoting the empty word}\label{sec:dum}
In the following it is convenient to have the following notation at our disposal. 
We introduce purely formal symbols of the form $(f,D)$ where $f\in H$ and $D$ is called \emph{dummy variable}, but the symbol  $(f,D)$ is just another explicit notation for the empty word $1$. The dummy variable
 $D$ is never listed in $\cY$. Its only purpose is that we have a unified notation for local equations (and avoid case distinctions). Since $(f,D)=1$ every \morph maps $(f,D)$ to $1$.  
The advantage is that with the help of a dummy variable, 
we may, whenever convenient, assume that every local equation has the form 
$$u(f,X)w(g,Y)v = u Z v.
$$
Here $X,Y,Z$ are  (perhaps dummy) variables and $u,v,w$ are words over constants. 

\subsection{The weight of  \esolu and the \fopro}\label{sec:wses}
We need a termination condition for the following compression procedure.
Therefore we define 
a \emph{weight} $\Abs{E,\alp,\sig}\in \N^3$ for the triple $(E,\alp,\sig)$ where 
$E= (W,B,\cX,\theta,\mu)$ is a state with  an entire
solution  $(\alp,\sig)$ by
 \begin{displaymath}
        \Abs{E,\alp,\sig} =  
{\begin{cases}
              \left(\sum_{X\in \cX}|\alp\sig(X)|, \Abs W, \abs B\right) & \text{ if $E$ is not final,}\\
         (0,0,0) & \text{ otherwise.}   
        \end{cases}}
    \end{displaymath}

At non-final states the weights 
$\Abs{W} = \abs W + {30}\del\sum_{Y\in \cY}|W|_Y$ and $\Abs{E} = (\Abs{E},\abs{B})$ were defined in \prref{eq:weightW}. Thus, actually for all states, we can write $\Abs{E,\alp,\sig}$ as a pair
  \begin{displaymath}
\Abs{E,\alp,\sig}= (\sum_{X\in \cX}|\alp\sig(X)|,\Abs{E}). 
\end{displaymath}
Being at $(E,\alp,\sig)$ we say that a \tra ${E}\arc h E'= (W',B',\cX',\theta',\mu')$ satisfies the \emph{forward property} if $E'$ has an \esolu $(\alp h,\sig') $ such that first $\Abs{E',\alp h,\sig'}<\Abs{E,\alp,\sig}$ and second, 
$\alp h \sig'(W') = \alp\sig(W).$

Following a \tra ${E}\arc h E'$ which satisfies the \fopro means that we switch from 
$(E,\alp,\sig)$ to $(E',\alp h,\sig') = (E',\alp',\sig')$. Typically after each such step we rename the tuple 
$(W',B',\cX',\theta',\mu',\alp',\sig')$ as $(W,B,\cX,\theta,\mu,\alp,\sig)$.
Since using a \tra satisfying the \fopro reduces the weight $\Abs{E,\alp,\sig}$, there are no infinite paths of \tra{s} where all \tra{s} on the path satisfy the \fopro. 

\subsection{Meta rules}\label{sec:meta}
Let $E=(W,B,\cX,\theta,\mu)$ be a state with an \esolu $(\alp,\sig)$.
We apply the following meta rules whenever possible. 

\subsubsection{Remove variables with short \solu{s}}\label{sec:meta1}
If $\abs{\sig(Y)}\leq {30}\del$ for some variable $Y\in \cY$ such that $|W|_Y\geq 1$, then follow a \subst \tra ${E}\arc \eps {E}'$ which is defined 
by a $B$-\morph $\tau$ such that $\tau(Z)= \sig(Z)$ if $\abs{\sig(Z)}\leq {30} \del$ and  $\tau(Z)= Z$ otherwise. The state ${E}'= (\tau(W),B,\cX',\theta',\mu')$
uses the same set of constants, but we have $\cX'\ssneq \cX$. We also have 
$\theta'\sse \theta$ and $\mu'$ is the restriction of $\mu$. (Thus, according to our convention we can also write ${E}'= (\tau(W),B,\cX',\theta,\mu)$.)
Let 
$\sig'$ be the restriction of $\sig$, then $(\alp,\sig')$ is an \esolu at ${E}'$ and 
we have  $\alp \eps \sig(W) =\alp\sig'\tau(W)$. Moreover, 
$\Abs{W'}< \Abs{W}$. Hence, $\Abs{{E}'}< \Abs{{E}}$ and the \tra reduces the weight of the state. 

As a consequence,  whenever we are at a state $E=(W,B,\cX,\theta,\mu)$ with an \esolu $(\alp,\sig)$, then we assume $\abs{\sig(Y)}> {30} \del$ for all $Y\in \cY$ where $|W|_Y\geq 1$.

\subsubsection{Remove useless constants}\label{sec:meta2}
We say that a letter $a\in B\sm A$ is \emph{useless} (\wrt $\sig$)  if 
$|\sig(W)|_{f\cdot a}=0$ for all $f\in H$. Note that a letter $a\in A$ is never useless. If $B$ contains a useless letter $a$, then define 
$B'=B\sm \set{f\cdot a}{f\in H}$. 
The inclusion of $B'$ into $B$ defines canonical embeddings
$M(B',\cX,\theta,\mu)\to M(B,\cX,\theta,\mu)$ and $ M(B',\theta,\mu)\to M(B,\theta,\mu)$ such that $W\in M(B',\cX,\theta,\mu)$ and $\sig(W)\in M(B',\theta,\mu)$.
The state ${E}'= (W,B',\cX,\theta,\mu)$ has an \esolu $(\alp,\sig')$ where $\sig'$ is the restriction of $\sig$. Moreover, $\Abs{{E}'}< \Abs{{E}}$. 
Hence, we can follow the  \comp \tra $E\arc \eps (W,B',\cX,\theta,\mu)$ which satisfies the \fopro.

As a consequence,  whenever we are at a state $E=(W,B,\cX,\theta,\mu)$ with an \esolu $(\alp,\sig)$, then we assume that $B$ doesn't contain any useless letters.

\begin{remark}\label{rem:useless}
We may have that $B\setminus A$  contains letters that are not {$H$-visible}, but a solution $\sig$ uses them.  Removing useless letters does not remove such letters. 
\end{remark}

\subsubsection{Moving to a final state}\label{sec:movefin}
Let $\EE=(W,B,\es,\es,\mu)$ be a standard state without any variables and with an \esolu $(\alp, \sig)$. Then  $\sig=\id{B}$ is the identity on $M(B,\es,\es,\mu)= M(B,\es,\mu)$ and we have $W= \ov W$. 
If $\EE$ is final, there is nothing to do. Hence, we assume that $\EE$ is not final. 
Since $W=\ov W$, \prref{def:exe} tells us  
 $$W=  \#x_1\#\cdots \#x_k \# u\# \# \ov u\# \ov{x_k}\#\cdots \# \ov {x_1} \#.$$
 Hence we can enlarge $B$ to a set $B'$ which contains all distinguished 
 $d_i$ for $1\leq i \leq k$. (By our convention none of the $d_i$ belongs to $B$ because the state $E$ is not final.)
 We define a $B$-\morph 
 $h\colon M(B',\es,\mu)\to M(B,\es,\mu)$ by letting $h(d_i)=x_i$. 
 Moreover, we let 
 $$W'=  \#d_1\#\cdots \#d_k \# u\# \# \ov u\# \ov{d_k}\#\cdots \# \ov {d_1} \#.$$ 
 We have no variables and $W'= \ov {W'}$. Hence $\EE'=(W',B',\es,\es,\mu')$ is final.
 The \esolu at $E'$ is  $(\alp h, \id{B})$ and we have $\alp h (W')= \alp(W)$ since none of the $d_i$ belong to $B$. Since $\Abs{E,\alp,\id{B}} >(0,0,0)= \Abs{E',\alp h,\id{B}}$  the \comp \tra 
$E\arc hE'$ satisfies the \fopro. Hence, we are done. 

As a consequence,  whenever we are at a standard state $E=(W,B,\cX,\es,\mu)$ 
with an \esolu $(\alp,\sig)$, then we assume that $\cX \neq \es$.
Moreover according to the other meta rules we have
$|\sig(X)| \geq {30} \del$ for all $X\in \cX$ and every constant $b\in B$ is $H$-visible in $W$.


\section{Compression round: the first phase}\label{sec:compr}
We perform the compression in rounds. 
Each round has two phases. The first phase is called \emph{\delper compression}, the second one is called \emph{pair compression}. During \delper compression we perform all meta rules whenever possible. 
 Recall how meta rules decrease the weight $(\Abs W, |B|)$  at states: removing a variable makes $\Abs{W}$ smaller, 
removing useless letters doesn't change $\Abs{W}$, but it makes $B$ smaller.  Moving to a final state decreases the weight of the state down to $(0,0)$ 
(which was the exceptional weight at final states) .
 None of these rules increases the
sum $\sum_{X\in \cX}|\alp\sig(X)|$. Therefore all  meta rules satisfy the \fopro according to \prref{sec:wses}.

\subsection{A simple, but useful, estimation}\label{sec:est}
During the rounds the length oscillates but it can be bounded by some 
function in $\Oh(\del n)$. In order to obtain such a bound we will later apply the following fact twice with different parameters. 

\begin{lemma}\label{lem:enno}
Let  $0\leq q<1$ and $c\geq 1$ for some real constants $q,c$, and let 
 $s\colon \N \to \N$ be a function with $s(0) \leq  \frac{c}{1-q} \del n$ and  which satisfies  a bound 
$$s(t+1) \leq q\, s(t) + c \del n$$ for all $t\in\N$.
Then 
$s(t) \leq \frac{c}{1-q}\cdot \del n$ for all $t\in \N$. 
\end{lemma}
\begin{proof} 
The statement is true for $t=0$. Assuming it is true for $t\geq 0$ then 
\begin{equation}\label{eq:sbuuf}
s(t+1) \leq q s(t) + c \del n \leq  q \frac{c}{1-q}\del n + c \del n=\frac{c}{1-q}\cdot \del n.
\end{equation}
\end{proof} In  $\Oh$-notation (\ref{eq:sbuuf}) reads as: 
 if  $s(0) \geq 0$ and $s(t+1) \in q s(t) + \Oh( \del n)$, then $s(t) \in  \Oh( \del n)$
for all $t$.

\subsection{Alphabet reduction at standard states}\label{sec:alpred}
During our procedures we introduce more and more letters, so the set $B$ grows, and  removing useless letters is not enough to  keep  the size of 
$B$ in $\Oh(\abs H \cdot \abs W)$. 

The following procedure which we call {\em alphabet reduction} is not a meta rule (which we may apply whenever possible).
If we call the procedure we explicitly say so. When we call it we wish that $B\sm A$ contains only  $H$-visible letters in $W$. 

We begin at a standard state $E=(W,B,\cX,\es,\mu)$ with an \esolu $(\alp,\sig)$ where there is some letter $b\in B$ which is not $H$-visible. Hence $|W|_{f\cdot b}=0$ 
for all $f\in H$. 
Removing useless letters is a meta rule. Hence, we may assume without restriction that all letters are useful and therefore we may assume $|\sig(X)|_b\geq 1$ for some variable. (That is, we are in the situation of \prref{rem:useless}.)
Define  $$B'=A\cup \set{a\in B}{\exists f\in H: |W|_{f(a)} \geq 1}.$$ 
Then we have $W\in M(B',\cX,\es,\mu)$. The procedure will takes us 
(via a \comp \tra defined by the inclusion $B'\sse B$) to the state 
$E'=(W,B',\cX,\es,\mu)$. Since $b\in B\sm B'$ we have $|B'| <|B|$ and therefore $\Abs{E'} < \Abs E$, too.

It is here where the notion of entire solution becomes important. We have  $\alp\colon M(B,\es,\mu)\to M(A,\es,\mu_0)$, so we can define a $B'$-\morph $\bet\colon  M(B,\es,\mu)\to M(A,\es,\mu_0)$ by 
$\bet(b) = \alp (b)$ for $b\in B\sm B'$. Since $M(B,\es,\mu)= B^*$ is a free monoid, we don't have to worry to check defining relations. 
Moreover, $\sig' = \bet \sig$ is solution at $E'= (W,B',\cX,\es,\mu')$. 
Thus, we can switch from $(E,\alp,\sig)$ to $(E,\alp,\sig') =(E,\alp,\bet \sig)$ via the \comp \tra
$(W,B,\cX,\es,\mu)\arc \eps  (W, B',\cX,\es,\mu)$. 
Since $\alp$ is an $A$-\morph we obtain $\alp = \alp \bet$. 
Hence,  $\alp\sig(W)=  \alp \bet\sig(W) = \alp \eps\sig'(W)$ as desired. 

As a consequence,  whenever we perform an  alphabet reduction, then we arrive at a standard state $E=(W,B,\cX,\es,\mu)$ with an \esolu $(\alp,\sig)$ such that every letter in $B\sm A$ is $H$-visible in $W$.
This means that after alphabet reduction the size of $B$ is at most
$\abs H \cdot (\abs A +\abs W)$.

\subsection{Mapping the positions from $\sig(W)$ to $W$}\label{sec:mappos}
Let $E= (W,B,\cX,\es,\mu)$ be a state with an empty type $\theta$ and let  
$\sig\colon M(B,\cX,\es,\mu)\to M(B,\es,\mu)$ be any $B$-\morph. Recall that $\smallset{1\lds m}$ (resp.~$\smallset{1\lds \ell}$) denotes the set of positions of $\sig(W)$ (resp.~$W$). Then $\sig$ induces a  mapping $\pi_\sig$ from $\smallset{1\lds m}$ to $\os{1\lds \ell}$ as follows. 
We define  $\pi_\sig$ from left-to-right.
We let $\pi_\sig(1)=1$. The first position in $\sig(W)$ is labeled with $\#$ and so is the first position in $W$.  No other position than $1$ is mapped to $1$. 
We shall keep the  invariant 
that $\sig (W[1,m'])=W[1,\ell']$ if $m'$ is the largest position which is mapped to 
$\ell'$. \Ip  we have $\sig( W[m'+1,m])=W[\ell'+1,\ell]$.

Now assume $\pi_\sig(i)$ is already defined for all $1\leq i \leq m'$ and $m'\leq m$. If $m'=m$ we are done. Otherwise we have $m'<m$ and we consider 
$\pi_\sig(m')=\ell'$. By the  invariant we know $\ell'<\ell$. We look at the label
of the position $\ell'+1$. It is labeled by a letter in $W$ and there are two cases. In the first case the label is a constant $b\in B$. In this case we let $\pi_\sig(m'+1)=\ell'+1$. In the second case the label is of the form $Y$ with 
$Y\in \cY$. In that case we map all positions in the interval 
$[m'+1, m' + |\sig(Y)|\,]$ to the single position $\ell'+1$.

Note that $\pi_\sig\colon \smallset{1\lds m}\to\smallset{1\lds \ell}$ enjoys the following properties. If $\ell$ is a position of $W$ which is labeled by a constant $b\in B$, then $\oi{\pi_\sig}(\ell)$ is a single position in $\sig(W)$ which is labeled by $b$, too.  If $\ell$ is a position of $W$ which is labeled by a variable $Y\in \cY$, then 
$\oi{\pi_\sig}(\ell)$ is a interval of length $|\sig(Y)|$  in $\sig(W)$. The label of that interval is just $\sig(Y)$.

\begin{definition}\label{def:vispos}
We say that a position $m'$ of $\sig(W)$ is \emph{visible (in $W$)} if 
$\pi_\sig(m')$ is a constant. Otherwise it is called \emph{invisible}.
An interval $[i,j]$ of positions of $\sig(W)$ is \emph{visible (in $W$)} (resp.~\emph{invisible}) if all positions in that interval are visible  (resp.~invisible)
positions.  If $[i,j]$ contains an invisible position, but $\abs{\pi_\sig[i,j]} \geq 2$, then we say that the interval $[i,j]$ is \emph{crossing}. 
\end{definition}

\subsection{The start of a compression round}\label{sec:compsr}
Each compression round starts at a standard state 
$E_r= (W_r,B_r,\cX_r,\es,\mu_r)$ with  an entire
solution  $(\alp_r,\sig_r)$. 
We may assume that no meta rule is applicable. 
The very first step is now  an alphabet reduction. 
For simplicity, we denote 
the state again by $E_r= (W_r,B_r,\cX_r,\es,\mu_r)$ and we have
$|B_r|\leq |H|\cdot |W_r|$.

\subsection{\delper  compression}\label{sec:delpercomp}
For convenience we rename the tuple 
$$(E_{r},W_{r},B_{r},\cX_{r},\mu_{r},\alp_{r},\sig_{r})= (E,W,B,\cX,\mu,\alp,\sig).$$ At this point we  know that no meta rule applies to $E$ and that $\abs B\leq \abs H \cdot \abs W$. 
 
Let us consider all \velo maximal \delper factors $w$ 
of $\sig(W)$ which have a maximal occurrence with at least one visible position.  (By maximal occurrence we mean that $w$ is not a factor of a longer \delper word at that occurrence.)  We assume that at least one such occurrence exists, otherwise we skip the main body of the \delper  compression and proceed directly to the end: \prref{sec:enddelcomp}.

We write $w=up^e rv$
with $\abs{u}= \abs{v} = 3\del$, $p$ is primitive of length at most $\del$ and $r$ is a nonempty prefix of $p$. (Recall  very long means $|w|\geq 10\del$ so $|p^er|\geq 4\del$.)
By \prref{lem:sydney}, we can encode the factor $up^e rv$ uniquely by writing the triple $(p,r,e)$.
Let us call $u$ and $v$ the \emph{borders} of the \velo maximal \delper factor
$up^e rv$. Consider  different occurrences $up^e rv$ and $u'p'^{e'} r'v'$ of 
\velo maximal \delper factors in $\sig(W)$. If the occurrences overlap, then this overlap takes place in the borders only, because otherwise the occurrence of the factor was not maximal.

It follows that the number of 
occurrences \velo maximal \delper factors with at least one visible position is less than $\abs W$. Thus we find some minimal index set $\Lam$ of size 
$\abs \Lam< \abs W$ such that
$$F_\Lam= \set{u_\lam p_\lam^{e_\lam} \, r_\lam v_\lam}{\lam\in \Lam}$$
is exactly the set of \velo maximal \delper factors of $\sig(W)$ which have a maximal occurrence with at least one visible position.

The idea is that at the end we arrive at a state with a solution where all these occurrences are replaced by $u[r,s,\lam]v$ where  $[r,s,\lam]$ is the notation for a fresh letter such that $p=rs$, $r\neq 1$, and $\lam\in \Lam$. We also {\em color} certain positions in $W$ and $\sig(W)$. At the end of the process a position will be {\em  green} \IFF it is labeled by some new letter $[r,s,\lam]$.

Note that $\lam$ is just a formal symbol: we need at most $\Oh(|W|)$ bits to encode it. 
We also define a set of primitive words
$$P_\Lam= \set{p_\lam}{u_\lam p_\lam^{e_\lam} \, r_\lam v_\lam\in F_\Lam}.$$
We have 
$$1\leq |P_\Lam| \leq |F_\Lam|  = |\Lam|.$$

Next we consider fresh variables which are denoted
as $[X,f(sr)]$ where $X\in \cX\sse \cX$, $f \in H$ and  for certain $p_\lam\in P_\Lam$ and then for all $rs=p_\lam$. These new variables will later be typed. We define the action of $H$ by $g\cdot [X,p]= [X,g(p)]$
and the \invol by $\ov{[X,p]}=[\ov X,\ov{p}]$. The idea is $\sig([X,p])\in p^*$; and thus,  $\sig(\ov{[X,p]})\in {\ov p}^*$ and $\sig((f,[X,p]))\in f(p)^*$. Note that $(f,[X,p]) = (g,[X,p])$ \IFF $\oi g f(p) = p$ and hence $\oi g f\in H_p$ is in the known stabilizer of $p$. 

The following routine introduces these new variables using \subst transitions. Recall that defining $\tau(X) = w$ 
substitutes $(f,X)$ by $f(w)$ and simultaneously  $(f,\ov X)$ by $f(\ov w) = \ov{f(w)}$ for all $f\in H$. 

\bigskip
\noindent\textsc{\textbf{begin procedure}} (insert new variables)\\ 
Initialize a set of fresh variables by $\Xnew=\es$ and put $E=(W,B,\cX\cup \Xnew,\es,\mu)$.\\
\textsc{\textbf{forall}} $X\in \cX$  \textsc{\textbf{do}} \\
 (Note this means we do the process once for $X$ and once for $\ov X$.)
\begin{enumerate}
\item Apply all meta rules whenever possible; \ip  $\abs{\sig(Y)}\geq {30} \del$ for all variables. 
\item  Let $q^d q'$ be the longest suffix 
of $\sig(X)$ such that $q$ is primitive, $\abs q \leq \del$, and $q'$ is a prefix of $q$. 
If $\abs{q^d q'} \leq 3 \del$, then do nothing. 
\item 
If $\abs{q^d q'} > 3 \del$, then define words $p$, $p'$, and $e\geq 0$ by 
$q^d q'= u p^e p'$ with $\abs u = 3 \del$, $\abs p =\abs q$, and $1\neq p'\leq p$.
(Note that $p$ is primitive: we have $p=q_2q_1$ for some factorization $q=q_1q_2$.)
We enlarge $\Xnew$ by a fresh variables $[X,sr]$  for all factorizations $p=rs$.
Moreover, if we enlarge $\Xnew$ by some $[X,p]$, then we also 
include $[X,f(p)]$ and $[X,f(\ov p)]$ for all $f\in H$. 

We can write $\sig(X) = xup^e p'$ with $\abs {xu} \geq 3\del$. 
Follow a \subst \tra $E\arc \eps E'= (\tau(W),B,\cX\cup \Xnew,\es,\mu')$ which is defined by $\tau(X) = X[X,p]p'$ and define an \esolu  at $E'$ by
$(\alp,\sig')$ where $\sig'(X)=xu$ and $\sig'[X,p]=p^e$. The \tra satisfies the 
\fopro.
(Due to the meta rules it can  happen that $\cX$  becomes smaller and/or that $\Xnew$ is not enlarged at all.) \item Rename $E',\tau(W),\mu',\sig'$ as
$E, W,\mu,\sig$. 
\end{enumerate}
\textsc{\textbf{endforall}}\\
\noindent\textsc{\textbf{endprocedure}}

The number of new variables $[X,p]\in \Xnew$ is bounded by $\Oh(\abs H \del{n})$ because $\abs \cX\leq n$. The factor $\del$ comes in because we  consider 
all cyclic permutations  $sr$ of $p=rs$. The factor $\abs H$ comes in because we
close under $H$-action.   On the other hand we don't need to list $[X,f(p)]$ in $\Xnew$ for $1\neq f \in H$ if $[X,p]$ is already listed. Thus, a list of $\Oh(\del{n})$
new variables suffices to specify the full set $\Xnew$ (which is closed under the action of $H$ and involution. Note that $\ov{[X,p]} = [\ov X,\ov p]\neq [X,p]$. So we keep the invariant that variables are not self-involuting. 

From now on until the end of \delper compression we only remove variables. So,
if $\cY'= H\cdot \cX'$ is the full set of variables we meet during the whole procedure, then we have:
\begin{equation}\label{eq:cYprimes}
\abs{\cX'} \in \Oh(\del n).
\end{equation}

Before the procedure we had $\cY= H\times \cX$ and $\cX\sse \cX$, and 
$\sum_{Y\in \cY} |W|_Y\leq n$ as required for a standard state. 
The corresponding set after the procedure is $\cY_{\text{new}}= \Xnew \cup H\times \cX$.
We  have to check that $\sum_{Y\in \cY_{\text{new}}} |W|_Y\leq 3n$ (because otherwise it is not an extended equation as per \prref{def:exe}). This bound is immediate. 

 In the case $\Xnew = \es$ we are still at a standard state and 
 $\sum_{Y\in \cY_{\text{new}}} |W|_Y\leq n.$ 
  For $\Xnew\neq \es$ we are  not at a standard state because $\cY_{\text{new}}$
 is not contained in  $H\times \cX$.

Since for $(W,B,\cX\cup \Xnew,\es,\mu)$ the type $\theta$ is empty,  we map the positions of $\sig(W)$ to the positions $W$ as explained above in \prref{sec:mappos}. Consider any occurrence  of a \velo  \delper factor $w= up^ep'v$ in $\sig(W)$  which is maximal and where at least one position is visible and where $|u|=|v|= 3 \del$.  Consider 
the occurrence of all \velo maximal \delper factors $w'$ where
$w'\in \set{g\cdot w, g\cdot \ov w}{g\in H}$.
Each  factor $w'$ can be written as $w'= u'w''v'$ where 
 $|u'|=|v'|= 3 \del$.

For all maximal occurrences of these factors $w'$ let us color the inner positions belonging to $w''$ {green}. Then only green positions are mapped to a variable $[X,q]\in \Xnew$. It is also clear that we we can write 
$q=sr$ for some factorization $p=rs$. Let us transport the green color to the corresponding positions in $W$. Then for all positions in $W$ which are labeled by a variable it holds that the position is green \IFF it is new variable. Note that green positions in $\sig(W)$ are separated by words of length at least $3 \del$.

In the next procedure we will  introduce a type $\theta$ which consists of defining relations 
of the form $[X,aq]a=a[X,qa]$, but it will be enough to apply such a rule where both positions in $W$ are green. Hence, the color of the positions will not be altered under this restriction. In order to define $\theta$ we use  $\Xnew\neq \es$, otherwise we skip the next procedure.

\bigskip
\noindent\textsc{\textbf{begin procedure}} (introduce a type $\theta$)
\begin{enumerate}
\item Define the type $\theta$ by 
\begin{equation*}
\begin{aligned}
\set{([X,as]a,\, a[X,sa])}{a\in B \wedge [X,as]\in \Xnew} \cup \set{([X,p]p,\, p[X,p])}{[X,p]\in \Xnew}.
\end{aligned}
\end{equation*}
Note that $[X,p]p= p[X,p]$ is actually a consequence of the other relations in $\theta$. We include it order to satisfy the definition of type in \prref{sec:twngi}, that if a variable (in this case $[X,p]$) appears in a type then there is a unique primitive word $p$ with which it commutes.
\item Choose any $[X,p]\in \Xnew$ and write $\sig([X,p]) = p^e$. (Note that we have $e\geq 10$ in this case, since a meta rule would remove the variable if it has a solution shorter than $30\del$).
Define a \morph 
$\tau\colon  M(B,\cX\cup \Xnew,\es,\mu)\to M(B,\cX\cup \Xnew,\theta,\mu)$ by
$\tau[X,p] = [X,p]p^5$. The \morph is well-defined since  $[X,q]q = q[X,q]$ in $M(B,\cX\cup \Xnew,\theta,\mu)$ for all $[X,q]\in \Xnew$. 
\item Follow the corresponding \subst \tra
$$(W,B,\cX\cup \Xnew,\es,\mu)\arc \eps (\tau(W),B,\cX\cup \Xnew,\theta,\mu').$$
The \tra satisfies the \fopro with the \esolu  
$(\alp,\sig')$ where $\sig'[X,p] = p^{e-5}$. Apply all meta rules. (After that we may have $\theta=\es$ again.) 
\item Rename $(\tau(W),\mu',\sig')$ as
$(W,\mu,\sig)$. \end{enumerate}
\noindent\textsc{\textbf{endprocedure}}

Using the relations from $\theta$ we can move the $\Xnew$ variables around over green positions. Thus, we can choose a word representation for $W\in M(B,\cX\cup \Xnew,\theta,\mu)$ 
as $W\in (B\cup \cY\cup \Xnew)^*$ such that every maximal green interval  $[i,j]$ of positions in $W$ is labeled by a word of the form
\begin{align}\label{eq:Wij}
W[i,j] &= [X_1,rs]\cdots [X_g,rs](rs)\cdots (rs)(rsr)\,[X_{g+1},sr]\cdots [X_d,sr]\\
         &=  [X_1,rs]\cdots [X_d,rs](rs)\cdots (rs)(rsr)\\
         &= (rsr) (sr)\cdots (sr)[X_1,sr]\cdots [X_d,sr].
\end{align}
In the following we simply say that $W[i,j]$ is a \emph{maximal green factor}
when we actually refer to the label $W[i,j]$ of a maximal green interval $[i,j]$ of positions. 
We can choose $0\leq d \leq 2$ because in a standard state all local equations are triangular, but this is not essential.  Without restriction we have $r\neq 1$ and that $(rs)\cdots (rs)(rsr) =(rs)^er$ satisfies $e\geq 1$. This is clear if $d=0$. For $d\geq 1$ it is enough to substitute
$[X_1,rs]$ by $[X_1,rs]rs$. 
For each $1\leq i \leq d$ there is $e_i\in \N$ such that  $\sig[X_i,rs] = (rs)^{e_i}$ 
and therefore: 
\begin{align}\label{eq:sigWij}
\sig(W[i,j]) &= (rs)^{e +e_1 + \cdots +e_d}\,r = r(sr)^{e +e_1 + \cdots +e_d}.\\
\sig(\ov{W[i,j]}) &= (\ov r \, \ov s)^{e +e_1 + \cdots +e_d}\,\ov r = \ov r(\ov s\, \ov r)^{e +e_1 + \cdots +e_d}.
\end{align}

At this point we will change $\theta$ to  $\theta'$ (defined below) since the defining relations $[X,aq]a=a[X,qa]$ with $a\in B$ will be of no use anymore. 
The idea is to replace the factors $rs$, $sr$, and $rsr$ by fresh letters denoted by $[rs]$, $[sr]$, and $[r,s,\lam]$. The $\lam$ is used to encode the sum $e +e_1 + \cdots +e_d$.
We will make the assumption $r\neq 1$,  then $[rs]$, $[sr]$, and $[r,s,\lam]$ are three different letters for $s\neq 1$, and there is no letter $[1,p_\lam,\lam]$, only $[p_\lam,1,\lam]$.
For the maximal green factor $W[i,j]$ we intend to define a word $W'[i,j]$ and a type $\theta'$ such that
\begin{align}\label{eq:Wij1}
W'[i,j] &= [X_1,rs]\cdots [X_g,rs]\,[rs]\cdots [rs]\, [r,s,\lam]\,[X_{g+1},sr]\cdots [X_d,sr]\\ \label{eq:Wij2}
&= [rs]\cdots [rs]\,[X_1,rs]\cdots [X_g,rs]\, [r,s,\lam]\,[X_{g+1},sr]\cdots [X_d,sr]
\\
&= [X_1,rs]\cdots [X_g,rs]\, [r,s,\lam]\,[X_{g+1},sr]\cdots [X_d,sr]\,[sr]\cdots [sr].
\label{eq:Wij3}
        \end{align}
More precisely, for each $u_\lam p_\lam^{e_\lam} \, r_\lam v_\lam\in F_\lam$ we associate a new letter $[r_\lam,s_\lam,\lam]$ with $\mu'([r_\lam,s_\lam,\lam])= \mu(r_\lam,s_\lam r_\lam)$,
and $[q]$ for every typed variable $[X,q]$ with  $\mu'[q] = \mu(q)$. Recall our notation that $u_\lam p_\lam^{e_\lam} \, r_\lam v_\lam$ is a \velo \delper word, 
$|u_\lam| = |v_\lam| =3 \del$, $p$ is primitive, and $r\neq 1$. 
It is important that $[r,s,\lam]$ is visible, whenever at least one the green positions is visible. This why in the different word representations (\ref{eq:Wij1})--(\ref{eq:Wij3}) for the same $W'[i,j]$ 
the  $[r,s,\lam]$  always sits between the variables. 

By introducing (if necessary) more fresh 
letters we close the set of fresh letters under \invol and $H$-action. We let:
\begin{align}\label{eq:Bnes}
\ov{[r_\lam,s_\lam,\lam]} =[\ov{r_\lam},\ov{s_\lam},\lam]&\;\text{ and }\; \ov{[q]} = [\ov q].\\
g\cdot {[r_\lam,s_\lam,\lam]} =[g\cdot{r_\lam},g\cdot{s_\lam},\lam]&\;\text{ and }\; g\cdot{[q]} = [g\cdot q] \text{ for $g\in H$.}\label{eq:Bness}
\end{align}
The set of these new letters is denoted by $\Bnew$. The number of new letters can be bounded by:
 $$\abs \Bnew \in \Oh( \abs H \cdot (|W| + \del n)).$$
Let $B'= B\cup \Bnew$. Next, we define the new type $\theta'$. 
For each typed variable $[X,q]$ there is exactly one commutation rule: 
$[X,q]q = q [X,q]$. The other defining relations say that $[r_\lam s_\lam]$ and $[s_\lam r_\lam]$ are ``conjugate'' due to the letter $[r_\lam,s_\lam,\lam]$.
Making this formal we specify $\theta'$ by:
\begin{align}\label{eq:thetaprime}
\theta'=&\set{([X,q][q],\,[q] [X,q])}{ [X,q]\in \Xnew}\\&\cup\set{[r,s,\lam] [s r],\, [r s] ([r,s,\lam]}{[r,s,\lam]\in \Bnew}.
\end{align}
Note that the defining relations in $\theta'$ are designed that (\ref{eq:Wij1})--(\ref{eq:Wij3}) hold. 

We can now define the rest of the \delper compression procedure. 
It is the analogue to Je\.z's ``block compression'' as described in \cite{DiekertJP16}. 
During the process sets of positions for $W$ and $\sig(W)$ change, but our process makes clear that we can always transport the green color: no change
involves an interval which has both colored and uncolored positions.  
We perform the following steps.

\bigskip
\noindent\textsc{\textbf{begin procedure}} (remove \velo \delper factors with a visible position)
\begin{enumerate}
\item Define the element $W'\in M(B',\cX\cup \Xnew, \theta',\mu')$  by replacing maximal green factors in $W$ just we as have done for $W[i,j]$ in \prref{eq:Wij} in order to produce $W'[i,j]$ in \prref{eq:Wij1}. Doing this everywhere defines $W'$ in a word representation as  $W'\in (B'\cup\cY\cup \Xnew)^*$.
\item Define a $B\cup\cX\cup\Xnew$-\morph 
$$h_1\colon M(B',\cX\cup \Xnew, \theta',\mu') \to M(B,\cX\cup \Xnew, \theta,\mu)$$ 
by $h[r,s,\lam]= rsr$ and $h[q]= q.$ We have $W= h_1(W')$ and 
we obtain a \comp \tra satisfying the \fopro
$$ E= (W,B,\cX\cup \Xnew, \theta,\mu) \arc {h_1}(W',B',\cX\cup \Xnew, \theta',\mu') = E'.$$
Note that $\Abs{E'} <\Abs{E}$ since $\Abs{W'} <\Abs{W}$ due to the fact that
(by our assumption) at least one green interval exists with a visible position exists; and therefore some new letter $[r,s,\lam]$ is visible in $W'$ (which represents the word $rsr$ of length at least $2$). 
The new \esolu at $E'$ is $(\alp',\sig') = (\alp h_1,\sig')$ where $\sig'(X) = X$ for 
$X\in \cX$ and $\sig'[X,p] = [p]^e$ if $\sig([X,p]) = p^e$. 
We apply the meta rules and then we rename $E',W',\alp',\sig'$ as $E,W,\alp',\sig'$ but we keep the notation for $B'$, $\cX$, $\Xnew$, $\theta'$, and $\mu'$ (although $B'$, $\Xnew$, and $\theta'$ may become smaller by the meta rules and $\mu'$ changes). 
\item
\noindent\textsc{\textbf{while}} there is letter $[p]\in \Bnew$ \textsc{\textbf{do}}
\begin{enumerate}
\item If $\Xnew \neq \es$, then choose some $[X,p]\in \Xnew$. 
Use a \subst \tra defined by  $\tau[X,p]= [X,p] [p]^2$ to make sure that 
 $\sig([X,p])$ is shorter than at the beginning of the loop and that we don't run out of
 letters $[p]$ as long as there are typed variables.  The invariant is that as long as 
 $\Xnew \neq \es$ there is some letter $[p]$ visible. 
\item  Use \tra{s} of the form $[X,p] \mapsto  [X,p] [p]$ in order to keep the invariant that $\sig([X,p]) = [p]^e$ where $e$ is even. 
Moreover, due the meta rules we maintain $|\sig([X,p])| \geq {30}\del $. At some point 
$|\sig([X,p])|$  might be too short, then we remove $[X,p]$ from $\Xnew$. 
We also maintain the invariant that $|\sig([X,p])| = |\sig([X,q])| =  |\sig([\ov X,\ov q])|$ for all $p,q$ and $X\in \cX$. Thus, if we remove one $[X,p]$, then all other typed variables using the symbol $X$ are removed simultaneously and $\theta'$ becomes smaller, too. 
\item If there is a maximal green factor
\begin{align}\label{eq:move2prime}
W[i,j]= [X_1,p]\cdots [X_g,p][p]^e[r,s,\lam]\,[X_{g+1},q]\cdots [X_d,q]
\end{align}
where $d\geq 0$ and $e$ is odd, then  
define an endo\morph 
$$h_\lam\colon M(B',\cX\cup \Xnew, \theta',\mu')\to M(B',\cX\cup \Xnew, \theta',\mu')$$
by $h_\lam([r,s,\lam])= [p][r,s,\lam]$. 
Thus, we can write 
\begin{align}\label{eq:move3prime1}W'[i,j] = h_\lam([X_1,p\cdots [X_m,p][p]^{e-1}[r,s,\lam]\,[X_{m+1},q]\cdots [X_\ell,q]).
\end{align}
This defines a new equation $W'$ and 
a $B'$-\morph $\sig'$ such that $h(W')=W$ and $\sig h_\lam(W')=\sig'(W)$.
Hence, there is a \comp \tra satisfying the \fopro
$$ E= (W,B,\cX\cup \Xnew, \theta,\mu) \arc {h_\lam}(W',B',\cX\cup \Xnew, \theta',\mu') = E'.$$
As above, $\Abs{E'} <\Abs{E}$ since $\Abs{W'} <\Abs{W}$. The new \esolu at $E'$ is $(\alp',\sig') = (\alp h_2,\sig')$.
We apply the meta rules and then we rename $E',W',\alp',\sig'$ as $E,W,\alp',\sig'$ .

\item Due to the previous steps: whenever we see a maximal green factor
$$[X_1,p]\cdots [X_g,p][p]^e[r,s,\lam]\,[X_{g+1},q]\cdots [X_d,q],$$ 
then $\sig[X_i,p] \in ([p][p])^*$ for $1\leq i \leq d$ and $e$ is even. 
Define an endo\morph 
$$h_3\colon M(B',\cX\cup \Xnew, \theta',\mu')\to M(B',\cX\cup \Xnew, \theta',\mu')$$
by $h_3([p])= [p]^2$ for all $[p]$ which appear in $\Bnew$.
Thus, we can write 
\begin{align}\label{eq:move3prime2}W'[i,j] = h_3([X_1,p]\cdots [X_m,p][p]^{e/2}[r,s,\lam]\,[X_{m+1},q]\cdots [X_\ell,q]).
\end{align}
This defines a new equation $W'$ and 
a $B'$-\morph $\sig'$ such that $h(W')=W$ and $\sig h(W')=\sig'(W)$.
Hence, there is a \comp \tra satisfying the \fopro
$$ E= (W,B,\cX\cup \Xnew, \theta,\mu) \arc {h}(W',B',\cX\cup \Xnew, \theta',\mu') = E'.$$
As above, $\Abs{E'} <\Abs{E}$ since $\Abs{W'} <\Abs{W}$. The new \esolu at $E'$ is $(\alp',\sig') = (\alp h,\sig')$.
We apply the meta rules and then we rename $E',W',\alp',\sig'$ as $E,W,\alp',\sig'$ .
\end{enumerate}
\noindent\textsc{\textbf{endwhile}}
\end{enumerate}
\noindent\textsc{\textbf{endprocedure}}

It is clear that the procedure terminates in some standard state. Let us denote that state and its \esolu as: 
\begin{equation}\label{eq:Ws1}
E_s= (W_s,B_s,\cX_s,\es,\mu_s) \text{\quad and \quad} (\alp_s,\sig_s).
\end{equation}
We began the routine at $E_r$. During the procedure 
we see symbols $[X,p], [rs]$ and $[r,s,\lam]$, and the length of the equation $W$ grows as we pop out letters in the suffix and prefix of each variable.
At the end all the new variables disappeared, either by the meta rules or
when maximal green factors are compressed into a single letter $[r,s,\lam]$. 
The only new letters in $B_s$ are of the form  $[r_\lam,s_\lam,\lam]$ and there are not more than
$|H|\cdot |W_{r}|$ of them. 

The following proposition summarizes all the changes that happen in the procedure.

\begin{proposition}\label{prop:delper}
Let $E_r= (W_r,B_r,\cX_r,\es,\mu_r)$ be the state where we started 
\delper-compression with an \esolu $(\alp_r,\sig_r)$. Let 
$E_s= (W_s,B_s,\cX_s,\es,\mu_s)$ denote
the standard state with the  \esolu $(\alp_s,\sig_s)$ where we finish \delper compression
and let $(W,B',\cX',\theta,\mu)$ be any state which we have seen (with the full set of variables $\cY'= H\cdot \cX'$) on the path from $E_r$ to $E_s$ during the procedure.

 Then we have the following.
 \begin{enumerate}
 \item 
 $B_s= B\cup \set{[r_\lam,s_\lam,\lam]}{\lam \in \Lam}$ for some $B\sse B_r$.
 \item 
$|B_s|\leq |H|\cdot (|W_{r}|+|W_{s}|)$. 
\item 
$ |B'| \in |H|\cdot (|W_{r}|+ \Oh(\del n))$.

\item
$ \Abs{W_{s}}\leq \Abs{W_{r}} +14\del n \leq \Abs{W_{r}} +20\del n$.
\item
$  \Abs{W} \in \Abs{W_r} + \Oh(\del n)\sse \abs{W_r} + \Oh(\del n)$.

\item 
$\cX_s\sse \cX_r \sse \cX$. 
\item 
$ \abs{\cX'}\in \Oh(\del n)\label{eq:iincss}$.

\item 
For each $X\in \cX_s$ the word $\sig_s(X)$ does not start or end with a \velo 
 \delper word. 

 \end{enumerate}
\end{proposition}

\begin{proof}
We justify each item as follows.
\begin{enumerate}\item Meta rules may remove some (useless) letters from the initial alphabet $B_r$, so we have $B\subseteq B_r$, and the only new constants that survive at the end  are of the form $ \set{[r_\lam,s_\lam,\lam]}{\lam \in \Lam}$.
\item $B_s$ consists of letters in $H\cdot B_r$ and those from $H\cdot \set{[r_\lam,s_\lam,\lam]}{\lam \in \Lam}$. Since we applied alphabet reduction at the beginning, $|B_r|\leq H\cdot |W_r|$, and since the letters $[r_\lam,s_\lam,\lam]$ cannot be eliminated by any compression during the procedure, their number is bounded by $H\cdot |W_s|$. The only other new constants added during the procedure are of the form $[p]$ but these are all eliminated by compression, so do not appear in $B_s$ by the meta rule.
\item Again we have $|B_r|\leq H\cdot |W_r|$. During the procedure we add letters $[rs]$ and $[r,s,\lam]$. For each variable $X\in \cX$  we add some $[rs]$ and $[r,s,\lam]$, then we need to multiply by $\delta$ since we have all cyclic permutations of these and $|rs|\leq \del$. Since the are at most $n$ variables this gives $\Oh(\del n)$ new constants, so applying the action we get the bound of $H\cdot \Oh(\del n)$ new constants.

\item We first pop out $\tau(X)=X[X,p]p'$ because we had $\sig(X)=xq^dq'=xup^ep'$ with $|q^dq'|\geq 3\del$. (If $|q^dq'|<3\del$ we do nothing.)
After applying $\tau$ we have $\sig(X)=xu$ and  $\sig([X,p])=p^e$.
If it is the case that $|p^ep'|<7\del$ then later we do not apply any compression to $[X,p]$ since it is not part of a very long factor, instead we simply pop it out. This contributes at most $14\del n$ in the worst case that this happens in the suffix of every $X\in \cX$.

If $|p^ep'|\geq 7\del$ then together with $u$ this gives a factor of length at least $10\del$, so $[X,p]$ is part of a very long \delper word so it is compressed down to a single letter. Thus $7\del$ is the most added at either end of any variable by the procedure. This gives $14\del n$. We give the larger bound $20\del n$ to simplify later calculations only.

\item Since $\Abs{W} \leq \abs{W} + 90 \del n\in \abs{W}+\Oh(\del n)$ at every state in $\cT$, it is enough to show $\Abs{W} \in \Abs{W_r} + \Oh(\del n)$.
However this only requires the estimation in \prref{sec:est}.  By that estimation we content ourselves to define a function $s\colon \N \to \N$ with $s(0) \leq c \del n$ and  which satisfies for all $t$ a bound 
$$s(t+1) \leq q\, s(t) + c \del n$$
for some  $q<1$ and $c\geq 1$.
To see where the $q$ comes from in our application, choose 
$s(0)$ to be number of letters $[p]$ at the state where they first appear. 
Each time we pass through a \tra defined by $h([p]) = [p][p]$ we half the number of these letters; and this shows that we can define $q=1/2$. Between these steps where we halve the number of $[p]$'s we create at most $c\del n$ new ones with $c\in \Oh(1)$.

\item During the procedure we add new variables $[X,p]$ but these are eliminated by the compression. Since we apply meta rules we may also remove variables $X\in \cX_r$. Thus
$\cX_s\sse \cX_r \sse \cX$.

\item 
This is justified above at  \prref{eq:cYprimes}.

\item 
Consider any $\sig_s(X)$ with $X\in \cX_s$. If for that $X$, the word $\sig(X)$ at the beginning at the procedure ``insert new variables'' had 
a \delper suffix of length more than $3\del$, then, due to the splitting of variables, 
the length of the maximal \delper suffix in $\sig_s(X)$ is  exactly $3\del$. Hence, there is no \velo \delper suffix. In the other case the suffix of length $3\del$ in $\sig(X)$ and $\sig_s(X)$ coincide. Thus, in this case the length of the maximal \delper suffix in $\sig_s(X)$ is  at most $3\del$.
Since $X\in \cX$ implies $\ov X\in \cX$ the same is true for the prefix of each variable.
\end{enumerate}

The proposition is therefore shown.
\end{proof}

\subsection{The end of the \delper compression}\label{sec:enddelcomp}
Recall that we started a compression round at a standard state $E_r$ with an equation $\abs{W_r}$; and we end standard state $E_s$ with an equation $\abs{W_s}$. (Possibly $r=s$.) If, due the meta rules, $E_s$ is final, we are done. Hence, we continue under the assumption that  $E_s$ is not a final state. 

\section{Pair compression}\label{sec:paircomp}
We enter the second phase of the compression round with ``pair compression'' directly after the end of \delper compression. We enter the pair compression procedure at  
a standard state $E_s$ which is not final and with an  \esolu $(\alp_s,\sig_s)$.
If $E_r$ and $(\alp_r,\sig_r)$ denote  the situation where we began the current compression round (where we began  \delper compression), then
\prref{prop:delper} tells us: 
\begin{equation}\label{eq:Ws2}
\Abs{W_s} \leq \Abs{W_r}+ 20 \del n. 
\end{equation}

During the pair compression all states are standard states. No type is needed. Phrased differently, we have $\theta= \es$, there are no typed variables and hence, all variables belong to $\cX$. Thus, the number of positions labeled with twisted variables is at most $n$, as it is required for standard states.

Our goal in this section  is to compress pairs $ab\leq W_s$  of constants into single letters without causing any conflict due to  overlap with other pairs or variables that are connected via twisted equations. In particular, compressing a pair linked to a \delper factor might cause problems, so we wish to avoid compressing those pairs.
 This leads us to define the following.
 \begin{definition}\label{def:pccond}
Let $E=(W,B,\cX,\es,\mu)$  be a standard state  with a \solu $\sig$.
We say that $(E,\sig)$ 
 satisfies the \emph{shrinking pair condition} if 
 there is no $X\in \cX$ such the word $\sig(X)$ starts with a \velo \delper word\footnote{Recall that a word $w$ is \velo \delper \IFF $\abs w \geq 10 \del$ and 
$w$ is a prefix of some word $p^{|w|}$ where $\abs p \leq \del$, see \prref{def:delpervelos}.}.
\end{definition}
Note this is 
the situation we find ourselves in at the conclusion of the \delper compression procedure: $E_s= (W_s,B_s,\cX_s,\es,\mu_s)$ with its \solu $\sig_s$ satisfies \prref{def:pccond}.
The shrinking pair condition is a necessary condition when proving 
\prref{lem:fareps} below. That technical lemma is one of the  key steps.
\subsection{Positions revisited}\label {sec:posrev}

Consider any standard state $E=(W,B,\cX,\es,\mu)$ together with a \esolu $(\alp,\sig)$.
 We introduce a precise notion of equivalence $\equiv$ between positions (and intervals) of $\sig(W)$, which we introduce now. The idea is that whenever we modify a \solu  $\sig$ at a position $i$, then we must modify  $\sig$ at all equivalent positions $j \equiv i$ in order to keep the property of being a \solu.  
 Moreover, $\equiv$ should be the finest equivalence relation with that property. For example, when we compress a factor $ab$ where $a$, $b$ are letters, then we want to compress only certain occurrences of $f(ab), f(\ov b \ov a)$ (where $f\in H$) and not all of them.

As explained in \prref{sec:mappos} there is a canonical mapping from the set $[1,\abs{\sig(W)}]$ 
 to the set of  positions in $W$. By $\lam(i)$ we denote the label of a position $i$ in $\sig(W)$.
Recall also the notion of duality: 
if $[l,r]$ is an interval in $[1,\abs{\sig(W)}]$, then $\ov{[\ell,r]}= [\ov r, \, \ov \ell]$ denotes the 
dual interval, where $\ov i = |\sig(W)|+1 -i$ for all $1\leq i \leq |\sig(W)|$. 
According to the definition of standard states, the set of twisted variables which appear in $W$ 
can be written as $\cY= (H\times \cX)$ and we have $\cX\sse \cX$. 
It is convenient to  fix a subinterval $I(X)$ of $[1,\abs{\sig(W)}]$ for each $X\in \cX$ such that 
$\ov{I(X)}= I(\ov X)$ as follows. Consider $\os{X,\ov X} \sse \cX$, then 
without restriction we have $X=X_i$ for some unique $1\leq i \leq |\cX|/2$ (and hence  $\ov{X}=\ov {X_i}$). Choose for $I(X)$ the left-most maximal interval  $[\ell_X, r_X]$ in $\sig(W)$ which is mapped to a position labeled by $X_i$. \Ip 
$r_X-\ell_X+1= \abs{\sig(X} $.
We let 
$\ov{I(X)}= I(\ov X)$. By the specific structure of $W$ being an \exe, we see that 
$\ov{I(X)}$ is the right-most maximal interval in $\sig(W)$ which is mapped to a position labeled by $\ov{X_i}$. 
To have a notation we let $2m=\abs{\sig(W)}$ and 
\begin{align} \label{eq:IcX}
I(\cX)&= \bigcup\set{I(X)}{X\in \cX},\\
I(B) &= \set{i\in [1,2m]}{i \text{ or $\ov i$ is mapped to a position labeled by a constant}}.
\end{align}
Note that $I(\cX)$ and $I(B)$ are disjoint sets of positions: if $i\in I(\cX)$, then there is some $X$ such that $i$ is mapped to a position labeled by $X$ and $\ov i$ is mapped to a position labeled by $\ov X$.
The next idea is to identify positions in $\sig(W)$ based on the fact that we can write $W$ in the form $W=\#U\#\, \# \ov V\#$ such that $\sig(W)=\sig(\ov W)\iff 
\sig(U)=\sig(V)$. In pictures we intend to place the positions of 
$\sig(U)$ and $\sig(V)$ on top of each other. The intuition is clear: we have 
$[1,2m]= \os{1\lds m, \, \ov m \lds \ov 1}$. The positions of 
$\# \ov V\#$ cover $\ov m \lds \ov 1$. Hence, we can think that $\#V\# = \ov {\#  \ov V\#} $ uses the same set of positions as  $\#U\#$, namely the set $\os{1\lds m}$. 
Thus, every $i\in [1,2m]$ has always two interpretations: for 
$i\leq m$ as a position either in $\#U\#$ or in $\#V\#$, for $m<i$ the situation is dual. Let us make this intuition formal. 

The  mapping from $[1,2m]$ to the positions of $W$ induces a relation 
$${\leadsto} \sse  [1,2m]\times ( I(\cX)\cup I(B)).$$ We define ${\leadsto}$ inductively.
If $i$ is mapped to a position in $W$ which labeled by a constant, then we have $i\in I(B)$ and we let $i \leadsto i$ and $\ov i  \leadsto  \ov i$.
In the other case $i$ is mapped to a position labeled by $Y=(f,X)$ for some $f\in H$ and  $X\in \cX$. Let $\ell$ the leftmost position in $[1,2m]$ which 
is mapped to the same position as $i$. Then we can write 
$i=\ell +k$ and we find $j=\ell_X +k$. In this case we let  $i \leadsto j$ 
and $\ov i \leadsto \ov  j$. Note that this implies $\lam(i) = f(\lam(j))$. 
The position $\ov  j$ belongs to $I(\ov X)$.

Up to duality, there are three cases:
\begin{enumerate}
\item $i\in I(B)$ and $\ov i\in I(B)$. Then there is only one $j$ such that $i \leadsto j$. 
\item $i\in I(B)$ and $\ov i\notin I(B)$. Then we have $i\leadsto i\in I(B)$ and 
$i\leadsto j\in I(\cX)$. Since $I(\cX)\cap I(B)=\es$, we have $i\neq j$.
\item $i\notin I(B)$ and $\ov i\notin I(B)$. Then $i\leadsto j$ and $\ov i\leadsto \ov k$ with $\os{j,\ov j, k, \ov k}\sse I(\cX)$. Hence, there 
there are  at most two  $j,k\in I(\cX)$ such that $i \leadsto j$ and $i \leadsfrom k$.
\end{enumerate}

Let us explain the meaning of this ``$\leadsto$'' relation 
by considering a local equation $u(f,X)w(g,\ov Y)v = uZv$. In $W$ this local equation corresponds to some factorization 
$$W = u_1 \#u(f,X)w(g,\ov Y)v\# u_2\;\ov{u_2}\, {\#}\, \ov{v} \ov{Z} \ov{u}\, {\#}\, \ov{u_1}.$$
Let $\ell = \abs{\sig(u_1 \#)} +1$ and $r=\abs{\sig(XwYv} -1$. Then the interval
$\sig(W)[\ell,r]$ is labeled by $u\sig(f,X)w(g,\ov Y)v$; and, 
since $\sig$ is a \solu,  we have $u\sig(f,X)w(g,\ov Y)v=u\sig(Z)v$.
For  each $i\in [\ell,r]$ we have $i\leadsto j$  where either $i=j$ or $j\in I(X)\cup I(\ov Y)$. If $j\in I(X)\cup I(\ov Y)$, then $\ov i \leadsto \ov k\in I(\ov Z)$ and therefore $ i \leadsto k\in I( Z)$, too. 
The positions of $u,w,v$ are visible in $W$, but \wrt 
$\ov W$ this true only for the positions of $u$ and $v$. 
For the positions of $u$ and $v$ the relation $\leadsto$ is the identity. 
The  relations are depicted in  \prref{fig:ontop}: the relation ${\leadsto}$ is given by
the positions in the middle row to the top and bottom row. Let $j \leadsfrom i$ denote $i\leadsto j$.  If $j \leadsfrom i \leadsto k$, then we write $j\sim k$.

Consider any $i\in  [1,2m]$ and $j,k\in I(\cX)\cup I(B)$ such that $j \leadsfrom i \leadsto k$.
(The interesting case is $j\neq k$.)  Hence, we have $j\sim k$ and in the pictures we put the positions $j$ and $k$  on top of each other. \prref{fig:ontop} gives an example, where $32 \sim 10\lds 35 \sim 13$, $98\sim 14, 99\sim 15$, and  $\ov{45} \sim 16\lds \ov{42} \sim 19$. 

\begin{figure}[h!]
\begin{center}

\begin{tikzpicture}[
	xscale=0.8,
	yscale=0.6
]

\def\linesep{0} 
\node at (.5,2.5) {$a$};
\node at (1.5,2.5) {$a$};
\node at (4,3.5) {$I(X)$};
\node at (6.5,2.5) {$a$};
\node at (7.5,2.5) {$b$};
\node at (10,3.5) {$I(\ov{Y})$};
\node at (12.5,2.5) {$b$};
\node at (13.5,2.5) {$b$};

\node at (1.5,0.5) {$a$};
\node at (0.5,0.5) {$a$};
\node at (12.5,0.5) {$b$};
\node at (13.5,0.5) {$b$};

\node at (7,-0.5) {$I(Z)$};

\draw (2,0+\linesep) rectangle (12,1-\linesep);

\foreach \i in {2,6,8,12} {
	\draw  (\i,2+\linesep) -- (\i,3-\linesep);
}

\foreach \i in {32,33,34,35} {
	\node at (\i-29.5,2.5) {\i};
}
\foreach \i in {42} {
	\node at (\i-33.5,2.5) {$\ov{45}$};}
\foreach \i in {43} {
	\node at (\i-33.5,2.5) {$\ov{44}$};}
\foreach \i in {44} {
	\node at (\i-33.5,2.5) {$\ov{43}$};}
\foreach \i in {45} {
	\node at (\i-33.5,2.5) {$\ov{42}$};}

\draw (0,1+\linesep) rectangle (14,2-\linesep);

\foreach \i in {92,...,105} {
	\node at (\i-91.5,1.5) {\i};
}

\draw (2,2+\linesep) rectangle (6,3-\linesep);
\draw (8,2+\linesep) rectangle (12,3-\linesep);

\foreach \i in {2,12} {
	\draw  (\i,0+\linesep) -- (\i,1-\linesep);
}

\foreach \i in {10,...,19} {
	\node at (\i-7.5,0.5) {\i};
}
\end{tikzpicture}
\caption{A local equation $bb(f,X)ab (g,\ov Y)aa= bbZaa$. The left side occupies positions $92\lds 105$ and dual of the right side (\ie $\ov b\,\ov b\, \ov Z\,\ov a\,\ov a$) positions $\ov{105} \lds\ov{92}$. 
Moreover, $\sig(W)[92,93]= aa$, $\sig(W)[98,99]= ab$, and $\sig(W)[104,105]=bb $. 
}
\label{fig:ontop}\end{center}
\end{figure}
Since $i \leadsto j\iff \ov i \leadsto \ov j$ we have 
\begin{equation}\label{equiv:cruc}
j\sim k \iff \ov j\sim \ov k.
\end{equation}
Moreover, $j\sim k$ implies that there are $f,g\in H$ such that
$f(\lam(j))= \lam(i) = g(\lam(k))$. Thus, $j\sim k$ ensures $\lam(j)\in H\cdot \lam(k).$

We have $I(\cX)\cup I(B)\sse  [1,2m]$ and  ${\sim} \sse  [1,2m]\times  [1,2m]$ is a symmetric relation. It is also reflexive
 on $I(\cX)\cup I(B)\times I(\cX)\cup I(B)$.

By $\approx$ we denote the reflexive and transitive closure of $\sim$. However, the relation $\approx$  is too fine, in general. Since $i\sim j$ implies $\lam(i)\in H\cdot \lam (j)$, we cannot expect that $i\approx \ov i$, because $\lam(\ov i)= \ov{\lam(i)}$
is typically not in $H\cdot \lam (j)$. Clearly, if we intend to change the label at position $i$ from, say, $a$ to $c$, then we must change the label 
 at position $\ov i$ from $\ov a$ to $\ov c$.

In the following, we write $i\darc j$ if $j= \ov i$ and we define $\equiv$ to be the equivalence relation over $[1,2m]$ which is generated by ${\sim}\cup{\darc}$. We have $\approx\sse \equiv$, but we have just seen that these relations are different, in general.  
Since $i\sim j \iff \ov{i}\sim \ov{j}$ by \prref{equiv:cruc}, we have
$$ i \darc \ov i \sim j \iff i \sim \ov j \darc j.$$
Hence,
\begin{equation}\label{equiv:equiv}
i \equiv j \iff \text{ either } i \approx j \text{ or  } i \approx \ov j.
\end{equation}

We extend the notation above to intervals. Let $i$, $j$ be positions in $\sig(W)$ such that $i\leadsto j$, and let ${p}\in \N$. Assume (by symmetry) that we have $i\leadsto j\in I(X)$ due to mapping the position $i$ of $\sig(W)$ to a position $q$ in $W$ which is labeled by a twisted variable $(f,X)$. If the position $i+p$ is also mapped to the  same  position $q$, then all positions in  the interval $[i,i+{p}]$ are mapped to $q$. 
We then write
$$[i,i+{p}]\leadsto [j,j+{p}].$$
As above we now define the relation $\sim$ on intervals. 
Again, we let $\approx$ be the generated equivalence relation, now on intervals. 
Finally, we relate an interval $[l,r]$ with $1 \leq l \leq r \leq 2m$ to the 
interval $[\ov r,\, \ov l]$ via 
$$[l,r] \darc \ov {[l,r]}= [\ov r,\, \ov l].$$
Thus, we also extend $\equiv$ to the equivalence relation on intervals which is generated by $\sim$ and $\darc$. Having this, the general form of (\ref{equiv:equiv}) becomes
\begin{equation}\label{eq:kequiv}
[i,i+{p}] \equiv [j,j+{p}] \iff  [i,i+{p}]\approx [j,j+{p}] \vee [i,i+{p}] \approx [\ov{j+{p}},\ov j].
\end{equation}
In the following we say that two positions (or intervals) are \emph{equivalent}
if they are related by $\equiv$. 

\begin{lemma}\label{lem:expan}
Let ${t}\geq 1$ 
and $\sig$ be a solution at a standard state $E=(W,B,\cX,\es,\mu)$  such that 
$|\sig(X)| \geq 2{t}$ for all $X\in \cX$. 
For each 
$X\in \cX$ write $\sig(X)=uxv$ with $|u|=|v|={t}$. Let
$E'=(W',B,\cX',\es,\mu')$ and  $\sig'$ denote state and \solu which we obtain by following a \subst \tra defined by $\tau(X) = uXv$ for all $X$. 
Let $I'(\cX')$ be the set of positions according 
to \prref{eq:IcX} and $\approx'$ the equivalence relation on intervals with respect to $\sig'$.
Let $i$ and $j$ positions in
$I'(\cX')$ 
such that   $i \approx' j$, then we have $[i-{t},i+{t}] \approx [j-{t},j+{t}]$ 
with respect to  $E$ and $\sig$.
\end{lemma}

\begin{proof} Let $X\in \cX'$. Then $X\in \cX$ and 
if $I(X) = [\ell_X,r_X]$, then the corresponding interval 
$I'(X)$ \wrt $\sig'$ is $I'(X)= [\ell_X+{t},r_X-{t}]$.
Thus, if $i \approx' j$ with $i\in I'(X)$ and $j\in I'(Y)$, then we obtain a ``domino tower'' as depicted in \prref{fig:lemexpan}.
The intervals \wrt $I'(\cX')$ are the white blocks and $i$ and $j$ are arranged such that $i$ sits along a vertical column above $j$. We obtain the corresponding tower \wrt $I(\cX)$ by adding the grey borders of length $p$ at each side  in order to switch from $I'(Z)$ to $I(Z)$. 
Thus,  $i \approx j$  \wrt  $W'$ implies $[i-{t},i+{t}] \approx [j-{t},j+{t}]$.
\end{proof}

 \begin{figure}[h!]
\begin{center}
\begin{tikzpicture}[scale=.6]

\draw[ FillBrown,fill] (11,6) -- (11,5)--(8.5,5) -- (8.5,6)-- cycle;
\draw[ FillBrown,fill] (1,6) -- (1,5)--(3.5,5) -- (3.5,6)-- cycle;
\draw[ FillBrown,fill] (5.5,5) -- (5.5,4)--(8,4) -- (8,5)-- cycle;
\draw[ FillBrown,fill] (21,5) -- (21,4)--(18.5,4) -- (18.5,5)-- cycle;
\draw[ FillBrown,fill] (19,4) -- (19,3)--(16.5,3) -- (16.5,4)-- cycle;
\draw[ FillBrown,fill] (2,4) -- (2,3)--(4.5,3) -- (4.5,4)-- cycle;
\draw[ FillBrown,fill] (1,3) -- (1,2)--(3.5,2) -- (3.5,3)-- cycle;
\draw[ FillBrown,fill] (11,3) -- (11,2)--(8.5,2) -- (8.5,3)-- cycle;
\draw[ FillBrown,fill] (4,2) -- (4,1)--(6.5,1) -- (6.5,2)-- cycle;
\draw[ FillBrown,fill] (20,2) -- (20,1)--(17.5,1) -- (17.5,2)-- cycle;

\draw  (11,5) -- (1,5) -- (1,6) -- (11,6) -- cycle;
\draw  (5.5,5) -- (21,5) -- (21,4) -- (5.5,4) -- cycle;
\draw  (19,4) -- (2,4) -- (2,3) -- (19,3) -- cycle;
\draw  (11,3) -- (1,3) -- (1,2) -- (11,2) -- cycle;
\draw  (4,2) -- (20,2) -- (20,1) -- (4,1) -- cycle;

\draw[gray] (8,6) -- (8,1);
\draw[gray] (8.5,6) -- (8.5,1);
\node[above] at (5,5.1) {$X$};
\node[below] at (15,2.0) {$Y$};

\draw [<->] (8.5,5.25) --(11,5.25);
\draw [<->] (1,5.25) --(3.5,5.25);
\draw [<->] (5.5,4.25) --(8,4.25);

\footnotesize
\node[align=left,above] at (8.25,6) {$i$};

\node[align=left,below] at (8.25,1) {$j$};

\tiny
\node[below] at (9.75,5.9) {${t}$};
\node[below] at (2.25,5.9) {${t}$};
\node[below] at (6.75,4.9) {${t}$};

\end{tikzpicture}
\caption{Example illustrating \prref{lem:expan}.}
\label{fig:lemexpan}\end{center}
\end{figure}

The \emph{distance} $d(i,j)$ between positions of $\sig(W)$ is denoted as usual: 
$d(i,j) = \abs{j-i}$. We continue with the same the notation as in \prref{lem:expan}.
We apply the lemma  with ${t}=10 \abs H \eps$. Recall that we fixed the parameters such that  
$ \eps= 30n$ and $\del = \abs H \eps$.
Hence, ${t} = 10 \del$. 
\begin{lemma}\label{lem:fareps}Let $t=10 \del$ and suppose that 
$E=(W,B,\cX,\es,\mu)$  satisfies the shrinking pair condition
(\prref{def:pccond}) \wrt $\sig$. Let 
${Z}\in \cX'$ be a variable with 
$I'({Z})= [l,r]$ (and hence, $I({Z})= [l-10 \del,\,r+10 \del]$); and let 
$l\leq i < j \leq r$ such that $i \approx' j$ with respect to $\sig'$. Then we have 
$d(i,j) > \eps$. 
\end{lemma}

\begin{proof}
According to \prref{lem:expan} we have $[i-{t},i+{t}] \approx[j-{t},j+{t}]$ \wrt $\sig$. Next, we choose ${k}\in \N$  as large as possible so that $[i-{k},i+{t}]$ is a subinterval of $I({Z})$ and $[i-{k},i+{t}] \approx [j-{k},j+{t}]$ with respect to $\sig$.
By induction on the number of steps using the relation 
$\sim$, this implies that there is some interval 
$[{\ell},{\ell}+{t}+{k}]\sse I(X')$ such that both, first
\begin{align*}
 [j-{k},j+{t}] \approx [i-{k},i+{t}] \approx [{\ell},{\ell}+{t}+{k}]
\end{align*}
and second, ${\ell}$ is the first  position in $I(X')$, see \prref{fig:fareps}. 
Assume, by contradiction, 
$d(i,j) \leq \eps$. Then, $\sig(W)[i-{k},i+{t}]$ and $\sig(W)[j-{k},j+{t}]$ are twisted conjugate with a positive offset $d(i,j)$ which is at most $\eps$. This implies that 
$\sig(W)[i-{k},i+{t}]$ is a \velo \delper word by \prref{cor:twconj}. Therefore, the  prefix 
$\sig(W)[\ell,\ell+{t}+{k}]$ of $\sig(X')$ is a \velo \delper word. This contradicts the hypothesis that $E$ satisfies the shrinking pair condition. 
\end{proof}

\begin{figure}[h!]
\begin{center}
\begin{tikzpicture}[scale=.5]

\draw[ FillBrown,fill] (4,6) -- (4,0)--(4.5,0) -- (4.5,6)-- cycle;
\draw  (11,5) -- (1,5) -- (1,6) -- (11,6) -- cycle;
\draw  (3,5) -- (15,5) -- (15,4) -- (3,4) -- cycle;
\draw  (19,4) -- (2,4) -- (2,3) -- (19,3) -- cycle;
\draw  (11,3) -- (1,3) -- (1,2) -- (11,2) -- cycle;
\draw  (4,2) -- (20,2) -- (20,1) -- (4,1) -- cycle;
\draw  (0,0) -- (10,0) -- (10,1) -- (0,1) -- cycle;

\draw[gray] (7.75,6) -- (7.75,0);
\draw[gray] (8.25,6) -- (8.25,0);

\node[above] at (3,5) {${Z}$};
\node[below] at (2,1) {${Z}$};
\node[above] at (18.4,1) {$X'$};
\node[above] at (4.25,1) {$\ell$};

\draw [->] (4.25,-.65) --(4.25,0);
\footnotesize
\node[align=left,above] at (8,5) {$i$};
\node[align=left,below] at (7,1) {$i$};
\node[align=left,below] at (8,1) {$j$};
\node[align=left,below] at (9,5.8) {$j$};

\tiny
\node[align=left,below] at (4.25,-0.5) {$j-k \approx \ell = $ first position in $I(X')$};
\end{tikzpicture}
\caption{Illustration for the proof of  \prref{lem:fareps}.}
\label{fig:fareps}
\end{center}
\end{figure}

\subsection{Red positions}\label{sec:redpos}
We use the notation of \prref {sec:posrev}.
Let $[l,r]$ be a maximal interval in $\sig(W)$ which is mapped to a position in $W$ which is labeled by some $(f,X)$ where $f\in H$ and $X\in \cX$. Thus, the factor $\sig(W)[l,r]$ is equal to the word $\sig(X)$. The
positions $l$ and $r$ at the borders of $[l,r]$ play a special role because if there is a factor $ab =\sig(W)[l-1,l]$, then we cannot compress $ab$ because 
$ab$ is ``crossing'': compressing $ab$ might be ``dangerous''. In order to signal  ``danger'' we color the first and the last position in each interval $I(X)$ \emph{red}.  Moreover, whenever $i\equiv j$ holds in $\sig(W)$, then  
color $j$ \emph{red}, too. For example, if we have a situation as depicted in \prref{fig:reddomino}, then the red color at the last position of $I(X)$ and the red color at the first position of $I(Z)$ yields two red columns. Note that the first red position in $I(X)$ and the last red position in 
in $I(\ov X)$ are equivalent: these are dual positions.
For convenience, we also color all positions in $\sig(W)$ \emph{red} which are labeled by the marker symbol $\#$. 
If the label of $i$ is  $\#$, then $i\equiv j\iff j\in \os{i,\ov i}$.

Since $i\equiv \ov i$ for all positions, it  follows that there are at most $n$ pairwise different equivalence classes of red positions. This counting will be used later.

\begin{figure}[h!]
\begin{center}
\begin{tikzpicture}[scale=.6]

\draw[ FillRed,fill] (5,6) -- (5,0)--(5.5,0) -- (5.5,6)-- cycle;
\draw[ FillRed,fill] (11,6) -- (11,0)--(10.5,0) -- (10.5,6)-- cycle;

\draw  (11,5) -- (1,5) -- (1,6) -- (11,6) -- cycle;
\draw  (5,5) -- (22,5) -- (22,4) -- (5,4) -- cycle;
\draw  (19,4) -- (3,4) -- (3,3) -- (19,3) -- cycle;
\draw  (11,3) -- (1,3) -- (1,2) -- (11,2) -- cycle;
\draw  (4,2) -- (20,2) -- (20,1) -- (4,1) -- cycle;
\draw  (2,0) -- (22,0) -- (22,1) -- (2,1) -- cycle;

\draw[gray] (7.75,6) -- (7.75,0);
\draw[gray] (8.25,6) -- (8.25,0);
\node[above] at (3,5.1) {$X$};
\node[below] at (17,1.0) {$Y$};
\node[above] at (17,4.1) {$Z$};

\draw [->] (5.25,-.25) --(5.25,0);
\draw [->] (10.75,-.25) --(10.75,0);

\footnotesize
\node[align=left,above] at (8,5.1) {$i$};

\node[align=left,below] at (8,0.9) {$j$};
\tiny
\node[align=left,below] at (5.25,0) {(red)};
\node[align=left,below] at (10.75,0) {(red)};
\end{tikzpicture}
\caption{Red positions induced by the red borders in $I(X)$ and $I(Z)$.}
\label{fig:reddomino}\end{center}
\end{figure}

Consider an interval of length two $I=[i,i+1]$ without red position. 
The idea is to compress 
 $I$ into a single position. The problem is overlapping:
we might have $[i-1,i]\equiv [i,i+1]$ or  $[i,i+1]\equiv [i+1,i+2]$. 
Note that 
$[i-1,i]\equiv [i,i+1]$ implies $i-1\equiv i$ or $i-1\equiv i+1$. Similarly, 
$[i,i+1]\equiv [i+1,i+2]$ implies $i\equiv i+1$ or $i\equiv i+2$.
Therefore, we start with intervals of length $4$ where all four positions are inequivalent: This enables us  to compress the middle interval of length $2$. 
We shall use the following lemma.

\begin{lemma}\label{lem:fgood}
Let $[i-1,i,i+1,i+2]$ be an interval of length $4$ without any red position and where the positions are pairwise inequivalent. Consider 
 $[i,i+1]\equiv [j,j+1] \equiv [k,k+1]$. Then there are two cases:
\begin{enumerate}
\item $[j,j+1] \cap [k,k+1]= \es$, 
\item 
$k=j$ and $[j,j+1] \not\approx [\ov{k+1},\ov k]$.
\end{enumerate}
\end{lemma}

\begin{proof}
Notice that  each of the intervals $[i-1,i,i+1,i+2], [j-1,j,j+1,j+2],  [k-1,k,k+1,k+2]$
is without red positions. We may assume that $[j,j+1] \cap [k,k+1]\neq \es$ because otherwise we are done. 

First, let $j=k$. By contradiction assume $[j,j+1] \approx [\ov{k+1},\ov k]$. 
Then $j+1\approx \ov k\darc k$ which implies $k=j\equiv j+1$.
Since $[i,i+1]\equiv [j,j+1]$ we obtain $i\equiv i+1$. This was excluded.

In the second case we have  $j\neq k$. Let us show that $j\neq k$ and $[j,j+1] \cap [k,k+1]\neq \es$ leads again to a contradiction. Since $j\neq k$ we cannot have $[j,j+1]= [k,k+1]$. Hence $j+1=k$ or $k+1 = j$. 
By symmetry in $j$ and $k$, we may assume $j+1=k$. 

We cannot have $[j,j+1] \approx [k,k+1]$ because then $j\equiv k$, but $k=j+1$,  and hence,
$j\equiv j+1$. This is impossible. Thus, $[j,j+1] \approx [\ov{k+1},\ov k]$
and $j \equiv k+1= j+2$. 
We remember $j \equiv j+2$.
If $[i,i+1] \approx [j,j+1]$, then (as no position is red)
$[i,i+1,i+2] \approx [j,j+1,j+2]$ implies $i \approx i+2$. This is impossible. 
Hence, the last option is $[\ov j -1,\ov j] \approx [i,i+1]$ and  $[\ov j-2,\ov j -1,\ov j] \approx [i-1,i,i+1]$. However, $j \equiv j+2$ implies $\ov j \equiv \ov j-2$. 
We have again a contradiction as $i-1 \not\equiv i+1$.
\end{proof}

 \prref{ex:doweneedthis} indicates why the assertion in \prref{lem:fgood} only holds in the middle interval $[i,i+1]$ of $[i-1,i,i+1,i+2]$, in general.  

\begin{example}\label{ex:doweneedthis}
We don't exclude that $H$ acts with involution. Thus, there might be an $a\in B$ and $f\in H$ such that $f(a) = \ov a$.
Consider the equation
$\ov X= (f,X)$ with the solution  $\sig(X) = dcba\ov b \ov c \ov d$ and where 
$f(x) = x$ for $x = b,c,d$.  Then we have 
$$\sig(\ov X) = dcb \ov a \ov b \ov c \ov d = f(\sig(X)).$$
The positions of $\sig(X)$ can be identified with $\os{1,\ldots, 7}$ 
with $i\equiv 8-i$ for all positions $1\leq i\leq 7$. Since positions $3$
and $5$ are equivalent, the interval $[2,5]$ contains equivalent positions. 
The four  
positions in the interval $[1,4]$ are pairwise inequivalent. 
However,  $[3,4]$ intersects with 
$[4,5] = [4,\ov 3]$. Thus, later on 
we cannot compress the interval $[3,4]$ corresponding to the pair 
$ba$. On the other hand, there is no obstacle to compress 
the interval $[2,3]$ which is labeled by $cb$. 
\end{example}

\begin{lemma}\label{lem:len10}
Let $\sig$ be a solution at a standard state $E=(W,B,\cX,\es,\mu)$ and 
$I=[p, p+9]$ be an interval of length $10$ in $\sig(W)$ without any red position such that
$i\approx j$ implies $i=j$ for all $i,j\in I$. Then $I$ contains a subinterval $Q$ of length 
$4$ where all positions are pairwise inequivalent.
\end{lemma}

\begin{proof}
 For simplicity of notation let $I=[1,10]$.
 If all positions in $[4,5,6,7]$ are pairwise 
inequivalent, we are done. 
\Ip there are $1\leq i < j\leq 10$ such that 
$i \equiv j$. Since, $i \not\approx j$ this implies $i \approx \ov j$
by (\ref{equiv:equiv}).
This in turn means that we cannot have $i \equiv j\equiv k$ where
$1\leq i < j<k\leq 10$ because this would lead to $i\approx k$
via $i \approx \ov j\approx k$. We say that $j$ is the partner of $i$ if 
$i\neq j$ but $i \equiv j$. We conclude that every $i\in I$ has at most one partner. 

We know $4+i\equiv 4+j$ for some $0\leq i < j\leq 3$. Hence 
$4+i\approx \ov{4+j}$. Let $1\leq k \leq 3$. Since $I$ is without red positions, this implies 
$$4+i-k \approx \ov{4+j}-k = \ov{4+j+k} \darc 4+j+k \equiv 4+i-k.$$
This means that every position 
$q\in Q=\os{4+i-3,4+i-2,4+i-1,4+i}$ has one partner 
in $P=\os{4+j+3,4+j+2,4+j+1,4+j}$. Since $Q\cap P=\es$ and $Q\cup P\sse I$, we are done: in $Q$ all four positions are pairwise inequivalent.
\end{proof}

\begin{definition}\label{def:comint}
We say that an interval $I= [i,i+1]$ in $\sig(W)$ is \emph{\cgood} if the following conditions hold. 
\begin{itemize}
\item Neither $i$ nor $i+1$ is  red. 
\item The positions $i$ and $i+1$ are visible,
\item Whenever $ [i,i+1]\equiv J\equiv K$, then either $J=K$ or $J\cap K=\es$.
\end{itemize}
\end{definition}

\begin{remark}\label{rem:nocon}
The definition of a \cgood interval $I=[i,i+1]$ excludes $i\approx i+1$. 
Indeed, since neither $i$ nor $i+1$ is  red, $i\approx i+1$ implies 
$[i-1,i,i+1]\approx [i,i+1,i+2]$, hence an overlap
$[i,i+1]\approx [i+1,i+2]$. 
However,  $I\approx [\ov{i+1},\ov i]$ is allowed. Hence, it may happen that $i\equiv i+1$. If such an $I$ is labeled with $ab$, then 
there is some $f\in H$ with $f(ab) = \ov b \ov a$. Hence, $f(a)=\ov b$ and 
$f(b) = \ov a$.  We deduce that we obtain a consistent  labeling for 
$$ I= [i,i+1] \approx [\ov{i+1},\ov i] \darc  [i,i+1].$$
Thus, we may compress $[i,i+1]$ into a single position with label $c$ and therefore, due to  $I \approx [\ov{i+1},\ov i]$,  we must also compress  $[\ov{i+1},\ov i]$ into a single position with label 
$f(c)$; and due to  $[\ov{i+1},\ov i]\darc  [i,i+1]$, we have to compress  $[\ov{i+1},\ov i]$ into a single position with label 
$\ov c$. But this is fine, we have $h(c) = ab$, $h(\ov c) = \ov b\, \ov a$, and $f(c)= \ov c$. 
\end{remark}

\subsection{The procedure}\label{sec:outpaircomp}
Recall that we have fixed  $\del = \abs H \eps$ and $\eps= 30n$.  Thus, $\del\in \Oh(\abs H n)$ and $\eps\in \Oh(n)$. We start at a standard state $E=(W,B,\cX,\es,\mu)$ together with an \esolu $(\alp,\sig)$ where none of the meta rules apply.
\Ip $\abs{\sig(X)} \geq 30 \del$ and  $\cX\sse \cX$ with $\cY= H\times \cX$. Hence, by definition: 
$$\sum_{Y\in \cY} |W|_Y \leq \sum_{Y\in \cX} |\Winit|_Y\leq n.$$

All local equations have the form 
$u(f,X)w(g,Y)v= uZv$. (As before dummy variables are allowed.) 
We define the equivalence relations $\approx$ and $\equiv$
over the set of positions of $\sig(W)$ as defined in \prref{sec:posrev}.
Let $E$ be a standard state with equation $W$ and \esolu $(\alp,\sig)$. 
Once, we found a \cgood interval $I$ in $\sig(W)$, we may call the following procedure for that interval.

\medskip
\noindent\textsc{\textbf{begin procedure}} (compress a good interval $I$)
\begin{enumerate}
\item Let $a,b\in B$  and let $ab$ be the label of  the \cgood interval $I=[i,i+1]$. 
Choose  a fresh letter $c$ with stabilizer $H_c=H_a\cap H_b$; and define a $B$-\morph from $B'= B \cup \set{f(c),f(\ov c)}{f\in H}$ to $B^2$  by $h(c) = ab$. 
Whenever $[i,i+1]\approx [j,j+1]$, then the label of $[j,j+1]$ is $f(ab)$ for some $f\in H$. Replace each of the intervals $[j,j+1]$ (resp.~$[\ov {j+1} ,\ov j]$) by a single new position and label this position with $f(c)$ 
(resp.~$f(\ov c)$). (There is no conflict in this relabeling, see \prref{rem:nocon}.) Since there is no red position
in $[j,j+1]$ and $[\ov{j-1} ,\ov j]$,  none of the intervals $[j,j+1]$ or $[\ov{j-1} ,\ov j]$ is ``crossing''. So, this gives a new 
but shorter equation $W'$. We have $h(W')=W$ and new \solu $\sig'$ such that $h\sig'(W')= \sig(W)$. 
\item Follow the corresponding \comp \tra
$$E=(W,B,\cX,\es,\mu) \arc h (W',B',\cX,\es,\mu')=E'.$$
We have a new state $E'$ with an \esolu $(\alp',\sig') = (\alp h,\sig')$. 
There is  also new numbering for the positions, but the red positions can still be identified.  
\end{enumerate}
\noindent\textsc{\textbf{endprocedure}}

\bigskip
We are now ready to define the procedure ``pair compression'' which uses 
``compress a good interval'' as a subroutine.

\bigskip
\noindent\textsc{\textbf{begin procedure}} (pair compression)
\begin{enumerate}
\item 
For every $X\in \cX$  write $\sig(X) = uxv$ with $\abs u= \abs v = 10 \del$. Follow a \subst \tra x
$$E=(W,B,\cX,\es,\mu) \arc \eps (W',B,\cX,\es,\mu')=E'$$
defined by the \subst $\tau(X) = uXv$ and $W'=\tau(W)$. This \tra satisfies the 
\fopro with the new \esolu $(\alp,\sig')$ where 
$\sig'(X) = x$ for all $X\in \cX$. Recall that $\sig(W)= \sig'(W')$.

After the preceding step, define the intervals $I'(X)$ \wrt $\sig'$  
as done in \prref {sec:posrev}. Use the \emph{red} color for the first and the last position in each $I'(X)$. Color in  red all equivalent 
positions in $\sig'(W')$ of red positions \wrt $\equiv'$, too. See \prref{sec:redpos}.
\item Rename $E',W',\mu',\sig',\approx',\equiv'$ as $E,W,\mu,\sig,\approx,\equiv$.
\item Define the alphabet $\Bold=B$. During the following loop we keep the invariant $\Bold\sse B$.
\item
\noindent\textsc{\textbf{while}} $\sig(W)$ contains a \cgood interval $I= [i,i+1]$ 
with a label in $\Bold^2$
\\ 
\noindent\textsc{\textbf{do}}
\begin{enumerate}
\item Choose any \cgood interval $I$ in $\sig(W)$.
\item Run the procedure ``compress $I$''.
\item Rename $E',B',W',\mu',\alp',\sig',\approx',\equiv'$ as $E,B,W,\mu,\alp,\sig,\approx,\equiv$ and transfer the induced coloring of red positions. 
\end{enumerate}
\noindent\textsc{\textbf{endwhile}}
\item Perform an alphabet reduction at the standard state $E$. 
\item Rename $E,B,\cX,W,\mu,\alp,\sig$ as $E_{r'},B_{r'},\cX_{r'},W_{r'},\mu_{r'},\alp_{r'},\sig_{r'}$.
\end{enumerate}
\noindent\textsc{\textbf{endprocedure}}

\begin{remark} The procedure ``pair compression'' may  not 
actually succeed in compressing any pair. Its first step always ``pops out" letters to make the equation longer (by $20
  \del$). After that if no pair is compressed, the procedure leaves the equation longer than before it was called.
This is intentional: if the equation becomes long enough, then one of  \delper- or pair compression is guaranteed to reduce the equation size by a positive fraction.
\end{remark}

\subsection{The end of pair compression ends the compression round}\label{sec:endcompsr}
We began the compression round at a standard state
$E_{r}= (W_{r},B_{r},\cX_{r},\es,\mu_{r})$ with an \esolu 
$(\alp_{r},\sig_{r})$.
We ended the \delper compression 
either by entering a final state or, in the other case, at a standard state $E_{s}= (W_{s},B_{s},\cX_{s},\es,\mu_{s})$ with an \esolu 
$(\alp_{s},\sig_{s})$ such that 
\begin{equation}\label{eq:Wrs}\Abs{E_{s},\alp_{s},\sig_{s}}\leq  \Abs{E_{r},\alp_{r},\sig_{r}} \text{ and } 
\Abs{W_{s}} \leq 
\Abs{W_{r}} + 20 \del n.
\end{equation} 
We started the pair compression  at the standard state $E_{s}= (W_{s},B_{s},\cX_{s},\es,\mu_{s})$ with the \esolu 
$(\alp_{s},\sig_{s})$. 
Compression took place only  for \cgood intervals which where labeled 
by words $ab$ with $a,b\in \Bold = B_{s}$. Each compression reduced the length of the equation because a good interval consists of two visible positions. 
Thus, at most $\abs{W_{s}}$ compressions were possible; and this shows that we did not  introduce more 
than $\abs H\cdot \abs{W_{s}}\in \abs H\cdot \abs{W_{r}} +\Oh(\del n)$ fresh  letters. Thus, every alphabet $B$ of constants we met during the entire round satisfied 
\begin{equation}\label{eq:boundB}
\abs{B}\in \abs H\cdot \abs{W_{r}} +\Oh(\del n).
\end{equation}
Now, let $E_{r'}= (W_{r'},B_{r'},\cX_{r'},\es,\mu_{r'})$ denote the standard state with the \esolu $(\alp_{r'},\sig_{r'})$ where we end the procedure ``pair compression''.
In the very first step of the procedure we followed a \subst \tra. This is enough to infer 
\begin{equation}\label{eq:progr}
\Abs{E_{r'},\alp_{r'},\sig_{r'}} < \Abs{E_{s},\alp_{s},\sig_{s}} \leq \Abs{E_{r},\alp_{r},\sig_{r}}. 
\end{equation}

\begin{proposition}\label{prop:pairshr}
Let  $E_{r}= (W_{r},B_{r},\cX_{r},\es,\mu_{r})$ be a standard state with an 
\esolu $(\alp_{r},\sig_{r})$ at the start of a compression round and $E_{r'}= (W_{r'},B_{r'},\cX_{r'},\es,\mu_{r'})$ be the standard state where we end the round with the 
\esolu $(\alp_{r'},\sig_{r'})$.
Then we have $$\Abs{W_{r'}} \leq \frac{29 \Abs{W_{r}}}{30} +\Oh(\del n).$$
Moreover, if $W$ is any equation which we see on the path
from $E_{s}$ to $E_{r'}$, then 
we have $\Abs{W} \leq \Abs{W_{r}} + \Oh(\del n)$.
\end{proposition}

\begin{proof}Each compression round has two phases. 
The \delper  compression stops at a standard state $E_{s}= (W_{s},B_{s},\cX_{s},\es,\mu_{s})$ with an 
\esolu $(\alp_{s},\sig_{s})$. By \prref{prop:delper} $\Abs{W_s} \leq \Abs{W_{r}} + 20\del n$ and all intermediate equations $W$ satisfy 
$\Abs{W} \leq \Abs{W_{r}} + \Oh(\del n)$.

Now, let $W$ be any equation being on the path
from $E_{s}$ to $E_{r'}$.  
The additional length for $W$ is due to the first step in pair compression when we substitute variable $X$ by $u_XXv_X$ with $\abs{u_X}=\abs{v_X}= 10 \del n$. This shows
$\Abs{W}\leq  \Abs{W_{r}} +40 \del n$. 
Moreover, 
$\Abs{W_{r'}}\in \abs{W_{r}} + \Oh(\del n)$ and $\abs{W_{s}}\leq \Abs{W_{s}}$.

Thus, by \prref{lem:enno} it suffices 
to prove 
\begin{equation}\label{eq:pairshr}
\abs{W_{r'}} \leq \frac{29 \abs{W_{s}}}{30} +\Oh(\del n).
\end{equation}
Let us have a closer look at a local equation $u(f,X)w(g,Y)v = uZv$ in $W_s$.
(We allow dummy variables). \Ip we can think that $W_s$ begins with a prefix
$U\#u_1Z_1v_1\#$ and ends with a suffix $\#\ov{v_1}\ov{Z_1}\ov{v_1}\#\ov U$. 
Having this, the word $W_s$ is covered by factors $\#u(f,X)w(g,Y)v\#$ and $\#\ov v \ov Z\ov u\#$.  
In the first steps of pair compression we follow \subst \tra{s} and a factor $\#u(f,X)w(g,Y)v\#$ becomes
$\#uf(u_X)(f,X)f(v_X)w g(u_Y)(g,Y)g(v_Y)v\#$ and $\#\ov v \ov Z\ov u\#$ becomes 
$\#\ov v\,  \ov{v_Z}\ov Z  \ov{u_Z}\,\ov u\#$.

Pair compression compresses all factors $uf(u_X)$, $g(v_Y)v$, $\ov v\,  \ov{v_Z}$ and  $\ov{u_Z}\,\ov u$ into single letters. This bounds the total increase by the first \subst \tra{s} by $2n$. 

We don't have such a simple bound for the factors $f(v_X)w g(u_Y)$ because the corresponding positions in $\sig(W_s)$ interact with the positions in $I(Z)$. Let $[\ell,r]$ be the interval in $[1,\, |\sig(W_s)| -1]$ corresponding to $f(v_X)w g(u_Y)$. Let us cut the interval in $[\ell,r]$ into a disjoint union of intervals, each of them having exact length $\eps= 30 n$. If a position belongs to any of these intervals of length $\eps= 30 n$, then we mark the position. Thus, at least 
$\abs w - \eps$ positions in the interval $[\ell,r]$ belonging to $f(v_X)w g(u_Y)$ are marked by these intervals. (We have no better bound since $X$ and  $Y$ might be dummy variables.) Removing if necessary at most $n$ of these intervals  we may assume that their total number is $\ell n$ with $\ell\geq 0$. The crucial observation is that 
we have 
\begin{equation}\label{eq:ellneps}
|\abs{W_{s}} - \ell \eps n|\in \Oh(\del n).
\end{equation}
Each interval of length $\eps$ is split in $3n$ intervals of length $10$. By \prref{lem:fareps} an interval of length $\eps$ can have at most $2n$ red positions. 
Thus, in each interval of length $\eps$  there are at least $n$ intervals of length $10$ without any red position. By \prref{lem:len10} each such interval $[i,i+9]$ contains  an interval of length $4$ where all positions are inequivalent. By \prref{lem:fgood} 
we can compress at least one interval in that interval of length $10$. (Note that 
\prref{lem:fgood} provides us with a compression inside  $[i+1,i+8]$. This means, the compression is guaranteed even if we compressed before an interval 
$[i-1,i]$ or $[i+9,i+10]$.) This means that the length of $\eps= 30 n$ is reduced to 
 at most $\eps= 29 n$ during compression. Hence, 
 \begin{equation}\label{eq:fracellneps}
\abs{W_{r'}} \in \frac{29}{30}\ell \eps n +\Oh(\del n).
\end{equation}
Due to (\ref{eq:ellneps}) we conclude $\abs{W_{r'}} \in \abs{W_{s}} +\Oh(\del n)$.
This shows (\ref{eq:pairshr}) and hence, the assertion of the proposition. 
\end{proof}

\begin{remark}\label{rem:constants}
\prref{prop:pairshr} tells us that there is a constant $\kappa_1\in \N$ 
such that $\Abs{W_{r'}} \leq \frac{29 \Abs{W_{r}}}{30} +\kappa_1\del n.$
We content ourselves with a generous bound by letting $\kappa_1 = 97$.
This bound suffices and it is an overestimation, as it can seen by the preceding  proof and by reversing the $\Oh$-notation into concrete constants. The 
value $\kappa_1 = 97$ was chosen such that the later constant $\kappa$ in \prref{cor:allround} is divisible by $100$.
Thus, we can conclude $W_{r'}$ and for every equation $W$ we see on the path
from $E_{r}$ to $E_{r'}$ the following upper bounds. 
\begin{enumerate}
\item If $\Abs{W_{r}} >  30 \kappa_1 \del n$,  then $\Abs{W_{r'}}< \Abs{W_{r}}$ 
 and  $\Abs{W}< \Abs{W_{r}}+ 40 \del n$.
\item If $\Abs{W_{r}} \leq  30 \kappa_1 \del n$,  then 
$\Abs{W}\leq \Abs{W_{r}}+ 40 \del n\leq 2960 \del n$. 
\end{enumerate} 
These estimations are used in next section. 
 \end{remark}

\section{Putting it all together: the overall compression method}\label{sec:cm}
Now we explain what we do if we start the first compression round at 
 the initial state $E_\init$ with a given initial entire \solu
$(\id{A^*},\sig_\init)$.
We begin a first compression round $r$ with $r=0$ and  $E_0= E_\init$ with a given initial entire \solu
$(\alp_0,\sig_0)=(\id{A^*},\sig_\init)$. We end the round after one phase each of \delper  compression and 
 pair compression with a standard state $E_1$ and an entire \solu
$(\alp_1,\sig_1)$ such that $\Abs{E_1, \alp_1,\sig_1}<\Abs{E_0, \alp_0,\sig_0}$. 
We repeat this process by starting the next round $r+1$ with   $E_{r}$ and
$(\alp_{r},\sig_{r})$ and ending the round in $E_{r'}$ and
$(\alp_{r'},\sig_{r'})$. For simplicity of notation we write $r+1=r'$. 
Thus, $$E_{r'}= E_{r+1} = (W_{r+1},B_{r+1},\cX_{r+1},\es,\mu_{r+1}) \text{ and }  (\alp_{r'},\sig_{r'}) = (\alp_{r+1},\sig_{r+1}).$$
We conclude  
\begin{equation}\label{eq:termrou}
\Abs{E_{r+1}, \alp_{r+1},\sig_{r+1}}<\Abs{E_{r}, \alp_{r},\sig_{r}}.
\end{equation}
By (\ref{eq:termrou}) the process terminates: there exists some round $t\geq 0$ and during that round  we reach a final state $\Efin$ without variables and with an \esolu $(\alpfin,\id{C})$. Hence, the entire process defines a path in $\cF$ which is labeled by
some $h_1\cdots h_t\in \End(C^*)$ such that $h_1\cdots h_t(\Wfin) =\Winit.$
We have $\abs\Winit=n$ and therefore $\Abs\Winit\leq 30 \kappa_1 \del n$.
(Note that for large $n\to \infty$ the ratio $\frac{\Abs\Winit}{30 \kappa_1 \del n}$ tends to $0$. For large $n$ the initial size $\Abs\Winit$ is much, much smaller than $30 \kappa_1 \del n$.
By \prref{rem:constants}, for all rounds $r$ with $0\leq r \leq t$  we can state:
\begin{align}\label{eq:kappa2960}
\Abs{W_{r}} \leq 30 \kappa_1 \del n +40 \del n \leq 2960 \del n.
\end{align}
We also need an estimation for the maximal weight of an equation in the middle of each round. \prref{prop:pairshr} says we have to add at most 
$40 \del n$ with respect to the starting point of a round. 
Thus, the conclusion of (\ref{eq:kappa2960}) is therefore: whenever we see an equation  $E = (W,B,\cX,\theta,\mu) $   on the path from $E_\init$ to $\Efin$ we have
\begin{align}\label{eq:kappa3000}
\Abs{W} \leq 2960 \del n +40 \del n \leq  3000 \del n.
\end{align}

\begin{corollary}\label{cor:allround}
Let 
$\kappa = 3000$ and let  $\cB$ the subautomaton of $\cF$ which is defined
defined as follows. The states of $\cB$ are the extended equations
$(W,B,\cX,\theta,\mu)$ where $$\Abs{W} \leq \kappa \del n.$$
Then $\cB$ is a finite and complete \subauto of $\cF$. Let
$\cAcS$ be the trimmed subautomaton of $\cB$, then the NFA $\cAcS$ accepts a rational set of 
$A$-morphisms $L(\cAcS) \sse \End(C^*)$ satisfying the following conditions from \prref{thm:central}
\begin{equation}\label{eq:goodAA}
\cSol(\cS) =\set{(h(d_{1})\lds h(d_{k})) \in C^* \times \cdots \times C^* }{h \in L(\cAcS)}.
\end{equation}
Moreover, $\cSol(\cS)=\es$ \IFF $L(\cAcS)= \es$;  and $\abs{\cSol(\cS)}<\infty$
\IFF $\cAcS$ doesn't contain any directed cycle. 
\end{corollary}

\begin{proof}
The automaton $\cB$ is finite because first, the number of states is finite and second, 
if $E$ is any state in $\cB$, then there are only finitely many $E\arc{h}E'$ \tra{s} in $\cF$  
where $E'\in \cB$. Thus, the out-degree is finite for every state in $\cB$. Since $\cB$ is finite, 
$\cAcS$ is finite, too. The NFAs $\cAcS$ and $\cB$ are both sound by \prref{prop:bf}. They are complete, this follows from  \prref{prop:bf}, since $\kappa = 3000$  is large enough by (\ref{eq:kappa3000}). This shows
$$\cSol(\cS) =\set{(h(d_{1})\lds h(d_{k})) \in C^* \times \cdots \times C^* }{h \in L(\cAcS)}.$$
Finally, \prref{prop:cAsounder} implies that $\cSol(\cS)=\es$ \IFF $L(\cAcS)= \es$ and that $\abs{\cSol(\cS)}<\infty$
\IFF $\cAcS$ doesn't contain any directed cycle. 
\end{proof}

\subsection{The $\NSPACE$ algorithm to compute the trim NFA $\cAcS$.}
The method is standard and is essentially the same as in \cite{jez16jacm,DiekertJP16,CiobanuDiekertElder2016ijac}. Therefore we give a rough sketch only. 
The key is the upper bound in \prref{cor:allround}: it is enough to consider 
states $(W,B,\cX,\theta,\mu)$ where $\Abs{W} \leq 3000 \del n\in \Oh(\abs H \cdot n^2) $. This implies that the maximal length of an equation and the maximal number of $H$-visible letters is in 
 $\Oh(\abs H \cdot n^2)\sse \Oh(\abs H {\Abs \cS}^2)$. This in turn gives the upper bound  $\Oh({\abs H}^2  {\Abs \cS}^2)$ on the alphabet $C$. It is also clear that we need at most    $\Oh({\Abs \cS})$ variables. To each symbol we have to attach its $\mu$-value in the finite monoid $N$. 

 By \prref{sec:mtw} storing a $\mu$-value costs
 $m(\cS)$ bits by (\ref{eq:moncS}).
  As a consequence
 we can specify a state $E$ (and therefore a \tra 
 $E\arc h E'$) in $\cAcS$ with $\Oh(|H|\cdot \Abs{\cS}^2\cdot \log |A|\cdot m(\cS)\cdot \log \Abs{\cS})$ bits. 
 
 Our algorithm must output all  \tra{s}
 $E\arc h E'$  which belong to $\cAcS$. Hence, we consider all candidates 
  $E\arc h E'$ based on the upper bound of bits for their specification one after another in some order, say in some lexicographical order. The algorithm has to decide if it outputs the \tra or whether it moves to the next candidate. 
  Thus, when considering whether or not  $E\arc h E'$ belongs to $\cAcS$, then the algorithm 
 guesses a path of \tra{s} from an initial state to the state $E$ and a path of \tra{s} from $E'$ to a final state. If the guess is successful, then it outputs $E\arc h E'$ and it moves to the next candidate. If unsuccessful, then we apply again the theorem of  Immerman-Szelepcs{\'e}nyi: $\NSPACE(|H| \Abs{\cS}^2\log |A|\, m(\cS)\log \Abs{\cS})$ is closed under complementation. 
Hence, the algorithm ``knows'' whether or not $E\arc h E'$ belongs to $\cAcS$ before moving to the next candidate. 

The proof of \prref{thm:central} is complete, and the first part of the paper is finished.

\section{Part 2: The existential theory with rational constraints for virtually free groups}
It was shown in \cite{DahmaniGui10,LohSen06} that the existential theory with rational constraints in f.g.~virtually free groups is decidable. Our main result (\prref{thm:virtfreestnf}) provides an effective \edtol description for the full set of satisfying assignments to  a Boolean formula in  free variables over equations and  rational constraints. In order to make our statement precise we need some preparation. 

\subsection{NFAs revisited}\label{sec:nfarev}
Let $M$ and $M'$ be finitely generated monoids. In the application $M$ and $M'$ are fixed and not part of the input. Therefore we can
 define 
the size of an NFA over $M$ (resp.~$M'$) which is, up to a constant, independent of the generating set. 

We begin by choosing any finite generating set $\Sig \sse M$. Then we specify an NFA for  $M$  as tuple $\cA=(Q,\Sig,\del,\inI,\finF)$ where 
the set of transitions $\del$ is finite and satisfies $\del\sse Q\times \Sig^* \times Q$. Having this, a natural definition for the \emph{input size of $\cA$} is 
\begin{equation}\label{eq:sizenfa}
\Abs{\cA}_{\text{in},\Sig}= \abs Q + \abs \del +\sum_{(p,u,q)\in \del}\abs u.
\end{equation}
The transitions in $\phi(\cA)$ might be labeled by words of length greater than $1$. However, this can be ``repaired'' easily by replacing
a \tra $(p,b_0\cdots b_k,q)$ with $b_i\in \Gam$ and $k>0$ by a sequence of \tra{s} 
$$(p,b_0,p_1)\lds (p_{k-1},b_{k-1},p_{k}),(p_{k},b_{k},q)$$
were $p_1\lds p_{k-1}$ are fresh states.
The input size of the new automaton is at most twice as large as before. 
The exact size of an NFA $\cA$ is of not important. We let 
\begin{align}\label{eq:insize}
\Absin{\cA}= \Theta(\Abs{\cA}_{\text{in},\Sig}).
\end{align}
This is well-defined, since if  we move to another finite generating set $\Sig'$ for $M$, then we see 
$\Abs{\cA}_{\text{in},\Sig'}\in \Oh(\Abs{\cA}_{\text{in},\Sig})$. 
Thus, it is convenient to denote $\cA$ simply as  $\cA=(Q,M,\del,\inI,\finF)$ because then the interpretation  $L(\cA)\sse M$ is encoded in the syntax. Still, we can use $\Abs{\cA}_{\text{in}}$ up to multiplicative constants.

Let $\phi\colon M\to M'$ be a \hom to a monoid with a finite generating set $\Sig'$, then the NFA  $\phi(\cA)$ is defined as $(Q,M,\phi(\del),\inI,\finF)$
where $\phi(p,a,q)= (p,\phi(a),q)$.
 For $s,t\in Q$ let $L(\cA,s,t)= L(\cA,M,\del,\os s,\os t)$.
If $\abs{\phi(a)}\in \Oh(1)$ for all $a \in \Sig$, then 
there is a  result which is again independent of the choice of $\Sig$ and $\Sig'$:. We have 
\begin{align}\label{eq:NFAtrans}
\Abs{\phi(\cA)}_{\text{in}} \in \Theta(\Abs{\cA}_{\text{in}}) \text{ and } \forall s,t\in Q:\, \phi(L(\cA,s,t))= L(\phi(\cA),s,t)).
\end{align}

\subsection{Exponential expressions}\label{sec:expexp}
The ideas and results in this section are not new. The notion of \emph{exponential expression} was  proposed, for  example, by Plandowski in \cite{pla04jacm}. 
 For the application to $\SL(2,\Z)$  exponential expressions are crucial to show a complexity within  $\PSPACE$.
Intuitively it is more natural to represent strings by allowing  exponents. For example, if $u$ is a word, then it is more natural  to write $u^{100}$ rather than in plain form by repeating $u$ a hundred times
$uuuuuuuuuuuuuuuuuuuuu \cdots$.

Exponential expressions (and plain exponential expressions) over an alphabet $\Sig$ and their sizes are defined inductively as follows.

\begin{enumerate}
\item Every word $w\in \Sig^*$ is a plain exponential expression of size $\Abs w= \abs w$.
\item Every plain exponential expression is an exponential expression.
\item If $E,E'$ are exponential expressions, then the concatenation 
$EE'$ is an  exponential expression of size $\Abs{EE'}= \Abs{E}+\Abs{E'}$.
If $E,E'$ are plain, then $EE'$ is plain, too.
\item  If $E$ is an exponential expression and $k\in \N$, then 
$E^k$ is an exponential expression of size $\Abs{E^k}= 1+ \Abs{E}+\log k$.
\end{enumerate}
Since $\Sig$ is equipped with an \invol, we define for all $k\in \Z$ the expression $E^{-k}$ as a synonym  for ${\ov E}^k$; and we let  $E^0$ denote the empty word $1$. The size of the  expression $E^0$ is still $\Abs E +1$.

In the following we allow that 
an equation appearing in a Boolean formula $\Phi$ is written as $E=E'$ where $E$ and $E'$ are exponential expressions. We view $E$ and $E'$ as words $\Sig^*$ which have a special encoding in a compact form.

\subsection{The existential theory  with constraints and expressions}\label{sec:exthco}
As above $M$ denotes a finitely generated monoid with \invol. 
We let $\Sig \sse M$ be any finite symmetric set of generators: that is, 
 $a\in \Sig \implies \ov a \in \Sig$. Let $\pi\colon \Sig^*\to M$ be the canonical morphism which is induced by the inclusion  $\Sig\sse M$. 
By $\OO$  we denote a countable set of variables such that $M\cap \OO=\es$. 
Without restriction we assume that 
$\OO$ is a set with \invol and $X\neq \ov X$ for all $X\in \OO$. 
As usual, we let $\ov g=\oi g$ for group elements.

The existential theory of $M$ with rational constraints and exponential expressions is defined with the help of Boolean formulae in free variables from $\OO$.
As we did in \prref{sec:mtw}, we obtain more accurate (and therefore better) complexity results if we define the size of a Boolean formula $\Phi$ 
as a pair $(\Abs{\Phi}_\text{eq},\Abs{\Phi}_\text{rat})$. 
The parameter $\Abs{\Phi}_\text{eq}$ behaves as if all NFAs defining the rational constraints were of constant size. Thus, essentially, it adds up the sizes of the equations of the exponential expressions defining the equations. This is reflected by the index ``eq''. 
The parameter $\Abs{\Phi}_\text{rat}$ adds up the input sizes for the NFAs which define the rational constraints. This is reflected by the index ``rat''.

The formal definitions are as follows. Here we assume that every constraint 
 $X\in L$ with   $L\in \Rat(M)$ is given as  $X\in \pi L(\cA)$ (resp.~$X\in L(\cA)$) where $\cA$ is an NFA  as in \prref{sec:nfarev}. Exponential expressions  
 were defined in \prref{sec:expexp}.
\begin{enumerate}
\item Every \emph{atomic formula} is Boolean formula. The atomic formulae are: 
\begin{itemize}
\item The constant $\bot$ (meaning ``false'')\\ $\Abseq{\bot}=\Absrat{\bot}=1$. 
\item Exponential expressions $E=E'$  
 over $(\Sig\cup \OO)^*$. \\ $\Abseq{E=E'} = 1+\Abseq{E} + \Abseq{E'}$
 and $\Absrat{E=E'} = 0$. 
 \item 
Constraints $X\in L(\cA)$. \\ 
$\Abseq{X\in L(\cA)}=1$ and $\Absrat{X\in L(\cA)}=\Absin{\cA}$.
\end{itemize}
\item If $\Phi,\Psi$ are Boolean formulae, then so are 
$(\Phi\vee \Psi)$, 
$(\Phi\wedge \Psi)$, and $(\neg \Phi)$, but  we omit brackets when possible.\\ 
$\Abs{\Phi\vee \Psi}_\star= \Abs{\Phi\wedge \Psi}_\star = \Abs{\Phi}_\star + \Abs{\Psi}_\star$, and 
$\Abs{\neg\Phi}_\star = \Abs{\Phi}_\star$ for $\star\in \os{\text{eq}, \text{rat}}$. 
\end{enumerate}
 
Let $\Phi$ be a Boolean formula and $\sig\colon \OO\to {M}$ be a \morph (that is, a mapping respecting the \invol). Then  the truth value 
$\sig(\Phi)$ is defined in the obvious way. 
If there exists some $\sig$ with $\sig(\Phi)= \text{ true}$, then we say that 
$\Phi$ is \emph{satisfiable}. We also say that  $\sig$ is a \emph{\solu} if $\sig(\Phi)= \text{ true}$ because it solves the satisfiability problem. So, we do not distinguish between satisfying assignments and \solu{s}. 
The \emph{existential (first-order) theory with rational constraints} refers to 
the set of satisfiable Boolean formulae 
$$\exists\FOTh({M},\Rat)= \set{\Phi}{\exists\; \sig\colon \OO\to {M} \text{ such that } \sig(\Phi)= \text{ true}}.$$

We are not only interested to \emph{decide} $\exists\FOTh({M},\Rat)$, what we aim for is an algorithm which produces on input a Boolean formula $\Phi$ an effective description of the \emph{full solution set $\cSol(\Phi, M)$}. To define it properly
we let $\cX_\Phi$ be the set of variables $X$ such that $X$ or $\ov X$ appears in $\Phi$. We let 
\begin{equation}\label{eq:fullsolset}
\cSol(\Phi, M)= \set{\sig\colon \cX_\Phi\to M}{\sig(\Phi)= \text{ true}}.
\end{equation}
Note that $\sig\colon \OO\to M$ satisfies $\Phi$ \IFF its restriction to 
$ \cX_\Phi$ satisfies $\Phi$. It is also clear that every \morph $\sig\colon \cX_\Phi\to M$ satisfying $\Phi$ can be extended to a \morph $\sig\colon \OO \to M$ satisfying $\Phi$. 
If the context of $M$ is clear, we abbreviate $\cSol(\Phi) = \cSol(\Phi, M)$. 
Once we have chosen a presentation 
$\pi\colon S\to M$ where $S$ is finite  and $\pi$ is onto, then we typically represent
elements of $M$ by words over $S$ and a \morph $\sig\colon \OO\to M$ is defined via a mapping $\sig\colon \OO\to S^*$. Moreover, without restriction $\OO$ comes with a linear order. If 
$\os{X_1\lds X_k}$ is the subset of the first $k$ variables with $X_i\leq X_j$ for all $i\leq j$, then we let 
\begin{equation}\label{eq:fullsset}
\cSol_{S,k}(\Phi)= 
\set{(\sig(X_1)\lds \sig(X_k)) \in S^*\times \cdots \times S^*}{\exists \sig:\, \pi\sig(\Phi)= \text{ true}}.
\end{equation}
 Clearly, to decide $\exists\FOTh({M},\Rat)$ is the same as to decide on input $\Phi$ whether or not $\cSol(\Phi)$ is empty. Moreover, $\cSol(\Phi)=\es \iff \cSol_{S,0}(\Phi)=\es$. Note that either $\cSol_{S,0}=\es$ or $\cSol_{S,0}=\os\es$. 
 We will see that $\cSol_{S,k}(\Phi)$ is an effective \edtol relation for every $k$ if $M$ is a \fg virtually free group.

When proving this result  for virtually free groups we make various transformations on NFAs (which up to a constant factor don't change ${\Abs \cA}_\text{in}$) before, eventually, 
we switch to Boolean matrices.

\subsection{Removing exponential expressions in $\Phi$}\label{sec:remexpexp}
Exponential expressions in Boolean formulae as in (\ref{eq:fullsolset}) are  used
because they may reduce the size of $\Phi$ significantly. On the other hand, with the help of more variables we can transform  $\Phi$ into a new formula 
$\Psi$ where all equations are written in plain form as $U=V$. The transformation is not expensive; and it doesn't change the full \solu set: 
\begin{proposition}\label{prop:SLP}
There is a deterministic algorithm working in linear space which takes as input 
a Boolean formula $\Phi$ using exponential expressions.

The output is a formula $\Psi$ having the following properties. 
\begin{enumerate}
\item Equations in $\Psi$ appear in plain form as $U=V$. Hence, 
$\Abs {UV}= \abs {UV}$.
\item $\Abseq {\Psi} \in \Oh(\Abseq {\Phi})$.
\item $\Absrat {\Psi} = \Absrat {\Phi}$.
\item $\cX_\Phi \sse \cX_\Psi$.
\item The restriction $(\sig\colon \cX_\Psi\to M) \mapsto (\sig\colon \cX_\Phi\to M)$ induces a bijection 
$$\cSol(\Psi)=\cSol(\Phi).$$
\end{enumerate}
\end{proposition}

\begin{proof}
 The method is standard: replace all exponential expressions by   straight-line programs (SLPs), see for example \cite{Lohrey2012survey,LohreySpringerbook2014}. More precisely, 
 as soon as an exponential expression $E=T^0$ with $e=0$ appears, replace the expression $E$ by the empty word $1$. 
If an exponential expression $T^1$ appears in $\Phi$, then replace every occurrence of $T^1$ simply by $T$. 
If an exponential expression $E=T^e$ with $e\geq 2$ appears in $\Phi$, then 
define a fresh variable $[T,e]$. (This implicitly means to introduce
 $\ov{[T,e]} = [\ov T,e]$, too. We don't repeat this anymore.)
 Whenever a variable $[T,\ell]$ is introduced where $\ell \geq 2$, then we  introduce another fresh variable  $[T,\floor{\ell/2}]$, too. \Ip $[T,e]$ and $[T,1]$ are introduced (but the condition $\ell \geq 2$  makes sure that  $[T,0]$ is never introduced). 
The  total number of fresh variables $[T,\ell]$ introduced that way is bounded 
 by $2(1+\log  e)\in \Oh(\log e)$.
 
After that step, replace all occurrences of $E$ by $[T,e]$, if $E$ was defined by $E=T^e$ in $\Phi$; and 
for each fresh $[T,\ell]$ with $\ell \geq 2$ introduce a new plain equation 
$$[T,\ell ] =  
\begin{cases}
 [T,\floor{\ell/2}\,]\;[T,\floor{\ell/2}\,]       & \text{ if $\ell$ is even,}\\
  [T,\floor{\ell/2}\,]\;[T,\floor{\ell/2}\,]\;[T,1]     & \text{ otherwise.}   
\end{cases}
$$
Moreover, introduce a single equation 
$[T,1] =T.$ The effect is that each occurrence of $E=T^e$, having size $\Abs E + \Abs T + 2 +\log  (1+e)$, is removed. The gain  of $\Abs T+\log  e$
is mitigated by $\Oh(\log e)$ new equations of constant size and one more equation $[T,1] =T$ of size $\Abs T+2$.

After that step replace $\Phi$ by the conjunction of 
$\Phi$ with the conjunction of the new equations.  Continue until all 
 equations are written in plain form. This defines the formula $\Psi$. 
 Note that is not necessary to add any constraint on the fresh variables $[T,\ell]$. Therefore, $\Absrat {\Psi}=  \Absrat {\Phi}$. The proposition follows. 
  \end{proof}

\section{Virtually free groups}
We restrict ourselves to the non-uniform complexity where the  given virtually free group is not part of the input. 
The restriction allows us to ignore the way a virtually free group is given to us. For example whether the group is given by a context-free grammar for the word problem or whether it is given as a fundamental group of a finite graph of finite groups may result in uniform complexities which differ exponentially. We refer the interested reader to the arXiv version of \cite{SenizerguesW18} for more details. See also \prref{rem:sizeGam}.

In the following  $G$ denotes a fixed finitely generated virtually free group.
Thus, there is a finitely generated free subgroup $F$ such that $H=G/F$ is finite. 
Replacing $F$ by the normal subgroup $\bigcap\set{gF\oi g}{g\in G}$ (which is of finite index in $G$) 
we can assume without restriction that $F$ is normal and that $H$ is a  finite group.   
That is, we start with some surjective \hom $\gam\colon G\to H$ where $H$ is finite 
and the kernel $\ker(\gam)$ is a \fg free group.  This yields a short exact sequence: 
\begin{equation}\label{eq:shexseq}
1 \to F \arc \iota G \arc \gam H\to 1.
\end{equation}
Choosing a generating set for $F$ and a set of coset representatives  from $H$, we obtain a generating set for $G$. We need generating sets which are closed 
under \invol, so  we are more specific. We use the following definition. 
\begin{definition}\label{def:stansetgen}
Let $G$ be given as in  \prref{eq:shexseq}. 
 We say that a subset $S$  of $G$ is 
a \emph{standard generating set} for $G$ 
  if the following conditions are satisfied.
\begin{itemize}
\item $S$ can be written as a union $A_+\cup A_-\cup H_+ \cup H_-\sse G$. 
\item $A_+$ is a basis for $F$, that is  $F=F(A_+)$.
\item $a\in A_+\iff \oi a\in A_-$ for all $a\in A=A_+\cup A_-$. 
\item 
$\gam$ induces a bijection between 
$H'= H_+\cup \os 1$ and $H$.  
\item $H_-=\set{h\in G}{\oi h \in H_+}$.
\end{itemize}    
\end{definition}
Every  standard generating set  is closed under the \invol with $\ov b = \oi  b\in G$. The three set $A_+$, $A_-$, and $H'$ are pairwise disjoint subsets of $G$. 
There is a bijection between $H_+$ and $H_-$, but perhaps 
$H_+\cap H_-\neq \es$.

Let $\pi_S\colon S^*\to G$ denote the canonical projection.
We say that $\wh w\in S^*$ is in  \emph{standard normal form} if we can write $\wh w=uh$ where  
$u\in A^*$ is a freely reduced word (that is without factors $a\ov a$) and $h\in H'$. By $\stnf_S(G)$ we denote the set of standard normal forms.
For every $w\in S^*$ there is a unique $\stnf_S(w)\in \stnf_S(G)$ such that 
 $w= \stnf_S(w)$ in $G$. The set of freely reduced words over $A$ becomes
 $A^*\cap\stnf_S(G)$; and we let ${\rednf}_A(G)=A^*\cap\stnf_S(G)$. Hence,
\begin{equation}\label{eq:stnf}
{\rednf}_A(G)=\set{{\stnf}_S(w)}{w \in A^*} = \set{w \in A^*}{ w \text{ is freely reduced}}.
\end{equation}

\begin{theorem}\label{thm:virtfreestnf}
Let $G$ be a finitely generated virtually free group. Then \wrt  any short exact sequence as in 
(\ref{eq:shexseq}) there is 
a standard generating set $S$  and an $\NSPACE(\Abs{\Phi}_{\text{eq}}^2 (\Absrat{\Phi}^2 +\log\Abseq{\Phi}))$ algorithm which performs the following task. 
It takes as input a Boolean formula $\Phi$ (according to \prref{sec:exthco})  with $\cX_\Phi= \os{X_1,\ov X_1\lds X_k,\ov X_k}$ such that $X_i$ is the $i$th variable in some fixed chosen linear order on $\OO$. 
The output is an extended alphabet $C$ of size $\Oh(\Abseq{\Phi}^2)$, letters $d_i\in C$ for all $1\leq i \leq k$, and a trim NFA ${\cA}_\Phi$ accepting a rational set of 
$S$-\morphs over $C^*$ such that the \edtol relation 
$$\set{(h(d_1)\lds h(d_k)) \in C^* \times \cdots \times C^* }{h \in L({\cA}_\Phi)}$$ is equal to the full solution set in standard normal forms 
$${\cSol}_{S,k}(\Phi)= \set{(\sig(X_1)\lds \sig(X_k))\in {\stnf}_S(G)^k}{\pi_S\sig(\Phi)= \text{ true}}.$$
Moreover, ${\cSol}(\Phi)=\es$ \IFF $L({\cA}_\Phi)= \es$;  and $\abs{{\cSol}(\Phi)}<\infty$  
\IFF ${\cA}_\Phi$ doesn't contain any directed cycle. 
\end{theorem}
\begin{remark}\label{rem:onepar}
In a simplified analysis using a single parameter, the natural choice is to  define $\Abs{\Phi}= \Abseq{\Phi} +\Absrat{\Phi}$.
This yields
\begin{equation}\label{eq:op1}
\NSPACE(\Abseq{\Phi}^2 (\Absrat{\Phi}^2 +\log\Abseq{\Phi}))\sse \NSPACE(\Abs{\Phi}^4) \sse \PSPACE.
\end{equation}
If $\Absrat{\Phi}\in \Oh( \sqrt{\log\Abseq{\Phi}})$, then we have
\begin{equation}\label{eq:op2}
\NSPACE(\Abseq{\Phi}^2 (\Absrat{\Phi}^2 +\log\Abseq{\Phi}))\sse 
\NSPACE(\Abs{\Phi}^2 \log\Abs{\Phi}).
\end{equation}
\end{remark}

\begin{remark}\label{rem:whykay}
Let us comment why we consider only a subset of the first $k$ variables of $\cX_+$ rather than all variables. The reason is that during the proof we manipulate $\Phi$ in various ways including some which introduce fresh variables. But these new variables are just auxiliary symbols, and we make sure that they don't enlarge  the full \solu set. If we introduce fresh variables, then we put them in the linear order behind the first $k$ variables.  Therefore, there is no risk to denote ${\cSol}(\Phi)$ as ${\cSol}_{S,k}(\Phi)$. 
\end{remark}

\subsection{Proof of \prref{thm:virtfreestnf}, Phase 1}\label{sec:prvf1}
Using \prref{prop:SLP} we may assume that all equations in $\Phi$ are written in plain form as $U=V$ where $U,V\in (S \cup \cX)^*$. 
In the following we introduce many fresh variables into $\cX$.
The enlarged set is still called $\cX$. Moreover, 
we choose a subset  of \emph{positive} variables $\cX_+$ such that 
$\os{X_1\lds X_k} \sse \cX_+$, $\cX= \cX_+ \cup \set{\ov X}{X\in \cX_+}$, 
 and $X\in \cX_+\iff \ov X \notin \cX_+$.
 
Having this, we push all negations to the atomic formulae using De Morgan's law. This increase the size at most by the number of atomic formulae. 
For each  inequality $\neg (U=V)$ we introduce a  fresh variable $X$ and then 
we replace  $\neg (U=V)$ by  the conjunction $U=VX\wedge  \neg(X\in \os 1)$.
This increases the size by the number of inequalities since the singleton $\os 1$ is accepted by a two-state NFA. Thus, without restriction $\Phi$ doesn't contain any negation and only three types of atomic formulae:
$$U=V,  \quad X\in \pi_S(L(\cA)),\quad \text{and } X\notin \pi_S(L(\cA)) \text{ where $U,V\in S^*$ and } X\in \cX_+.$$
Here, denotes  an NFA of the form 
$\cA=(Q,S,\del,\inI,\finF)$ with $\del\sse Q\times S \times Q$. We may assume   because
 for every NFA $\cA$ there is another NFA $\ov \cA$ of the same size such that $L(\ov \cA)=\ov{L(\cA)}$ (the complement of $L(\cA)$).
 We also write $X\in L(\cA)$ or $ X\notin L(\cA)$ because $S\sse G$ and therefore we can view 
 $L(\cA)$ directly as a rational subset of $G$. 
\begin{lemma}\label{lem:redphase1}
Let $\gam\colon G\to H$ be as above. \Ip $F = \oi\gam(H)$ is  free and $H$ is finite. It is enough to prove \prref{thm:virtfreestnf} under the following assumptions about the input formula $\Phi$.  
\begin{itemize} 
\item $\Phi$ implies $\bigwedge\set{X\in F}{X\in \cX_+}.$ (Note that the syntax $X\in F$ makes sense since the \fg free group $F$ is a rational subset in $G$.)
\item If an NFA $\cA$ appears in $\Phi$, then $L(\cA)\sse F$ where 
$\cA$ is an NFA over $G$ and the transitions are labeled with an arbitrary, but fixed, finite set of generators of $G$. 
\item  $\Phi$ is a conjunction where each 
atomic formula is either an equation in plain form $U=V$ with $UV\in (S \cup \cX)^*$ or $X\in L(\cA)$ or  $X\notin L(\cA)$.
\end{itemize}
\end{lemma}

\begin{remark}\label{rem:consredphase1}
Assume that $\Phi$ satisfies the assumptions of \prref{lem:redphase1}. Then 
\prref{thm:virtfreestnf}  implies that there is standard set of generators $S$ containing a basis 
$A_+$ of $F$ such that 
$${\cSol}_{S,k}(\Phi)= \set{(\sig(X_1)\lds \sig(X_k))\in {\rednf}_A(G)^k}{\pi_S\sig(\Phi)= \text{ true}}.$$
\Ip the full solution set ${\cSol}_{S,k}(\Phi)$ is an \edtol relation over freely reduced words of $A$. 
\end{remark}

The proof of \prref{lem:redphase1} is based on two closure properties:
1. finite unions of \edtol (resp.~rational) languages in a monoid $M$ are \edtol (resp.~rational); and 
2. if $L\sse M$ is an \edtol (resp.~rational) language and $m\in M$, then $Lm$ is \edtol (resp.~rational). The analogous statements hold for \edtol relations. 

\begin{proof}[Proof of \prref{lem:redphase1}]
The difficult part is to show the first and third item  because we have to respect the given space bounds.  The second item is very easy to show, and  we prove it ``on the fly'' when showing the first item. 

Let  $\Phi$ be any input formula for  \prref{thm:virtfreestnf}. We wish to add the constraint $X\in  F$ for all variables. 
This requires the  introduction of fresh variables. More precisely, for each $X\in \cX_+$ and $g\in H'$  we introduce a new variable $X_g$
with  $\ov{X_g} = \ov X_g$; and we construct an NFA $\cA_g$ such that 
 $L(\cA_g) = L(\cA)\ov g$  for each $\cA$.
The NFA  $\cA_g$ is obtained by adding a single new final  state and new 
 \tra{s} from the former final states to the new single one, all of them are labeled by the letter $\ov g$. The size of $\cA_g$ increases by a constant. 
Moreover, each function $\eta\colon \cX_+ \to H$ defines a new formula 
$\Phi'_\eta$ over the variables $X_{\eta(X)}$ as follows: 
every occurrence of $X$ (resp.~$\ov X$) inside an 
equation is replaced by $X_g\ov g$ (resp.~$g \ov X_g$) where $g=\eta(X)$. The length of each equation is at most doubled. 
Every constraint 
$X\in L(\cA)$ (resp.~$X\notin L(\cA)$) is replaced by the 
$X_g\in L(\cA_g)$ (resp.~$X_g\notin L(\cA_g)$).  Recall that $L(\cA_g) = L(\cA)\ov g$.
 Let $\Phi'_\eta$ denote the result of that transformation. Then we let 
\begin{equation}\label{eq:phieta}
\Phi_\eta= \Phi_\eta' \wedge \bigwedge \set{X_{\eta(X)}\in  F}{X\in \cX_+}.
\end{equation}
Note that a constraint $X_g\in  F$ is the same as $\gam(X_g)=1$. Therefore we can use $H$ as a recognizing finite monoid for all $X_g$. Since $H$ is of constant size, the size of $\Phi_\eta$ is in $\Oh(\Abseq{\Phi},\Absrat{\Phi})$. 
All variables in $\Phi_\eta$ are of the form $X_{\eta(X)}$ or  $\ov{X}_{\eta(X)}$. The old variables $X\in \cX$ are still present but not used in any $\Phi_\eta$. Therefore, inside each $\Phi_\eta$ we rename all $X_{\eta(X)}$
by $X$. After that the variables $X_g$ are superfluous: we remove them from
$\cX$. Thus, each $\Phi_\eta$ uses the same set of $\cV_+$ as $\Phi$ did.

Each formula $\Phi_\eta$ is written again in disjunctive normal form 
$\Phi_\eta= \bigvee\set{\Phi_{\eta,j}}{j \in I_\eta}$ where each 
index set  $I_\eta$ has again (at most) exponential in the size of $\Phi$.

Having this, we see that $\Phi$ is equivalent to  the following disjunction
\begin{align}\label{eq:disginH}
\wt\Phi= \bigvee\set{\Phi_{\eta,j}}{\eta\in H^{\cX_+} \wedge j \in I_\eta}.
\end{align}
Note that $\Phi$ and $\wt\Phi$ use the same set $\cX_+$ of positive variables. It is also clear how to transform a \solu $\sig$ for $\Phi$ into a \solu $\wt \sig$ for $\wt \Phi$ and vice versa: 
if $\sig$ solves $\Phi$, then $\wt \sig(X_g) =\sig(X)g$ solves $\wt \Phi$ and 
if $\wt \sig$ solves $\wt \Phi$, then $\sig(X) =\wt \sig(X_g)\ov g$ solves $\Phi$.

Since $\abs H$ is a constant it is easy to see that the  number of disjunctions in 
\prref{eq:disginH} is (at most) exponential in the size of $\Phi$. But it
can also happen that 
the size of $\wt\Phi$ is exponential in the size of $\Phi$, so in general we have  no way to store $\wt \Phi$ within the given space bound. 
What we do instead is to construct NFAs $\cA_{\eta,j}$ for each $\Phi_{\eta,j}$, one after another, such that $L(\cA_{\eta,j})$ defines the \edtol relation of the full \solu set for $\Phi_{\eta,j}$. 

More precisely, suppose we have shown \prref{thm:virtfreestnf} for each $\Phi_{\eta,j}$ which
is a conjunction of constraints and equations. Then indeed, 
for all $(\eta,j)$, one after another, we can output some NFA $\cA_{\eta,j}$ where the transitions are 
 labeled by \Endos over (the same) extended alphabet $C$ such that
 \begin{equation}\label{eq:goodeta}
{\cSol}_{S,k}(\Phi_{\eta,j}) =\set{(h(d_{1})\lds h(d_{{k}})) \in C^* \times \cdots \times C^* }{h \in L(\cA_{\eta,j})}.
\end{equation}
We can also assume that all these NFAs use exactly the same 
set of distinguished letters $\os{d_1 \lds d_k}$. As an output of the overall algorithm we obtain the disjoint union over all these NFAs $\cA_{\eta,j}$. Without restriction
$H' \cup H_- \sse C$, but the elements of $H' \cup H_-$ are not used in any 
NFA so far. Moreover, 
for each $d_i,\ov{d_i}$ we may assume that there are letters $c_i,\ov{c_i}$, again still not used by  any 
$\cA_{\eta,j}$.
We add one more new state and connect this new 
state with all final states in $\cA_{\eta,j}$ via a single transition
labeled by the \Endo $h\in \End(C^*)$ which is defined by $h(c_i)= d_ig$
where $d_i$ corresponds to the variable $X_i\in \cX_+$ and $g=\eta(X_i)\in H'$. 
The new state becomes the single final state of the ``union'' automaton. 

We conclude that it is enough to show \prref{thm:virtfreestnf} for each $\Phi_{\eta,j}$. Since $\Phi_{\eta,j}$ satisfies properties as required by \prref{lem:redphase1}, the lemma follows.\end{proof}

\subsection{Proof of \prref{thm:virtfreestnf}, Phase 2.\\
Embedding into a semi-direct product}\label{sec:prvf2}
Let $E$ be a finite set with \invol. Then $\F(E)$ denotes the group
$\F(E)=E^*/\set{e\ov e = 1}{e\in E}$. If the involution on $E$ is without fixed points, then we can write $E=E_+\cup E_-$ such that $e\in E_+\iff \ov e \in E_-$; and the inclusion $E_+$ into $E$ induces an isomorphism between the free group 
$F(E_+)$ with basis $E_+$ and $\F(E)$. The group $\F({E})$ is called \emph{specular} in \cite{BertheFDDLPRR17}, which means  it is  the free product of a free group with groups of order two.

In the following we use that $G$ is the fundamental group of a finite graph of finite groups \cite{Karrass73}, which 
enables us to reduce questions about equations with rational constraints in $G$ to questions  about twisted word equations with rational constraints.

Suppose that the group $H$  acts on  $E$ via a \morph $H\to \Aut(E)$. Thus, for  $(f,e)\in H\times E$ we have 
$f(e)\in E$ and $f(\ov e) = \ov{f(e)}$. 
We have  $\Aut(E)=\Aut(E^*)\sse \Aut(\F(E))$ and the action of $H$ on 
$E^*$ and $\F(E)$ defines two different (but related) semi-direct products 
${E}^*\rtimes H$  and $\F({E})\rtimes H$. 
The elements of ${E}^*\rtimes H$ (resp.~$\F({E})\rtimes H$) are the pairs 
$(u,f)\in {E}^*\times H$ (resp.~$\F({E})\times H$) and the multiplication is defined by
$$(u,f)\cdot (v,g) = (u f(v), fg).$$
The semi-direct product ${E}^*\rtimes H$ is a monoid with \invol by
$$\ov{(u,f)} = (\oi f(\ov u), \oi f).$$
It is also clear that $\oi{(u,f)} = (\oi f(\oi u), \oi f)$ in the group 
$\F({E})\rtimes H$.

The free monoid ${E}^*$ embeds into ${E}^*\rtimes H$ via $e\mapsto (e,1)$ and the group $H$ embeds into ${E}^*\rtimes H$ via $f\mapsto (1,f)$.
Having this, we obtain:
\begin{align*}
\F({E})\rtimes H&= ({E}^*/\set{e\ov e = 1}{e\in {E}})\rtimes H \\ &= {E}^*\times H/\set{(e,1)(\ov e,1) = (1,1)}{e\in {E}}.
\end{align*}
Since we identify $E$ with $E\times \os 1$ and $H$ with $\os{1} \times H$, we can write:
\begin{align}\label{eq:sdprod1}
\forall a\in E, g\in H: ga\ov g &= (1,g)(a,1)(1,\ov g) = (g(a),1) = g(a). 
\end{align}
Thus, $gx\ov g= g(x)$ for $g\in H$, $x\in E^*$, and  $x\in \F(E)$.
Let ${\Gam}= E\cup H $ and $H\cap E^*=\os1$. Thus, $1$ the identity element in $H$ is identified with the empty word in $E^*$.  It also appears as a letter in $\Gam$. The interpretation $e\in E$ as $(e,1)$ and  $f\in H$ as $(1,f)$ yields  canonical surjective \morph{s}
\begin{align}\label{eq:pressd}
\pi_\Gam\colon {\Gam}^* \to E^*\rtimes H \to \F({E})\rtimes H.
\end{align}
Our proof of \prref{thm:virtfreestnf} relies on  \prref{prop:DW17}.
 \begin{proposition}[\cite{DiekertW17crm}, Sec.~2.4.5]\label{prop:DW17}
Let $G$ be a finitely generated virtually free  group and $\gam\colon G\to H$ be a \hom onto a finite group $H$ such that the kernel $F=\ker(\gam)$ is free. 
Then $G$ embeds into 
a semi-direct product of the form $\F({E})\rtimes H$; and we can construct an injective \hom $\phi\colon G\to \F({E})\rtimes H$ and a partition $E=A\cup T$ into two subalphabets such that 
$F= \set{x\in G}{\phi(x)\in \F(E)}$ is isomorphic to $\F(A)$. Moreover, using that 
isomorphism, we can embed $\F(A)$ into $G$ such that
$\phi(a) \in T^*aT^*$   such that $\phi(a)$ is freely reduced in $E^*$. The embedding of $G$ into $\F({E})\rtimes H$ is also depicted in \prref{fig:embed}.
\end{proposition}
\begin{proof}
 The first assertion is the nontrivial direction in \cite[Cor.~2.4.23]{DiekertW17crm}. The corollary says that $F= \set{x\in G}{\phi(x)\in \F(E)}$ is 
 a free factor of  $\F(E)$. This means that $\F(E)$ is a free product of $\F(A)$ with $\F(T)$. The additional property  $\phi(a) \in T^*AT^*$ for 
all $a\in A$ is a special case of \cite[Prop.~2.4.22]{DiekertW17crm}.
\end{proof}

\begin{figure}
\begin{center}
    \begin{tikzpicture}[scale=0.9, >=latex,shorten > =1pt,auto,initial text={}, every state/.style={minimum size=8mm}, node distance=2cm]
    \node[] (1) {$1$};
    \node[right of=1] (F) {$\F(A)$};
    \node[right of=F] (G) {$G$};
    \node[right of=G] (H) {$H$};
    \node[right of=H] (1r) {$1$};
    \node[below of=1] (12) {$1$};
    \node[right of=12] (F2) {$\F(E)$};
    \node[right of=F2] (G2) {$\F(E) \rtimes H$};
    \node[right of=G2] (H2) {$H$};
    \node[right of=H2] (1r2) {$1$};

    \draw[->] (1) to (F);
    \draw[->] (F) to (G);
    \draw[->] (G) to node[above] {$\gam$} (H);
    \draw[->] (H) to (1r);

    \draw[->] (12) to (F2);
    \draw[->] (F2) to (G2);
    \draw[->] (G2) to (H2);
    \draw[->] (H2) to (1r2);

    \draw[->] (F) to node[right] {$\phi_A$}(F2);
    \draw[->] (G) to node[right] {$\phi$} (G2);
    \draw[->] (H) to node[right] {$\id{H}$} (H2);
    \end{tikzpicture}
   \caption{Embedding of $G$ with $E=A\cup T$ and $\phi_A(a) =\phi(a) \in T^*aT^*$.}
     \label{fig:embed}
\end{center}
\end{figure}

\begin{remark}\label{rem:sizeGam}
As we deal here with non-uniform complexity, we content ourselves to know that the embedding $\phi\colon G\to \F(E) \rtimes H$ can be effectively computed and therefore we can treat $\abs E$ as a constant. But in fact the proof of Lemma~30 in the arXiv version of \cite{SenizerguesW18} shows
\begin{equation}\label{eq:estimE}
\abs E \leq (\abs A + 2\abs H) \cdot \abs H.
\end{equation}
Thus, if $G$ is given to us as the fundamental group of a finite graph of finite groups, then the interested reader could derive uniform complexity bounds from the material presented here. 
\end{remark}

\begin{corollary}\label{cor:DW17}
We use the same notation as in \prref{prop:DW17}. Define a \morph
$\psi$ {}from $\F(E)$ onto $\F(A)$ by $\psi(a) = a$ and $\psi(t) = 1$ 
for $a\in A$ and $t\in E\sm A$. Then $\psi$ maps freely reduced words $\wh w\in \phi(A^*)$ to freely reduced words in $A^*$. 
\end{corollary}

\begin{proof}
The subgroup  $\phi(\F(A))$ of $\F(E)$ is generated by words of the form 
$\phi(v)$ where $v\in A^*$ is freely reduced over $A$. 
Thus, we can write every element in $w\in \phi(\F(A))$ as a word
$$w = \phi(a_1)\cdots \phi(a_m) \in T^*a_1 T^* \cdots T^*a_m T^*$$
such that $a_{i} \neq \ov{a_{i+1}}$ for $1\leq i <m$. 
Now, every freely reduced word  $\wh w$ can be obtained from some word $w$ as above by cancellation of factors $e \ov e$. Since $a_{i} \neq \ov{a_{i+1}}$ for $1\leq i <m$, we obtain 
$$\wh w \in \phi(a_1)\cdots \phi(a_m) \in T^*a_1 T^* \cdots T^*a_m T^*.$$
\Ip $\psi(\wh w) = a_1\cdots a_m$, and $a_1\cdots a_m$ is freely reduced by definition. 
\end{proof}

\begin{remark}\label{rem:BST80}
Let us give a few more comments how  \prref{prop:DW17} and \prref{cor:DW17} are shown in \cite{DiekertW17crm}. Since $G$ is a fundamental group of a finite graph of finite groups, it acts on its Bass-Serre tree $\cT$ without edge inversion \cite{serre80}. As the notation suggests, $\cT$ is indeed a tree: a connected acyclic undirected graph. 
The same is true for the free subgroup $F=\ker(\gam)$ of $G$: it acts on $\cT$ as a  graph \auto without edge inversion.  It follows  that $F$ has trivial intersection with all vertex groups because the vertex groups are finite and embed into $G$, see again \cite{serre80}. So, if the intersection was not trivial, the $G$ would have a finite nontrivial subgroup, but free groups are torsion free. Thus, $F$ acts on $\cG$ without vertex stabilizers and without  without edge inversion.

Now, let $\cG$ be the quotient graph $\cG= F\sm \cT$. The finite group 
$H$ acts on $\cG$: it permutes  the edges and vertices of $\cG$ by respecting the incidence relation. Moreover, $F$ appears as the fundamental group of the finite and connected simplicial graph $\cG$. This can be viewed as  the main structure theorem about groups acting on trees,  \cite[Thm.~13]{serre80}. That is, we can write $F=\pi_1(\cG)$. The point is that we always have two views on fundamental groups of a simplicial graph. The first view is to choose a base point $\star$ and we write
$\pi_1(\cG)= \pi_1(\cG,\star)$ where  $\pi_1(\cG,\star)$ is a set of paths from $\star$ to $\star$. The second view is to choose a spanning tree $T$ of $\cG$ and we realize $\pi_1(\cG)$ as 
$\pi_1(\cG)= \pi_1(\cG,T)$. That is $\pi_1(\cG)= \pi_1(\cG,T)= \F(E)/\set{t=1}{t\in T}$. The isomorphism between $\pi_1(\cG,\star)$ and $\pi_1(\cG,T)$ is induced 
by the inclusion of $\pi_1(\cG,\star)$ into $\F(E)$ followed by the projection 
$\F(E)$ onto the quotient $\pi_1(\cG,T)= \F(E)/\set{t=1}{t\in T}$.

Let ${E}$ be the set of edges in $\cG$. It is (in the sense of  \cite{serre80}) a finite alphabet: a finite set  with \invol without fixed points. 
Thus, we can view $E$ as a disjoint union $E_+\cup E_-$ with $e\in E_+\iff \ov e\in E_-$. As a set we identify $\F(E)=F({E_+})= {E}^*/\set{e\ov e = 1}{e\in {E}}$ with the regular set of reduced words in $E^*$.
Recall that a word is reduced \IFF no factor $e\ov e$  for $e\in E$ appears. 
Since $H$ acts  on the graph $\cG$,  each $g\in H$ acts on $E^*$ via a length preserving \auto{} which respects the \invol. Hence,  $w$ is reduced \IFF $g(w)$ is reduced. 
 \end{remark}

\subsection{Proof of \prref{thm:virtfreestnf}, Phase 3.\\
Transformation of $\Phi$ to $\Psi$}\label{sec:prvf3}
Let $\Phi$ be the input formula for showing  \prref{thm:virtfreestnf}. 
By \prref{lem:redphase1} we may assume that $\Phi= \Phi_{\eta,j}$ where 
$\Phi_{\eta,j}$ appears in \prref{eq:disginH}. Thus, $\Phi$ 
is given as a single conjunction of a special form where every variable is bounded by a constraint 
$X\in \F(A)$. It follows that the choice of $H'$ for the standard generating set doesn't effect
${\cSol}_{S,k}(\Phi)$ from this point on. 
Therefore, we can write 
\begin{equation}\label{eq:S2A}
\pi_A{\cSol}_{A,k}(\Phi)= \pi_A{\cSol}_{S,k}(\Phi)=\pi_S{\cSol}_{S,k}(\Phi)
\end{equation}
Next, we use the embedding of $G$ into 
the semi-direct product $\F(E)\rtimes H$ as given by  \prref{prop:DW17}
and \prref{fig:embed}. We are going to transform 
 $\Phi$ into a formula $\Psi$ over $\F(E)\rtimes H$ such that the inclusion of $G$ into the semi-direct product defines a bijection between
 $\cSol(\Phi)$ and $\cSol(\Psi)$. 

We construct $\Psi$ according to  the following steps.
\begin{enumerate}
\item We extend the embedding $\phi\colon G\to \F(E)\rtimes H$ to an embedding 
 $\phi\star \id{\cX}\colon G\star \cX^*\to (\F(E)\rtimes H)\star \cX^*$ and we replace 
 every equation $U=V$ in $\Phi$ by  $\phi\star \id{\cX}(U)=\phi \star \id{\cX}(V)$. 
 Identifying $E$, $H$, and $\cX$ with subsets of $(\F(E)\rtimes H)\star \cX^*$, we 
see that  $E\cup H\cup \cX$  generates the group 
$(\F(E)\rtimes H)\star \cX^*$. Hence,  every equation $\phi\star \id{\cX}(U)=\phi \star \id{\cX}(V)$ can be written as a plain equation over the alphabet $E\cup H\cup \cX$. As we have defined $\Gam= E\times \os 1 \cup \os 1 \times H = E\cup H$, we have $E\cup H\cup \cX = \Gam \cup \cX$. 

\item We replace every $\cA$ which appears in $\Phi$ by $\cA_1= \phi(\cA)$.
 That is $L(\cA_1)= \phi(L(\cA))$. By assumption we have
 $L(\cA_1)\sse \F(E)$. Hence, $L(\cA_1)$ is a rational subset of the free group $\F(E)$. 
 
Note that $\Abs{\cA_1}\in \Oh(\Abs{\cA})$. 
 Without restriction, we may assume that the transitions in $\cA_1$ are labeled by elements from $\Gam$. 
The property $\Abs{\cA_1}\in \Oh(\Abs{\cA})$ is not effected by that assumption.
Let $\Psi_1$ be the intermediate formula. 
It is clear that $\phi$ induces a bijection 
between $\cSol(\Phi)$ and $\cSol(\Psi_1)$.
\item We transform  each $\cA_1$ appearing 
in $\Psi_1$ into an NFA $\cB$  such that first, $L(\cA_1) = L(\cB) \sse \F(E)$ and second, the \tra{s} $\cB$ use labels from $E\cup \os{1}$, and  third $\Abs{L(\cB)}\in \Oh(\Abs{L(\cA_1)})$.
This  well known by \cite{sen96actainf,Silva02}, but not completely obvious.  In \prref{lem:sen96sil} we give a slightly simplified proof for the special situation 
of semi-direct products. 

 Let $\Psi$ be the corresponding formula. 
Since $L(\cA_1) = L(\cB)$, we have $\cSol(\Psi)=\cSol(\Psi_1)$.
\end{enumerate}
The construction of $\Psi$ is finished. We have 
\begin{equation}\label{eq:OhPhiPsi}
\phi(\cSol(\Phi))= \cSol(\Psi) \text{ and } (\Abseq{\Psi},\Absrat{\Psi})\in \Oh(\Abseq{\Phi},\Absrat{\Phi}). 
\end{equation}

\begin{lemma}[\cite{sen96actainf,Silva02}]\label{lem:sen96sil}
Let $\cA=(Q,\Gam,\del,\inI,\finF)$ be an NFA 
with the property $L(\cA)\sse \F(E)$.
Then there is an NFA $\cB$  such that $L(\cA) = L(\cB)$ where the \tra{s} $\cB$ use labels from $E\cup \os{1}$, and $\Absin{L(\cB)}\in \Oh(\Absin{L(\cA)})$.
\end{lemma}

\begin{proof}
{In the beginning we let $\del$ be any finite subset of  $Q\times \Gam^*\times Q$. By \prref{sec:nfarev}
and perhaps by doubling the size of $\cA$, we may assume that 
$\del\sse Q\times (\Gam\cup \os 1)\times Q$. Since 
$\Gam = E\cup H$ and $1\in H$ we may assume that $\del \sse Q\times H(E \cup \os 1) \times Q$. Thus, the label of every \tra is either an element from $H$ or from $E \cup \os 1$ or a product $ha$ where $h\in H$ and $a\in E$.}
 Moreover, we may assume that $\cA$ is trim. \Ip if we reach a state $p$ when reading a word $u$ from an initial state, then there is a word 
$v$ such that $uv\in L(\cA)$. Now, $L(\cA)\sse \F(E)$.  
We let  $\gam(p)= \gam(u)$. This is well-defined as $\gam(u) \oi{\gam(v)}=1$. 

For every state $p$ with $\gam(p)=g$  we  introduce exactly one more state 
$[p,g]$ and \tra{s} $p \arc {\ov g} [p,g] $ and $[p,g] \arc {g} p.$
This does not change the language accepted, and the NFA is still trim with 
$\gam([p,g])=1$. 
For each outgoing transition $p \arc {ha} q$ with $h\in H$ and $a\in E\cup \os 1$ we have 
$\gam(q)= f=  gh$; and 
there
is some $b\in E\cup \os 1$ such that $bf=fa$ in $G$ and  hence, we add a transition 
$[p,g] \arc {b}[q,f]$ as depicted in \prref{fig:sensil}.
\begin{figure}[h]
\begin{center}
 \begin{tikzpicture}[scale=0.9, >=latex,shorten > =1pt,auto,initial text={}, every state/.style={minimum size=8mm}, node distance=1.5 cm]

       \node[] (F) {$p$};
    \node[right of=F] (G) {$q$};

      \node[below of=F] (F2) {$[p,g]$};
    \node[right of=F2] (G2) {$[q,f]$};

     \draw[->] (F) to node[above] {$h a $}(G);
     
       \draw[->] (F2) to node[above] {$b$} (G2);
    
    \draw[->] (F) to node[right] {$\ov g$}(F2);
    \draw[->] (G2) to node[right] {$f$} (G);
       \end{tikzpicture}
   \caption{The equations $ b f = af$ and $f=gh$ imply $\ov g b f = ha$. }
     \label{fig:sensil}
\end{center}
\end{figure}

This doesn't change the language accepted as $\ov g b f= ha $ in $G$.
The larger NFA still accepts $L$, but the crucial point is that for  
$h_1a_1 \cdots h_ka_k\in L(A)$ we can accept the same element in $G$ by reading just labels from $E\cup \os 1$. This is easy to see by induction on $k$. 
Now, we remove all original states (they are no longer needed)
 and make
$[p,1]$ initial (resp.~final) \IFF $p$ was initial (resp.~final), to obtain the NFA $\cB$. 
By construction, we have 
$\Abs{\cB}_{\text{in},\Gam}\leq 2\Abs{\cA}_{\text{in},\Gam}$. This implies $\Absin{\cB}\in \Oh(\Absin{\cA})$. Recall that $\Absin{\cA}$ is well-defined up to a multiplicative constant only by \prref{eq:insize}. This makes $\Absin{\cA}$ independent of the choice of a finite generating set.  
\end{proof}

\subsection{Proof of \prref{thm:virtfreestnf}, Phase 4.\\
{}From $\Psi$ to $\Psiben$: applying the techniques of Benois}\label{sec:prvf15}
The transformation in this subsection doesn't effect the equations in $\Psi$.
We only change the NFAs $\cB$ such that they accept with every word $w\in E^*$ also the word $\wh w$ which is obtained by canceling all factors 
$e \ov e$. Nevertheless the rational subset $\pi_E(L(\cB))\sse \F(E)$ will not change. The techniques for the transformation is well known 
 by the work of Mich\`ele Benois \cite{ben69}. Therefore we call the new formula 
$\Psiben$. We will see that $\pi_E(\cSol_{E, k}(\Psi))
= \pi_E(\cSol_{E, k}(\Psiben))$.
For convenience of the reader, we explain the transformation in detail. We use notation from string rewriting. 

For  $u,v\in E^*$ we write $u\RAS{}{} v$ if $u=pq$ and $v=pe\ov e q$ 
for some $p,q\in E^*$ and $e\in E$. By $\RAS{*}{}$ we mean the reflexive and transitive closure of $\RAS{}{}$. Clearly, $u\RAS{*}{} v$ implies $\pi_E(u)=\pi_E(v)$. 
Moreover, $\pi_E(u)=\pi_E(v)$ implies $u\RAS{*}{}w\LAS{*}{}v$ for some $w\in E^*$. 

By $\F$ we denote the set of freely reduced words over $E$; and we identify $\F(E)$ with the regular set $\F\sse E^*$. These a words without any factor $e\ov e$ where $e\in E$ or, equivalently, the set of words $u$ such that $u \RAS*{} v$ implies $u=v$. The identification as sets is possible because $\pi_E$ yields a bijection from $\F$ onto $\F(E)$. 

The formula $\Psi$ uses NFAs $\cB$ where the transitions are labeled by letters from $E$ or by the empty word $1$, see \prref{sec:prvf3}. The interpretation so far is that  $L(\cB)$ denotes a rational subset in $\F(E)$. Now, we switch the viewpoint: $L(\cB)$ denotes a regular subset in $E^*$; and we replace all constraints $X\in \pi_E L(\cB)$ 
(resp.~$X\notin \pi_E L(\cB)$) by $X\in L(\cB)$ 
(resp.~$X\notin L(\cB)$). This is nothing but a change of notation, so we call the new formula still $\Psi$. However, we now on we consider the full solution set 
$\cSol_{E, k}(\Psi)$ as a relation over $E^*$. Thus, a \solu $\sig$ is given as 
a \morph $\sig\colon \cX\to E^*$. Recall that the ``actual'' \solu over the group 
$\F(E)\rtimes H$ is therefore given by $\pi_E \sig\colon  \cX\to \F(E)$.

\begin{lemma}[\cite{ben69}]\label{lem:benflood}
Let $\cB=(Q,E,\del; \inI, \finF)$ an NFA which appears in $\Psi$.
Thus, $\del\sse Q\times (E\cup \os 1)\times Q$ and $L(\cB)\sse E^*$.  
Then we can transform $\cB$ 
into an NFA $\cB'=(Q',E,\del'; \inI', \finF')$
such that $\abs{Q'}=\abs Q$, $\del'\sse Q'\times (E\cup \os 1)\times Q'$, 
and 
\begin{equation}\label{eq:benflood}
\pi_E(L(\cB')\cap \F)=\pi_E(L(\cB')) = \pi_E(L(\cB)).
\end{equation}
\end{lemma}

\begin{proof}
 Let $\cB=(Q,E,\del,\inI, \finF)$ an NFA over $E^*$ where 
 $\del \sse Q\times (E\cup \os 1) \times Q$. We run the following while-loop.
 
 While there there are a letter $e\in E$ and states $s,t\in Q$ such that $(s,1,t)\notin \del$ but $e\ov e\in L(\cB,s,t)$ enlarge $\del$ by the $\eps$-\tra $(s,1,t)$. 
 
 The while-loop terminates after at most $\abs{Q}^2$ rounds with the desired  NFA $\cB'$. The number of states is same as before.

 The inclusion $L(\cB') \sse \set{v\in E^*}{u\RAS*{} v\wedge u \in L(\cB)}$ is trivial. The converse  follows by induction on the length of $u$. 
 Moreover, for each $u\in E^*$ there is a (unique) $\wh u\in \F$ such that
 $u\RAS*{} \wh u \in \F$. This shows (\ref{eq:benflood}) and hence the lemma.
\end{proof}

Let us define a new formula $\Psiben$ in two steps: 

\begin{enumerate}
\item Every constraint $X\in \pi_E(L(\cB))$ (resp.~$X\notin \pi_E(L(\cB))$ 
is replaced by $X\in L(\cB')$ 
(resp.~$X\notin L(\cB')$ where $\cB'$ is the NFA constructed in \prref{lem:benflood}.  
Let $\Psi$ be the new formula. 
\item  Define $\Psiben$ by
\begin{equation}\label{eq:frrec}
\Psiben = \Psi \wedge \bigwedge\set{X\in \F}{X\in \cX}.
\end{equation}
\end{enumerate}

\begin{lemma}\label{lem:S2AE}Let $\Phi$ satisfy the properties in \prref{lem:redphase1}. 
Then the embedding $\phi\colon G\to \F(E)\rtimes H$ induces a bijection between 
$\pi_A({\cSol}_{A,k}(\Phi))\sse \F(A)^k$ and \\$\pi_E({\cSol}_{E,k}(\Psiben))\sse \F(E)^k$. Moreover, $(\Abseq{\Psiben},\Absrat{\Psiben})\in \Oh(\Abseq{\Phi},\Absrat{\Phi})$. 
\end{lemma}

\begin{proof}
The proof is immediate by  (\ref{eq:OhPhiPsi}), (\ref{eq:benflood}), and the construction of $\Psiben$ which makes sure that all variables satisfy  the constraint  $X\in \F$. Thus a constraint $X\in L(\cB')$ is equivalent to a constraint $X\in L(\cB')\cap \F$ and a constraint $X\notin L(\cB')$ is equivalent to a constraint $X\in \F\sm L(\cB')$.
\end{proof}

\subsection{Proof of \prref{thm:virtfreestnf}, Phase 5.\\
Switching from NFAs to finite monoids: {}From $\Psiben$ to   $\Psimon$}\label{sec:prvf16}
The goal is to reduce the proof \prref{thm:virtfreestnf} to \prref{thm:central}. This requires that we represent regular constraints by 
recognizing \morph{s}. In the following a \emph{guess} means to run deterministically over all possibilities. That is, there is deterministic  transducer which respects the space bound in \prref{thm:virtfreestnf} and produces all possible outputs one after another. The corresponding \edtol relations are calculated separately and then everything is put together as we did when we split 
$\wt \Phi$ into formulae $\Phi_{\eta,j}$ in \prref{eq:disginH}.

Let $L(\cB_1) \lds L(\cB_\ell)$
 be the list of NFAs which appear in $\Psiben$. We have $\ell \geq 1$ and without restriction $L(\cB_\ell)=\F$. 
According to \prref{ex:redword} there is a \morph $\mu_\ell\colon E^*\to N_\ell$ to  an $H$-monoid $N_\ell$ of size 
$2+\abs E^2 - \abs E$ such that $u\in \F\iff \mu_\ell(u)\in F_\ell$ where 
$F_\ell =\mu_\ell(\F)$. Since $\abs E \in \Oh(1)$, the monoid $N_\ell$
is of constant size. For the other constraints we cannot expect such a small recognizing monoid, and we  use Boolean matrices instead. 
For $1\leq i <\ell$ let $q_i$ be the number of states of the NFA $\cB_i$.
According to \prref{sec:NFA2Bo} and \prref{ex:dgh05IC} in \prref{sec:regiH}
we find for each $1\leq i <\ell$ a \morph to a \morph $\mu_i\colon E^*\to N_i$ to  a monoid with \invol $N_i$ of size 
$4^{q_i^2}$ such that $u\in L(\cB_i) \iff \mu_\ell(u)\in F_i$ where 
$F_i =\mu_i(\F \sm L(\cB_i))$ for a negative constraint and $F_i =\mu_i(L(\cB_i))$ for a positive constraint. Recall that the monoid $N_i$ is a submonoid of $\B^{2n \times 2n}$. Let $N$ be the direct product
$N_0 = N_1 \times \cdots \times N_\ell$. Let $\pi_i\colon  N_0\to N_i$ the canonical projection, 
Then we obtain a single \morph $\mu\colon E^* \to N_0$ such that 
$\mu_i = \pi_i \mu$ for all $1\leq i <\ell$.

Now, for each $X\in \cX$ we guess a value $\nu(X) \in N_0$.
Each time we make a guess  $\nu(X) \in N_0$ we check that it is consistent with the constraints. Thus, for each $0\leq i \leq \ell$ and each $X\in \cX$ 
we do the following. If there there a positive constraint $X\in \cB_i$, then we check 
$\pi_i \nu(X) \in \mu_i(L(\cB_i))$. If there is a negative constraint $X\notin \cB_i$, then  we check 
$\pi_i \nu(X) \notin \mu_i(L(\cB_i))$. 
If the guess is not consistent, then the guess is not successful and 
the corresponding output is empty. 

For a consistent guess $\nu$ we define the following formula
\begin{equation}\label{eq:Psirat}
\Psimonnu= \bigwedge\set{U_j=V_j}{j\in J}\wedge\bigwedge\set{X=\nu(X)}{X\in \cX}.
\end{equation}
Here, $\set{U_j=V_j}{j\in J}$ is the set of equations which appear in the conjunction $\Psiben$. 
By a slight abuse of language we call a conjunction as in \prref{eq:Psirat}
still a Boolean formula. It is clear what we mean 
by a \emph{\solu} of $\Psimon$, it is given by 
\morph $\sig\colon \cX\to E^*$ such that 
\begin{enumerate}
\item $\pi_E\sig(U_j)=\pi_E\sig (V_j)\in \F(E)$ for all $j\in J$.
\item $\mu \sig(X) =\nu(X)$ for all $X\in \cX$.
\end{enumerate}
For an inconsistent guess we let $\Psimonnu=\bot$. 
Using this interpretation we have 
\begin{equation}\label{eq:Psimon}
{\cSol}_{E,k}(\Psiben) = \bigcup\set{{\cSol}_{E,k}(\Psimonnu)}{\nu\in N_0^{\cX}}.
\end{equation}
The  size of the finite monoid $N_0$ is in $2^{\Oh(\Absrat\Phi^2)}$. 
Thus, in general we cannot store  the disjunction over all guesses in  $\PSPACE$. So, we produce the required NFAs for each $\Psimonnu$ again one after another. 
We are approaching our goal prove \prref{thm:virtfreestnf}. For that we use  
the following proposition. 
\begin{proposition}\label{prop:nikola}
Let $\Phi$ satisfy all conditions in \prref{lem:redphase1}.
Then there is an 
 $\NSPACE(\Abseq{\Phi}^2 (\Absrat{\Phi}^2 +\log\Abseq{\Phi}))$ algorithm which performs the following task. 
It takes as input a Boolean formula 
\begin{equation}\label{eq:nikola}
\Psimon=\bigwedge \set{U_j=V_j}{j\in J}\wedge\bigwedge\set{X=\nu(X)}{X\in \cX}
\end{equation}
which appears as $\Psimonnu$ in (\ref{eq:Psirat}), 
The output is an extended alphabet $C$ of size $\Oh({\Abs{\Phi}}^2)$, letters $d_i\in C$ for all $1\leq i \leq k$, and a trim NFA ${\cA}_{\Psimon}$ accepting a rational set of 
$S$-\morphs over $C^*$ such that the \edtol relation 
$$\set{(h(d_1)\lds h(d_k)) \in C^* \times \cdots \times C^* }{h \in L({\cA}_{\Psimon})}$$ is equal to the full solution set in freely reduced words 
$${\cSol}_{E,k}(\Psimon)= \set{(\sig(X_1)\lds \sig(X_k))\in \F^k}{\pi_A\sig(\Psimon)= \text{ true}}.$$
Moreover, ${\cSol}_{E,k}(\Psimon)=\es$ \IFF $L({\cA}_{\Psimon})= \es$;  and $\abs{{\cSol}_{E,k}(\Psimon)}<\infty$  
\IFF ${\cA}_{\Psimon}$ doesn't contain any directed cycle. 
\end{proposition}

\subsection{Proof of \prref{prop:nikola}}\label{sec:prvf160}
In the proof of \prref{prop:nikola}  we wish to apply  \prref{thm:central}.
In order to do so, we define another parameter  $m(\Phi)$ by the following equation. 
\begin{equation}\label{eq:mPhi}
m(\Phi)\cdot  \log(\Abseq{\Phi})=\Absrat{\Phi}^2.
\end{equation}
Recall that $\log(m) \geq 1$ by the definition in \prref{sec:compl}. 
Therefore we have: 
\begin{equation}\label{eq:mPhivsmS}
\NSPACE(\Abseq{\Phi}^2 (\Absrat{\Phi}^2 +\log\Abseq{\Phi})) =
\NSPACE(\Abseq{\Phi}^2 m(\Phi) \log\Abseq{\Phi}).
\end{equation}
Hence, for the proof we can use the space bound $\NSPACE(\Abseq{\Phi}^2 m(\Phi) \log\Abseq{\Phi})$ which is in the form we need 
for  \prref{thm:central}.

 \subsubsection{From $\Psimon$ to  the system  $\cSPhi$}\label{sec:prpr1}

Let  $\Psimon$ be written as in (\ref{eq:nikola}). Then we define 
$\cS'$ to be the following system of equations (without any constraint)
$$ \cS' = \bigwedge\set{U_j=V_j}{j\in J}.$$
Recall that we have defined \morph{s} $\mu\colon E^*\to  N_0$ and $\nu:\cX\to N_0$. 
We join   $\mu$ and $\nu$ to a 
 \morph $\mu_0\colon (E\cup\cX)^*\to  N_0$ 
 by letting $\mu_0(e) = \mu(e)$ for $e\in E$ and $\mu_0(X) = \nu(X)$ for $X\in \cX$. 

The group $H$ acts on $E$, but neither on $\cX$ nor on $N_0$. Therefore we perform two steps. First, we embed the set of variables  $\cX$ into a larger set of  \emph{twisted}  variables 
$$\cY= H\times \cX.$$ In order to have an embedding we identify $Z\in \cX$ with $(1,Z)\in \cY$. 

The group $H$ acts freely without fixed points on $\cY$ by 
$g\cdot (f,X)= (fg,X)$ and $\ov{(f,X)}= (f, \ov X)$. In this way every \morph
$\sig\colon  \cX \to E^*$ extends uniquely to an $H$-\Hmorph
$\sig\colon  \cY^*\to E^*$. 

Second,  we embed $N_0$ into a larger $H$-monoid $N$ as constructed in \prref{sec:ngi}. 
Moreover, using the universal property of the $H$-monoid $N$, we extend 
$\mu_0:(E\cup\cX)^*\to  N_0$ uniquely to a \morph $\mu_N:(E\cup\cY)^*\to  N$ of $H$-monoids by $\mu_N(f,Z) = f(\mu_0(Z))$. 
Twisted variable of the form $(f,Z)$ appear in $\cS'$ only if $f=1$, but formally the set of variables is now $\cY=H\times \cX$ and for each variable 
$Y\in \cY$ the value $\mu_N(Y) \in N$ is defined. The \morph $\mu_N$ is respects the \invol and  the action of $H$, so does every \solu  $\sig\colon  \cY^*\to E^*$.

We define the system $\cSPhi$ by the system $\cS'$ with the  set 
of variables $\cY$ and where for each $Y\in \cY$ there is a constraint $\mu_N(Y) \in N$.

\subsubsection{Triangulation: {}From $\cSPhi$ to $\cStri$ }\label{sec:prvf4}
The ``problem'' with the system  $\cSPhi$ is that equations
$U=V$ are written as words $U,V\in (\Gam\cup \cY)^*$ where 
$\Gam= E\cup H \sse \F(E) \rtimes H$. In a twisted word equation the $U$ and $V$ should be words over $E\cup \cY$. 

Now, let $W\in (\Gam\cup \cY)^*$ be any word. We intend to  move letters $g\in H$ to the right. If $W$ contains a factor $fg$ with $f,g\in H$, then we replace 
$fg$ by the letter $h$ if  $fg=h$ in $H$. Whenever we see a letter $h=1\in H$, then we remove it. If $W$ contains a factor 
 $ga$ where $a\in E$ and $g\in H$, then we replace it  by $bg$ where $b\in E$ corresponds to the letter $g a \ov g\in E$ according to (\ref{eq:sdprod1}). Since $b$ is letter $g$ moves to the right without increasing the length of $W$. 
 The last rule is that  we  replace every factor $g(f,Y)$ with $Y\in \cX$ and $g\in H'$ by $(gf,Y)g$. Again, $g$ moves to the right without increasing the length.
Thanks to this  rule twisted variables other than $(1,Z)$ appear in the equations. 
Thus, every $U=V$ with $U,V\in (\Gam\cup \cY)^*$ can be written 
as $U'f=V'g$ such that $U'V'\in (E\cup \cY)^*$, $f,g\in H$, and $|U'V'|\leq |UV|$. 

Moreover, if $\sig\colon \cY \to E^*$ is any morphism, then $\sig(U) = \sig(U')f$ and 
$\sig(U) = \sig(U')g$. Since $\sig(U')\in E^*$ and $\sig(V')\in E^*$, 
we have  $\pi_E \sig(U)=\pi_E \sig(V)$ only if $f=g$. 
Thus, whenever we find $f\neq g$, then we can stop: there is no \solu. 
On the other hand, for $f=g$ we have $\pi_E \sig(U)=\pi_E \sig(V)\iff \pi_E \sig(U')=\pi_E \sig(V')$. Hence, we can replace 
the equation $U=V$ by  $U'=V'$ and $U'=V'$ is a twisted word equation over
$E$ in twisted variables $\cY= H\times \cX$ and with regular constraints defined 
by an \Hmorph $\cY\to N$. 

Using standard techniques as described in \prref{sec:trisys} we can assume without restriction that all equations are in triangular form. The number of fresh 
(twisted) variables and the increase of the length over all equations is thereby bounded by $\Oh(\Abseq{\Phi})$. The set of variables is still called $\cX$ and the set twisted variables is still called $\cY= H \times \cX$. 
Thus after the modifications above and triangulation we obtain the system of  twisted word equation with regular constraints 
$\cStri$. We have $\phi(\cSol_{A,k}(\Phi)) = \cSol_{E,k}(\cStri)$ and all equations 
 $\cSol_{E,k}(\cStri)$ have the form 
\begin{equation}\label{eq:eqin5G}
(1,Z)=f(x)g(y)
\end{equation}
where $Z\in \cX$ is a variable, $f,g\in H$,  and $x,y\in E\cup \cY$. 
For example, we could have $x= e\in E$ and $y=(h,Y)\in \cY$. Then the equation becomes $(1,Z)=f(e)g((h,Y))= e'(gh,Y)$.
The triangular form is convenient to achieve the property that solutions are in freely reduced words.
\subsubsection{From $\cStri$ to $\cSfin$: solutions in freely reduced words}\label{sec:prvf6}
This subsection mimics \cite{dgh05IC} in the context of virtually free groups. $\cSfin$ will be the ``final'' system in the sequence of transformations. 
Recall that $\F\sse E^*$ denotes the regular subset of freely reduced words. Clearly, if $puq\in \F$ and $f\in H$, then $f(u)\in \F$, too. 
Another crucial observation is that 
for all freely reduced words $x,y,z\in\F$ and $f,g\in H$ we have 
$z=  f(x)g(y)$ in $\F(E)$ \IFF there are freely reduced words $p,q,r$ such that  $$z=pr\quad \wedge \quad x=\oi f(p)q\quad \wedge \quad 
y = (\oi{g}f)(\ov q) \oi{g}(r).$$
According to (\ref{eq:eqin5G}) every equation in the triangular system $\cStri$ has the form  $(1,Z)=  f(x)g(y)$.
For each such equation and each $h\in H$ we introduce six  fresh twisted variables
$$(h,P),(h,\ov P),(h,Q),(h,\ov Q),(h,R),(h,\ov R)$$
After that we replace the equation $(1,Z)=  f(x)g(y)$ by the conjunction of three new equations: 
\begin{equation}\label{eq:cSfin}
(1,Z)=(1,P)(1,R)\quad \wedge \quad x=(f,P)(1,Q)\quad \wedge \quad y= (\oi{g}f,\ov Q)\, (\oi{g},R).
\end{equation}
For simplicity, the new set of twisted variables is still  called  $\cY = H\times \cX$.

We obtain a system $\cSfin$, and  this finishes the construction of the new formula $\cSfin$. 
Let $\sig\colon  \cX\to E^*$ be any \Hmorph such that $\pi_E \sig(1,Z)=  \pi_E \sig(f(x)g(y))$. Then we there is some $\sig':\cX\to \F$ such that 
first, 
$\pi_E \sig = \pi_E \sig'$, and second $\sig'$ solves the three equation in 
(\ref{eq:cSfin}) in freely reduced words. That is $\sig'$ solves the three equation under the constraint $(h,Y) \in \F$ for all $(h,Y)$. 
For the other direction: if $\pi_E \sig$ solves the three equations in $\F(E)$ without any constraint on twisted variables, then there is some 
$\sig'\colon \cX\to \F$ such that $\pi_E \sig'$ solves the equation 
$(1,Z)=  f(x)g(y)$ in $\F(E)$. 
The remaining problem is that our formalism asks to define 
values $\mu_N$ for each new variables. (That is $2/3$ of all variables). 
The only way to do so in the given space bound is to guess the correct value. We can write the equations appearing in $\cSfin$ as 
a system 
\begin{equation}\label{eq:sysNM}
\bigwedge\set{x_j=(f_j,P_j)(g_j,R_j)}{j\in J}
\end{equation}
where $x_j\in E\cup \cX$ and $\mu_N(x_j)$ is already fixed. 
We can read this system as a system of equations over the finite monoid $N_M$. To check whether such systems have a \solu is actually $\PSPACE$ hard, however we don't need $\PSPACE$-hardness. It is enough that within our space bound we can output by guessing-and-checking all possibilities to assign 
$\mu_N$ values to each of the fresh twisted variables. Such an assignment is again a tuple $\nu\in  N_M^{\cX}$. 
Formally we can write a  $\cSfin$ as a  disjunction 
\begin{equation}\label{eq:sysNMnu}
\cSfin= \bigvee\set{\cSfinnu}{\nu\in N_M^{\cX}}.
\end{equation}
Some of the systems $\cSfinnu$ can be empty. If all are empty, then we can stop: $\cSfin$ is not solvable. 


\begin{lemma}\label{lem:Psi22}
Let $\Phi$ satisfy all conditions in \prref{lem:redphase1}.
Then there is an 
 $\NSPACE(\Abseq{\Phi}^2 (\Absrat{\Phi}^2 +\log\Abseq{\Phi}))$ algorithm which performs the following task. 
It takes as input a Boolean formula 
non-empty system $\cSfinnu$ from the  disjunction in  \prref{eq:sysNMnu}.

The output is an extended alphabet $C$ of size $\Oh(\Abseq{\Phi}^2)$
with $E\sse C$, letters $d_i\in C$ for all $1\leq i \leq k$, and a trim NFA ${\cA}_{\cSfinnu}$ accepting a rational set of 
$E$-\morphs over $C^*$ such that the \edtol relation 
$$\set{(h(d_1)\lds h(d_k)) \in C^* \times \cdots \times C^* }{h \in L({\cA}_{\cSfinnu})}$$ is equal to the full solution set in freely reduced words
$${\cSol}_{E,k}(\cSfinnu)= \set{(\sig(X_1)\lds \sig(X_k))\in \F^k}{\pi_E\sig \text{ solves } \cSfinnu}.$$
Moreover, ${\cSol}_{E,k}(\cSfinnu)=\es$ \IFF $L({\cA}_{\cSfinnu})= \es$;  and $\abs{{\cSol}_{E,k}(\cSfinnu)}<\infty$  
\IFF ${\cA}_{\cSfinnu}$ doesn't contain any directed cycle.
\end{lemma}

\begin{proof}
The existence of the NFA ${\cA}_{\cSfinnu}$ with the desired properties is a formal consequence of \prref{thm:central}. For the complexity issues we need an estimation of $m_{\cA_{\cS_{\text{fin},\nu}}}(N_M)$. It is however clear from the construction that we have \[m_{\cA_{\cS_{\text{fin},\nu}}}(N_M)\in \Oh(m(\Phi))\] where 
$m_{\cA_{\cS_{\text{fin},\nu}}}(N_M)$ was defined in  \prref{eq:moncS} and $m(\Phi)$ was defined in \prref{eq:mPhi}. 

 $\NSPACE(\Abseq{\Phi}^2 (\Absrat{\Phi}^2 +\log\Abseq{\Phi})) =
\NSPACE(\Abseq{\Phi}^2 m(\Phi) \log\Abseq{\Phi})$ is due to \prref{eq:mPhivsmS}. Thus the complexity follows again by  \prref{thm:central}.
\end{proof}

 We did various modifications to the input formula $\Phi$ to arrive at a system 
 $\cSfinnu$  mentioned  \prref{lem:Psi22}. Each step on the way from $\Phi$ to $\cSfinnu$ involved a splitting or guessing, which are realized by transducers respecting the space bound. 
 In order to define the NFA ${\cA}_{\Psimon}$ which is needed for \prref{prop:nikola}, 
 we put all the pieces together. Thus, \prref{prop:nikola} is shown.

\begin{corollary}\label{cor:Psi2}
Let $G$ be a finitely generated virtually free group given by  a short exact sequence as in 
(\ref{eq:shexseq}) and let $\phi\colon G\to \F(E)\rtimes H$ the embedding of $G$ into a semi-direct product as in \prref{fig:embed}. 

 Then there is an  $\NSPACE(\Abseq{\Phi}^2 (\Absrat{\Phi}^2 +\log\Abseq{\Phi}))$  algorithm which performs the following task. 
It takes as input a Boolean formula $\Phi$. The output is an extended alphabet $C$ of size $\Oh(\Abseq{\Phi}^2)$
with $E\sse C$, letters $d_i\in C$ for all $1\leq i \leq k$, and a trim NFA ${\cA}_{E,\Phi}$ accepting a rational set of 
$E$-\morphs over $C^*$. The corresponding 
\edtol relation 
$$\cR({\cA}_{E,\Phi}) =\set{(h(d_1)\lds h(d_k)) \in C^* \times \cdots \times C^* }{h \in L({\cA}_{E,\Phi})}$$ 
 satisfies the following properties. 
 \begin{enumerate}
\item We have $\cR({\cA}_{E,\Phi})\sse \F^k$. Thus, each for each $h\in  L({\cA}_{E,\Phi})$ and $1\leq i \leq k$ the word $h(d_i)$ is freely reduced. 
\item We have $\phi(\cSol_{A,k}(\Phi)) = \cR({\cA}_{E,\Phi})$.
\end{enumerate}
\end{corollary}

\begin{proof}
 As above we can  use  the same techniques of splitting and guessing based on  (\ref{eq:OhPhiPsi}),  (\ref{eq:frrec}), and (\ref{eq:Psimon}). Hence it is possible to construct the NFA
 ${\cA}_{E,\Phi}$ by putting exponentially many NFAs of the form ${\cA}_{\Psimon}$
 provided by \prref{prop:nikola}. Again we may use a transducer which satisfies the required space bound since all pieces can be constructed one after another. 
\end{proof}

\subsection{Proof of \prref{thm:virtfreestnf}. {}From the NFA ${\cA}_{E,\Phi}$ back to $\Phi$.}\label{sec:prvf7}

We have $A\sse E$ and in \prref{sec:prvf2} we defined an $A$-\morph $\psi\colon E^*\to A^*$
by $\psi(t)=1$ for all $t\in E\sm A$. Since $\phi(a) \in T^*aT^*$ we see 
$\psi \phi(a)= a$ for all $a\in A$. See the commutative  diagram in \prref{fig:diachase}.
Therefore, the second statement in \prref{cor:Psi2} yields
$$\cSol_{A,k}(\Phi) = \psi(\cR({\cA}_{E,\Phi})).$$
The first statement says that $\cR({\cA}_{E,\Phi})$ is an \edtol relation in freely reduced words over $E$; and \prref{cor:DW17} asserts that $\psi$ maps freely reduced words to freely reduced words over $A$. Using one state more than the NFA  
${\cA}_{E,\Phi}$ (actually a new initial state) and a \tra labeled by $\psi$ from the old initial state to the new one, we obtain the desired  NFA $\cA_\Phi$. Hence, we can realize $\cSol_{A,k}(\Phi)$ as an effective \edtol relation in freely reduced words over $A$. Thus, the projection $\pi_A\colon A^*\to \F(A)$ yields a bijection between 
$\cSol_{A,k}(\Phi)$ and the full \solu set $\pi_A(\cSol_{A,k}(\Phi))\sse \F(A)^k$.
This concludes the proof of \prref{thm:virtfreestnf}. 
\begin{figure}
\begin{center}
    \begin{tikzpicture}[scale=0.9, >=latex,shorten > =1pt,auto,initial text={}, every state/.style={minimum size=8mm}, node distance=2cm]
       \node[] (F) {$A^*$};
    \node[right of=F] (G) {$T^*AT^*$};
    \node[right of=G] (H) {$A^*$};
  
      \node[below of=F] (F2) {$\F(A)$};
    \node[right of=F2] (G2) {$\F(E)$};
    \node[right of=G2] (H2) {$\F(A)$};

     \draw[->] (F) to node[above] {$\phi$}(G);
    \draw[->] (G) to node[above] {$\psi$} (H);
  
       \draw[->] (F2) to node[above] {$\phi$} (G2);
    \draw[->] (G2) to node[above] {$\psi$} (H2);

    \draw[->] (F) to node[right] {$\pi_A$}(F2);
    \draw[->] (G) to node[right] {$\pi_E$} (G2);
    \draw[->] (H) to node[right] {$\pi_A$} (H2);
    \end{tikzpicture}
   \caption{$\psi \phi= \id{A}$ and $\psi$ maps freely reduced words to freely reduced words by \prref{cor:DW17}.}
     \label{fig:diachase}
\end{center}
\end{figure}

\fi
\section{ $\SL(2,\Z)$}\label{sec:sl2z}
 In this section  we apply our results to the perhaps most  prominent example of a (non-free) virtually free 
group: the special linear group $\SL(2,\Z)$ of $2\times 2$ matrices over $\Z$.
It is well known that $\SL(2,\Z)$ is isomorphic to the amalgamated product $\Z/4\Z \star_{\Z/2\Z} \Z/6\Z$. Possible generators  to establish the isomorphism between 
\SLZ and $\Z/6\Z \star_{\Z/2\Z} \Z/4\Z$ are the matrices
$\rho = \vdmatrix0{-1}{1}1 $  and $\tau = \vdmatrix0{1}{-1}0$ of orders $6$ and $4$ respectively. We have $\rho^3=\tau^2 = \vdmatrix {-1}00{-1}$. We also denote the matrices $\vdmatrix {-1}00{-1}$ and $\vdmatrix {1}00{1}$ as $-1$ and $1$ respectively.\footnote{Typical proofs  for $\SL(2,\Z)\cong\Z/4\Z \star_{\Z/2\Z} \Z/6\Z$ use a ``ping-pong-argument'' for  a faithful action 
of the projective linear group $\PSL(2,\Z)=\SL(2,\Z)/\os{\pm 1}$
on $\R\sm \Q$.}
When working with algebraic problems over $\slz$, like solving equations, it is more natural that the constants are just matrices (with entries written as binary numbers) rather than words over a finite generating set. Moreover,  there is no reason to see a sum or a factor of two matrices 
$\vdmatrix{a}{b}{c}{d}$ and $\vdmatrix{a'}{b'}{c'}{d'}$  because we would add or multiply the matrices together. 
For a matrix $M= \vdmatrix abcd$ in $\slz$ we let $\Absone M= \max\os{|a|+|c|,\,|b|+|d|}$; and we define its \emph{binary size} 
$${\Absbin{M}} = \log \Absone M.
$$
Note that $\Absone M$ is the usual matrix \emph{one-norm} of the matrix $\vdmatrix acbd$. 
We use the notion of binary size to define the size of equations and Boolean formulae where constants are matrices. The only difference is that the size of a constant $M$ in $\slz$ is not $1$ as for a  finite generating set, but ${\Absbin{M}}$. We leave it to the reader to define the size of a Boolean formula accordingly. To have a notation for Boolean formulae $\Phi$ as well, we denote the new size by $\Absbin \Phi$.

The aim of this section is to prove the following  result. 
\begin{corollary}\label{cor:sl2z}
There exists a  
generating set $S$ for \SLZ of $21$ letters, and an $\NSPACE(m(\Phi) \Absbin \Phi^2\log\Absbin \Phi)$ algorithm which performs the following task. 
It takes as input a Boolean formula $\Phi$ where the constants are matrices over  
\SLZ (counted in their binary size) and in variables from{} $\cX=\cX_+\cup \cX_-$ such that $X\in \cX_+\iff \ov X\in \cX_-$ and 
$\cX_+=\os{X_1\lds X_k}$, where each variable has size $1$ for simplicity. 
The output is an extended alphabet $C$ of size $\Oh(\Absbin \Phi^2)$, letters $d_i\in C$ for all $1\leq i \leq k$, and a trimmed NFA ${\cA}_\Phi$ accepting a rational set of 
$A$-\morphs over $C^*$ such that the \edtol relation 
$$\set{(h(d_1)\lds h(d_k)) \in C^* \times \cdots \times C^* }{h \in L({\cA}_\Phi)}$$ is equal to the full solution set in standard normal form as given in  \prref{eq:fullsset}
$${\cSol}_{S,k}(\Phi)= \set{(\sig(X_1)\lds \sig(X_k))\in {\stnf}_S(G)^k}{\pi\sig(\Phi)= \text{true}}.$$
Moreover, $\cSol(\Phi)=\es$ \IFF $L(\cA)= \es$;  and $\abs{\cSol(\Phi)}<\infty$  
\IFF $\cA$ doesn't contain any directed cycle. 
\end{corollary}

The proof of  \prref{cor:sl2z} covers the rest of the section. 
In a first part, we make the reduction to the framework of \prref{thm:virtfreestnf} fully explicit. The main message is that a few elementary facts  are enough to  apply \prref{thm:virtfreestnf} to $\slz$ without any reference to Bass-Serre theory \cite{serre80} with the resulting black box \prref{prop:DW17}. In fact, what we use about \SLZ predates the invention of Bass-Serre theory.
For that we reformulate \prref{thm:virtfreestnf} for \SLZ in  \prref{cor:sl2z}.
The point is that we view  \prref{cor:sl2z} directly as a corollary to  \prref{thm:central}. 

In the second part we show that working with matrices doesn't increase the complexity. 
To see  the difference, let $w\in \os{\rho,\rho^{-1},\tau,\tau^{-1}}^*$ be  a word in the (symmetric) set of natural generators with $n=|w|_{\rho^{\pm 1}}= |w|_{\rho}+|w|_{\rho^{-1}}$, and let ${\vdmatrix abcd}$ denote its image in $\slz$, then a straightforward calculation shows $\Absone{\vdmatrix abcd} \leq F_{n+2}$, where $F_{n+2}$ is the $(n+2)$nd Fibonacci number. \Ip
$$\Absbin{\vdmatrix abcd} \leq  |w|_{\rho^{\pm 1}}\leq |w|.$$ 
This means that working with matrices and their binary size doesn't increase the input size \wrt the reduced word lengths over $\os{\rho,\rho^{-1},\tau,\tau^{-1}}^*$. However, an exponential gap between $\Absbin{\vdmatrix abcd}$ and  $|w|_{\rho^{\pm 1}}$ is possible. 
For example, we have $(\tau\rho)^n= \vdmatrix1n01$ and 
$\Absbin{\vdmatrix 1n01}  = \log (n+1).$
It is easy to see that $(\tau\rho)^n$ is the shortest word  in $\os{\rho,\rho^{-1},\tau,\tau^{-1}}^*$ which represents the matrix $\vdmatrix1n01$.
Thus, the matrix representation of (shortest) words can lead to an exponential compression.
However,  in \cite{GurevichS07} Gurevich and Schupp  give an exponential 
representation of a matrix $M$  in \SLZ by words over $\os{\rho,\tau}^*$ where the bit complexity of the exponential 
representation is linear in $\Absbin M$. 
(In an exponential representation exponents over factors are written in binary.)
In order to prove the lemma of Gurevich and Schupp we use the matrices  $L=\vdmatrix10{1}1= \tau\rho^2$ and $U=\vdmatrix1101= \tau\rho$ in \prref{sec:embedsl2z-b}.

\subsection{Explicit embedding of $\SL(2,\Z)$ into a semi-direct product}\label{sec:embedsl2z-a}
Throughout, we use $\SL(2,\Z)\cong G$ and we work with 
$G=\Z/6\Z \star_{\Z/2\Z} \Z/4\Z$ and its quotient $\PSL(2,\Z)\cong G'=\Z/3\Z \star\Z/2\Z$. The group $G'$ is the modular group\footnote{Note that $G'$ is \textbf{not} the derived 
subgroup $[G,G]$.} which is frequently denoted as $\Gam$ in the literature. 
There are natural actions of $G$ and $G'$ (and hence of $\slz$ and $\PSL(2,\Z)$) on the complete bipartite graph $K_{3,2}$. The actions are defined below and the graph $K_{3,2}$ is depicted in \prref{fig:sl2z}.
We give an orientation to  the set of undirected edges
 in $K_{3,2}$ according to that picture\footnote{Actually, $K_{3,2}$ is  the quotient graph of the Bass-Serre tree for $\SL(2,\Z)$ modulo the action by that group. \prref{fig:sl2z} as well as some subsequent calculations appear in \cite{DiekertW17crm}, too.}. 
 Denoting the set of directed edges 
 $\os{a\lds f}$, we obtain an alphabet $E= \os{a,\ov a\lds f,\ov f}$. As usual, an undirected edge is two-element set $\os{y,\ov y}$.) 
 \begin{figure}[h!]
	\begin{center}		
		\begin{tikzpicture}[
		xscale=3.0,
		yscale=0.6,
		inner sep=.5,
		arc/.style={->, >=latex},
		node/.style={circle},
		label/.style={pos=0.5, above=2pt}
		]
		
		
		\node [node] (P1) at (0,4) {$P_1{=}\star$};
		\node [node] (P2) at (0,2) {$P_{{\rho}}$};
		\node [node] (P3) at (0,0) {$P_{{\rho}^2}$};
		\node [node] (R1) at (2,3) {$R_1$};
		\node [node] (R2) at (2,1) {$R_{\tau}$};
		
		
		\draw [arc] (P1) to node [label] {$a$} (R1);
		\draw [arc] (P1) to node [label] {$b$} (R2);
		\draw [arc, bend right=80,dashed,semithick, distance=2.8cm] (R1) to node [label] {$c$} (P2);
		\draw [arc] (P2) to node [label] {$d$} (R2);
		\draw [arc, bend left=70,dashed,semithick, distance=2.2cm] (R1) to node [label] {$f$} (P3);
		\draw [arc] (P3) to node [label] {$e$} (R2);
		
		\end{tikzpicture}
		\caption{The complete bipartite (plane) graph $K_{3,2}$ with a oriented spanning tree  $\os{a,b,d,e}$ and set of directed chords $c$ and~$f$.}\label{fig:sl2z}
			\end{center}
\end{figure}
Let us write $G=\gen \rho \star_{\Z/2\Z} \gen \tau$ and 
$G'=\gen r \star \gen t$
where $\rho^6=r^3=1$ and $\tau^4=t^2=1$. 
 The action of $G$ on the bipartite graph $K_{3,2}$ is as follows.
The generator $\tau$ (resp.~$t$) stabilizes the vertices $P_\alp$ for $\alp\in \os{1,\rho,\rho^2}$, and we let  $\tau {R}_{1}= t {R}_{1}= {R}_{\tau}$ with 
${R}_{\tau^2}= R_{1}$. Thus, $\tau$ and $t$ are \emph{transpositions}. 
The elements $\rho\in G$ and $r\in G'$ are \emph{rotations}. 
Both stabilize the vertices $R_{1}$ and $R_{\tau}$, and we let  
$\rho {P}_\alp= r {P}_\alp={P}_{\rho\alp}$ (with 
${P}_{\rho^3}= P_{1}$). The action of $G'$ is faithful, and its image 
in the $\Aut(K_{3,2})$ is the direct product $\Z/3\Z \times\Z/2\Z$. 
Thus, $H'=\Z/6\Z$ acts on $K_{3,2}$ by identifying the action of  $2\in \Z/6\Z$ with the action of $r$ and by identifying the action of  $3\in \Z/6\Z$ with the action of $t$. 

This leads to surjective \hom{s} $\gam\colon G \to \Z /12\Z$ 
and $\gam'\colon G' \to \Z /6\Z$ by $\gam(\rho)=\gam'(r)=2$ and  $\gam({\tau})=\gam'(t)=3$. Note that $\gam$ is a \hom since $\gam(\rho^3)=6 = \gam(\tau^2)$. Moreover, $\gam$ induces an isomorphism between the kernel of the canonical projection of $G$ to $G'$ (which is the center of $G$ generated by $\tau^2$)  and the  the kernel of the canonical projection of $\Z/12\Z$ to $\Z/6\Z$. Thus, the mapping $\rho\mapsto r$ and $\tau\mapsto t$ induces a canonical \iso between the kernels $\ker(\gam)$ and $\ker(\gam')$ by \prref{fig:cnt}. 

\begin{figure}
\begin{center}
    \begin{tikzpicture}[scale=0.9, >=latex,shorten > =1pt,auto,initial text={}, every state/.style={minimum size=8mm}, 	xscale=2.8, yscale=1.8			
]
\path (2,2.5) node (1tl) {$1$};
\path (3,2.5) node (1tr) {$1$};
\path (1,1.8) node (1zll) {$1$};
\path (4,1.8) node (1zrr) {$1$};
\path (2,1.8) node (1zl) {$\gen{\tau}$};
\path (3,1.8) node (1zr) {$\Z/2\Z$};
\path (1,-0.4) node (bll) {$1$};
\path (0,1) node (1) {$1$};
\path (1,1) node (F) {$[G,G]$};
\path (2,1) node  (G) {$G$};
\path (3,1) node  (H) {$\Z/12\Z$};
\path (4,1) node  (1r) {$1$};
\path (0,0.2) node (12) {$1$};
\path (1,0.2) node (F2) {$[G',G']$};
\path (2,0.2) node (G2) {$G'$};
\path (3,0.2) node (H2) {$\Z/6\Z$};
\path (4,0.2) node (1r2) {$1$};

\path (2,-0.4) node (1bl) {$1$};
\path (3,-0.4) node (1br) {$1$};
\draw[->] (1zll) to (1zl);
\draw[->] (1zr) to (1zrr);
\draw[->] (1tl) to (1zl);
\draw[->] (1zll) to (F);
\draw[->] (F2) to (bll);

\draw[->] (G2) to (1bl);
\draw [->, >=latex] (1zl) -- (1zr) node[midway, above] {=};

\draw[->] (1tr) to (1zr);
\draw[->] (H2) to (1br);
\draw [->, >=latex] (1zl) -- (G) node[midway, above] {};
\draw [->, >=latex] (1zr) -- (H) node[midway, above] {};
\draw[->] (1) to (F);
\draw [->, >=latex] (F) -- (G) node[midway, above] {};
\draw[->] (G) -- (H) node[midway, above]{$\gam$};
\draw[->] (H) to (1r);
\draw[->] (12) to (F2);
\draw [->, >=latex] (G2) -- (H2) node[midway, above] {$\gam'$};
\draw[->] (F2) to (G2);
\draw[->] (H2) to (1r2);
\draw[->, >=latex] (F) -- (F2) node[midway, right] {=};
\draw [->, >=latex] (G) -- (G2) node[midway, right] {can.};
\draw [->, >=latex] (H) -- (H2) node[midway, right] {can.};
\end{tikzpicture}
   \caption{Exact rows and columns with $\gam(\tau) =3$ and $\tau^2=1$ yields 
   the canonical \iso $[G,G]=[G',G']$ on the left.}
     \label{fig:cnt}
\end{center}
\end{figure}

The following proposition is well-known. It is stated in \cite[Lem.~1]{Newman62} (without proof) since, according to the author Morris Newman, \prref{prop:nie48} it is based on a more general result by Jacob Nielsen \cite{nie48}. For the proof of \prref{prop:nie48} one might use the structure of \SLZ as an amalgamated product as well as the well-known fact that the matrices $A'=\vdmatrix 2111$ and  $B'=\vdmatrix 1112$ generate a free subgroup in \SLZ. However, in the spirit of the paper, let us give a purely combinatorial proof of \prref{prop:nie48} using the structure of $G'$ as a free product of $\Z/3\Z$ and $\Z/2\Z$.  
\begin{proposition}\label{prop:nie48}
The kernel of $\gam'$ is the commutator subgroup $[G',G']$ of $G'=\PSL(2,\Z)$ is a free group of rank two, which is generated by the commutators $[t,r]$ and $[t,r^2]$.
\end{proposition}
\begin{proof}
Clearly, $[G',G']\leq \ker(\gam')$. We have 
$[t,r]= trt^{-1}r^{-1}=trtr^2$ and $[t,r^2]= tr^2tr$. Let us show first that 
$\gen{[t,r],\,[t,r^2]}= \gen{trtr^2,\, tr^2tr}$ is a normal subgroup in $G'$. Indeed: 
\begin{enumerate}
\item $t[t,r]t^{-1}= t(trtr^2)t^{-1}= rtr^2t= (trtr^2)^{-1} = [r,t]$.
\item $r^{-1}[t,r]r= r^2(trtr^2)r= r^2trt= [r^2,t]=[t,r^2]^{-1}$.
\item $r[t,r]r^{-1}= r(trtr^2)r^2= rtrtr= rtr (rttr^2) tr= (rtr^2t)(tr^2 tr)= [r,t][t,r^2]$. \label{eq:trtrtr}
\end{enumerate}
This implies that $\gen{[t,r],\,[t,r^2]}$ is the normal subgroup generated by the commutator $[t,r]$. By definition, $G'/[G',G']=G'/\os{[t,r]=1}$. Hence, $\gen{[t,r],\,[t,r^2]}=[G',G']$ is generated by two elements. 

Second, we claim $[G',G']= \ker(\gam')$. We know 
\begin{align*}
 G'/[G',G']=G'/\os{[t,r]=1} &=\os{r,t}^*/\os{r^3=t^2=1, rt=tr}\\ &= \Z/3\Z \times \Z/2\Z = \Z/6\Z.
\end{align*}
Since
$\gam'$ induces a surjective \hom $G'/[G',G']\to \Z/6\Z$, this induced \hom is an \iso.\footnote{More generally, if $K=C_1\star C_2$ is a free product of two cyclic groups, then $[K,K]$ is equal to the kernel of the canonical projection of $K$ to the direct product $C_1\times C_2$.} Hence, the claim.  

The commutator subgroup is henceforth denoted by $F'$. We wish to show that $F'$ is free.
For that step let $r,s,t$ be three letters such that $s=r^{-1}$ in $G'$ and $\Sig=\os{A,\ov A, B, \ov B}$ be four letters. 
Using 
$t^2=1$ we obtain a \hom $\psi':\os{r,s,t}^*\to G'$. 
Let us define the letters $A$ and $B$ as the words
$A=trts$ and $B=tstr$ in $\os{r,s,t}^+$. The aim is to show the restriction of $\psi'$ to $\os{A,B}^*$ defines an \iso $\psi:F(A,B)\to F'$. 
Since $\psi'(A)$ and $\psi'(B)$ generate $F'=[G',G']$, we content ourselves to show that 
a nonempty freely reduced word in $w\in \Sig^+$ is not mapped to $1\in G'$. In $G'$ we use the reduced normal form which is obtained by a confluent and length-reducing rewriting system which replaces $rs$, $sr$, and $t^2$ by $1$ and which replaces $r^2$ by $s$ and  $s^2$ by~$r$. 

We claim that we can detect the last letter of $w\in \Sig^+$ by knowing the  last four letters 
in the reduced normal form of $\psi'(w)$. This is clear for $\abs w=1$. 
Hence, we may assume that $\abs w=k$ with $k\geq 2$ and that the claim is correct
for words of length at most $k-1$. For that we define a finite automaton with 
eight states and transitions which are labeled by the letters $A=trts$, $\ov A= rtst$, 
$B=tstr$, and $\ov B= strt$.
Those transitions where the label differs from the target are given by the following list: 
\begin{multicols}{2}
\begin{enumerate}
\item\label{eq:rtst}\hspace{1cm}$rtst \arc {tstr} trtr$
\smallskip
\item\label{eq:strt}\hspace{1cm}$strt \arc {trts} tsts$
\smallskip
\item\label{eq:trts}\hspace{1cm}$trts \arc {strt} rtrt$
\smallskip
\item\label{eq:tstr}\hspace{1cm}$tstr \arc {rtst} stst$

\item\label{eq:trtr}\hspace{1cm}$trtr \arc {rtst} stst$
\smallskip
\item\label{eq:tsts}\hspace{1cm}$tsts \arc {strt} rtrt$
\smallskip
\item\label{eq:rtrt}\hspace{1cm}$rtrt \arc {trts} tsts$
\smallskip
\item\label{eq:stst}\hspace{1cm}$stst \arc {tstr} trtr$
\end{enumerate}\end{multicols}
We don't define initial or final states. Since we intend to read only freely reduced words, each state has out degree $3$. Hence, there are $24$ transitions but only $8$ of them are listed in the table above.  
The automaton is deterministic. Consider a non-empty freely reduced
word $w=c_1\cdots c_k$ with $c_i\in \Sig$ and $k\geq 1$. 
After reading $c_1$ we are in the corresponding state.  Hence, the claim is correct for $k=1$. For example,
if $c_1=\ov B$, then we are in the state $strt$ which is the 
left-hand side in line (\ref{eq:strt}). For $k\geq 2$ we can assume by induction that the state of the left-hand side in the corresponding represents the last four letters in the reduced normal form of $\psi(c_1\cdots c_{k-1})$.
For example, assume the state is the 
left-hand side in line (\ref{eq:tsts}) which is the right-hand side in line (\ref{eq:rtrt}). This implies $c_{k-1}=A$. Therefore 
$c_k\in \os{A,B,\ov B}$. If $c_k=A$, then reading $w$ leads to the state $trts$. If $c_k=B$, then reading $w$ leads to the state $B=tstr$.
If $c_k=\ov B= strt$, then reading $w$ leads to the state $rtrt$.
Thanks to symmetries, the same type of argument applies in all situations. Since after reading the word $w$ we are in a state where the reduced normal form has length $4$, it is not $1$ in $G'$. Hence, the proposition.
\end{proof} 
\begin{corollary}\label{cor:nie48}
The kernel of $\gam$ is the commutator subgroup $[G,G]$ of $G=\SL(2,\Z)$ is a free group of rank two, generated by the commutators $[\tau,\rho]$ and $[\tau,\rho^2]$.
\end{corollary}
\begin{proof}
We know that the mapping $\rho$ to $r$ and $\tau$ to $t$ defines a \hom{} {}from $G$ to $G'$. This \hom maps 
$[\tau,\rho]$ to $[t,r]$ and $[\tau,\rho^2]$ to $[t,r^2]$ which yields an \iso between $[G,G]$ and $[G',G']$, see \prref{fig:cnt}. Hence, the statement about the commutator subgroup for $G$ follows from \prref{prop:nie48}.
\end{proof}

The action of $H'=\Z /6 \Z$ on $K_{3,2}$ induces a faithful action on the set of directed edges
$E=\os{a,\ov a\lds f,\ov f}$ which respects the \invol. For example:  $\tau(c) = \ov d$, $\tau\rho(a) = d$, and  $\rho(c) = f$ etc.
Therefore, the canonical \hom $H \to H'= \Z/6\Z \leq \Aut(E)$ yields  a semi-direct product $\F(E)\rtimes H$. 
The action of $H$ on $E$ is not faithful: for all $y\in E$ and $m\in \Z$ we have $(\tau\rho)^m(y) =y\iff m\in 6\Z$. 
 In the next step let us show that 
$[G,G]$ is the fundamental group of the graph $K=K_{3,2}$.
Simultaneously, we will derive the desired result that 
$\SL(2,\Z)$ embeds into the semi-direct product $\F(E)\rtimes H$ where $\F(E)=E^*/\set{y\ov y=1}{y\in E}$ 
and $H=\Z /12 \Z$. We choose a spanning tree of $K$ by the solid edges in $K$ according to \prref{fig:sl2z}, 
and we let $\star= P_1$ be a base point in $K$.
Then the fundamental 
group $\pi_1(K,\star)$ (which is, by definition, a subgroup in $\F(E)= F(a,b,c,d,e,f)$)
can be identified with the free group $F(c,f)$ of rank $2$. 
The identification is due to the fact that $c,f$ are the chords for the chosen (directed) spanning tree $T=\os{a,b,d,e}$. Indeed, the  isomorphism $\phi_1\colon F(c,f)\to \pi_1(K,\star)$ is given by 
\begin{align*}
\phi_1(c) = ac d \ov b \quad\text{ and }\quad \phi_1(f) = af e \ov b.
\end{align*}
To see this, say for $c$, just follow the shortest path in $T$ from $\star$ to the source of $c$, traverse the chord $c$ and choose the shortest path in $T$ back to $\star$.
Consider the canonical projection 
$\pr{c,f}: \F(E)\to F(c,f)$ which maps the edges of $T$ to $1$. Then $\pr{c,f} \phi_1$ is the identity on $F(c,f)\leq \F(E)$. 

Define $\psi: \F(c,f)\to [G,G]$ by $\psi(c)=[\tau,\rho]= \tau\rho\tau\rho^2$ and $\psi(f)=[\tau,\rho^2]= \tau\rho^2\tau\rho$. 
It is an \iso by \prref{cor:nie48}.

Finally, guided by $\tau(P_1)= P_1$ and  $\rho(P_1)= P_\rho$  we define a \hom
$\phi\colon G
\to \F(E)\rtimes H$ where
\begin{align}\label{eq:phisemi}
\phi(\tau) = (1,t) \text{ and } \phi({\rho}) = (b\ov d, r).
\end{align}
The \hom $\phi$ is well-defined since 
$$\phi(\tau^2) = (1,t^2) = (1,r^3) = \phi({\rho}^3) \text{ and } (1,t^2)^2 = (1,1).$$
Another direct calculation shows 
$$\phi\psi(c)= \phi(\tau {\rho} \tau {\rho}^2) = (\phi_1(c),1) \text{ and } \phi\psi(f)=\phi(\tau {\rho}^2 \tau {\rho})= (\phi_1(f),1).$$
Thus, the identity $\text{id}_{F(c,f)}$ factorizes as follows: 
$$\text{id}_{F(c,f)}: F(c,f) \arc{\psi}[G,G]\arc{\phi} \F(E)\times \os{1}= \F(E) \arc{\tau} F(c,f).$$

As a consequence, we obtain a commutative diagram \prref{fig:embsl2z} where we let $\SL(2,\Z)=G$. 
Since $\psi$ is bijective,  $\phi\colon \SL(2,\Z) \to \F(E) \rtimes H$ is injective. Hence,
$\phi$ induces an \iso between $[G,G]$ and  the subgroup $\pi_1(K,\star)\leq \F(E)= \F(E)\times \os 1 \leq \F(E) \rtimes H$.  
\begin{figure}
\begin{center}
    \begin{tikzpicture}[scale=0.9, >=latex,shorten > =1pt,auto,initial text={}, every state/.style={minimum size=8mm}, 	xscale=2.8, yscale=1.8			
]
\path (0,1) node (1) {$1$};
\path (1,1) node (F) {$F(c,f)$};
\path (2,1) node  (G) {$\SL(2,\Z)$};
\path (3,1) node  (H) {$\Z/12\Z$};
\path (4,1) node  (1r) {$1$};
\path (0,0) node (12) {$1$};
\path (1,0) node (F2) {$\F(E)$};
\path (2,0) node (G2) {$\F(E) \rtimes \Z/12\Z$};
\path (3,0) node (H2) {$\Z/12\Z$};
\path (4,0) node (1r2) {$1$};
\draw[->] (1) to (F);
\draw [->, >=latex] (F) -- (G) node[midway, above] {$\psi$};
\draw[->] (G) -- (H) node[midway, above]{$\gam$};
\draw[->] (H) to (1r);
\draw[->] (12) to (F2);
\draw[->] (F2) to (G2);
\draw[->] (G2) to (H2);
\draw[->] (H2) to (1r2);
\draw[->, >=latex] (F) -- (F2) node[midway, right] {incl.};
\draw [->, >=latex] (G) -- (G2) node[midway, right] {$\phi$};
\draw [->, >=latex] (H) -- (H2) node[midway, right] {$\id{\Z/12\Z}$};
\end{tikzpicture}
   \caption{Embedding of $\SL(2,\Z)=G$ into a semi-direct product.}
     \label{fig:embsl2z}
\end{center}
\end{figure}

Since $G$ is a finitely generated subgroup, $\phi(G)$  is a rational subset in $\F(E) \rtimes H$. Hence, we can reduce the question about solving equations in $\SL(2,\Z)$ to twisted word equations over 
$\F(E)$ with rational constraints.\footnote{This is the approach of \cite{DahmaniGui10} to solve word equations in virtually free groups, too.}
However, \prref{thm:central} is more  ambitious. In order to apply \prref{thm:central} 
we need \ip an explicit construction of a  set of standard generators. We obtain such a  set $S$ by defining
$S= A_+\cup A_-\cup H_+\cup H_-$  where 
$A_+= \os{c,f} = \os {\tau {\rho} \tau {\rho}^2,\, \tau {\rho}^2 \tau {\rho}}$ and 
$H_+=\os{\rho^1\lds \rho^5,\, \tau, \rho^1\tau\lds \rho^5\tau}$.
We have $H_+\cap H_-= \os{\rho^1\lds \rho^5}$ and $\rho^3$ becomes a self-involuting letter in $S$. 

\begin{remark}\label{rem:notgeonf}
Let $h = \rho\tau$, $h'=\rho$, and $g=\rho^2\tau$ be letters in $H_+$. 
Then the element $h h'\in S^*$ has length $2$. The corresponding element 
in standard normal form is $\ov c f g\in A^*H_+$ which has length $3$. 
This yields a concrete example  showing that the  standard normal forms are not geodesic, in general.
\end{remark}

\subsection{Euclidean matrix calculation}\label{sec:embedsl2z-b}
For the proof of \prref{cor:sl2z} it remains to show that the complexity is not worse than $\NSPACE(\Absbin \Phi^2m(\Phi)\log\Absbin \Phi)$. This is done next. 
We have
$\oi{L}= \rho \tau= \vdmatrix10{-1}1$ and hence, $\oi U L= \rho$. Since $\rho$, $ \tau$ generate \SLZ as a monoid,  we see that $L,U$ generate \SLZ as a group. It is therefore clear that every matrix in \SLZ can be written as a word in $\os{L,\oi L, U, \oi U}^*$, but of course the representation is not unique 
as for example $(\oi U L)^6=1$. 

Let $a_0,a_1\in \N$ with $a_0>a_1>0$. Using the extended Euclidean algorithm for computing  the $\gcd(a_0,a_1)$  we define natural numbers $k_i$ for $0\leq i <g$ and $a_i$  
$0\leq i \leq g+1$ with 
$$a_0>a_1>\cdots >a_{g-1} > a_g = \gcd(a_0,a_1) > a_{g+1}= 0$$ such that for  $i\geq 0$ we have
\begin{equation}\label{eq:euclid}
0\leq a_{i+2}= a_i - k_{i}a_{i+1} < a_{i+1}.
\end{equation}
The sequence finishes with some 
$1\leq g \in \Oh(\log \abs{a_0})$ such that $k_{g-1} a_g= a_{g-1}$ and 
$a_g =\gcd(a_0,a_1)$. The last value is therefore indeed $a_{g+1}=0$.
We say that $(k_0\lds k_{g-1})$ is the \emph{$\gcd$-sequence} defined by  $a_0,a_1$. Note that $(k_0\lds k_{g-1})$ together with $a_g$ uniquely define $(a_0\lds a_{g})$. Note also that $k_i\geq 1$ for $0\leq i <g-1$ and $k_{g-1}= a_{g-1}\geq 2$.

By $\SL(2,\N)$ we mean the following submonoid of \SLZ:
$$\SL(2,\N)= \set{\vdmatrix{a}{b}{c}{d}}{a,b,c,d \in \N \wedge ad-bc=1}.$$
It is a well-known classical fact and not difficult to see that $\SL(2,\N)$ 
is a free monoid with unique basis $\os{U,L}$, see for example \cite[Chap.~8.12]{edam16} or \cite{KarpRabin87} 
for an application to fast randomized pattern matching. 
The following quantitative lemma belongs probably to folklore. 
It can be easily derived from \cite{GurevichS07}, but for lack of a reference for the precise statement we give a proof. 
\begin{lemma}\label{lem:gcd}
Let $M=\vdmatrix{a_0}{a_1}{c_0}{c_1}\in \SL(2,\N)$ with $a_0>a_1>0$ and let $(k_0\lds k_{g-1})$ be the $\gcd$-sequence defined by  $a_0,a_1$.  Then there is a (unique) $c_g\in \N$ such that following assertions hold.
\begin{enumerate}
\item $0< k_0\cdots k_{g-1}\cdot \min\os{1,c_g}< a_0+c_0= \Absone M$.
\item If $g$ is even, 
then 
$$M= L^{c_g}U^{k_{g-1}} L^{k_{g-2}}\cdots U^{k_1} L^{k_{0}}.$$
\item If $g$ is odd, 
then $c_{g}>0$ and
$$M= L^{c_{g}-1}U L^{k_{g}-1} \cdots U^{k_1} L^{k_{0}}.$$
\end{enumerate}
\end{lemma}

\begin{proof}
For the following we don't need the uniqueness of $c_g$. It follows from the fact that $\os{L,U}$ forms a basis for  the free monoid  $\SL(2,\N)$, which in turn follows easily from the present proof. We leave this part to the interested reader. 

Consider a matrix  $M_1=M=\vdmatrix{a_0}{a_1}{c_0}{c_1}\in \SL(2,\N)$ with $a_0>a_1>0$.
Note that  this implies $\gcd(a_0,a_1)=1$. Moreover, $c_0\geq c_1>0$ 
because $a_0c_1=a_1c_0+1$. (The case $c_0= c_1$ is possible only for 
$M=\vdmatrix{a_0}{a_0-1}{1}{1}$.)
 Let us treat the case $a_1=1$ as a special case first. That is: $M= 
\vdmatrix{k_{0}}{1}{c_{0}}{c_{1}}$. 
We obtain $k_0= a_0$ and 
$$M= L^{c_{1}-1}U L^{a_{0}-1}.$$
Moreover, $c_0 = k_0c_1-1$. Since $a_0\geq 2$ we have
$1\leq c_0 < k_0c_1 < a_0 +c_0$.

For the rest of the proof we may assume $g\geq 2$. We let $(k_0\lds k_{g-1})$ (and $(a_0\lds a_{g-1},1)$) be the $\gcd$-sequences defined by  $a_0,a_1$. Next, we define matrices $M_i$ for $1\leq i \leq g$ according to the
following rules. 
\begin{enumerate}
\item If $1\leq i <g$ and $i$ is odd and $M_{i}= \vdmatrix{a_{i-1}}{a_{i}}{c_{i-1}}{c_{i}}$ is defined, then we let 
$$M_{i+1} = M_{i}L^{-k_{i-1}}= \vdmatrix{a_{i+1}}{a_{i}}{c_{i+1}}{c_{i}}.$$
\item If $1\leq i <g$ and $i$ is even and $M_{i}= \vdmatrix{a_{i}}{a_{i-1}}{c_{i}}{c_{i-1}}$ is defined, then we let 
$$M_{i+1} = M_{i}U^{-k_{i-1}}= \vdmatrix{a_{i}}{a_{i+1}}{c_{i}}{c_{i+1}}.$$
\end{enumerate}
It follows by induction that $M_i\in \SL(2,\N)$ for all $1\leq i \leq g$.
Having this we can deduce, again by induction, for all $1\leq i \leq g$:
\begin{align}\label{eq:euclida}
0 &< k_{i-1} a_{i}\leq a_{i-1}\\
\label{eq:euclidaa}
0&<k_0\cdots k_{i-1}a_i\leq a_0 
\end{align}
The situation for the $c_i$ is slightly different. For for all $1\leq i \leq g-1$
\begin{align}\label{eq:euclidac}
0 &<k_{i-1} c_{i}\leq c_{i-1}\\
\label{eq:euclidacc}
0&< k_0\cdots k_{i-1} c_{i} \leq c_0\\
\label{eq:euclidaccc}
0&\leq k_0\cdots k_{g-2}\cdot \max\os{c_{g-1},\, k_{g-1}c_{g} -1} \leq c_0
\end{align}
To see (\ref{eq:euclidaccc}) we observe that $1\leq c_{g-1} = k_{g-1}c_{g} \pm 1$.
Hence, we can use (\ref{eq:euclidacc})  to conclude (\ref{eq:euclidaccc}).
Considering $i=g-1$ shows the first claim in the lemma, because 
(\ref{eq:euclidaccc}) implies 
$k_0\cdots k_{g-1}c_g \leq  c_0+k_0\cdots k_{g-2}$ and  
$k_0\cdots k_{g-2}< k_0\cdots k_{g-1} \leq a_0$ by $k_{g-1}\geq 2$ and  (\ref{eq:euclidaa}).

For the last matrix is $M_g$ and depending on whether $g$ is odd or even, we have
 two options. If $g$ is even we let $D=L$ and $D=U$ otherwise. We obtain:
\begin{align}\label{eq:euclidas}
M_1\cdot L^{-k_0}U^{-k_1} \cdots D^{-k_{g-2}} = M_g= \begin{cases}
 \vdmatrix{a_{g}}{a_{g-1}}{c_{g}}{c_{g-1}} =  \vdmatrix{1}{k_{g-1}}{c_{g}}{c_{g-1}} & \text{ if $g$ is even;}\\
 \vdmatrix{a_{g-1}}{a_{g}}{c_{g-1}}{c_{g}}=  \vdmatrix{k_{g-1}}{1}{c_{g-1}}{c_{g}} & \text{ if $g$ is odd.}   
\end{cases}
\end{align}
 
\noindent \textbf{First case.} Let $g$ be even, hence $M_g= \vdmatrix{1}{k_{g-1}}{c_{g}}{c_{g-1}}$. 
Then we 
have 
$$M_g U^{-k_{g-1}}L^{-c_g} = \vdmatrix{1}{0}{0}{1}.$$
It is possible that $c_g=0$ in the line above. 

\noindent \textbf{Second case.} 
Let $g$ be odd, hence $M_g= \vdmatrix{k_{g-1}}{1}{c_{g-1}}{c_{g}}$. 
Then we 
have 
$$M_g L^{1-k_{g-1}}U^{-1}L^{1-c_{g}} = \vdmatrix{1}{0}{0}{1}.$$
Note that for $g$ odd, we have $c_{g}\geq c_g>0$ and $k_{g-1}= a_{g-1} >a_g=1$. 
Using (\ref{eq:euclidas}) and a case distinction (whether or not $g$ is even)
yields the result.
\end{proof}

\begin{proposition}[Gurevich and Schupp \cite{GurevichS07}]\label{prop:GS07}
Let $M= \vdmatrix abcd \in \slz$ and $m= \max\os{\abs a, \abs b,\abs c, \abs d}$. 
Then there are words $u,v\in \os{\rho,\tau}^*$ and positive integers 
$e_0 \lds e_{\ell}$ with $0\leq {\ell} \in \Oh(\log m)$ such that
\begin{align*}
|uv| &\in \Oh(1),\\
0< e_0\cdots e_{\ell}&  <2m \\
M&= uL^{e_0}U^{e_1}\cdots L^{e_{{g-2}}}U^{e_{\ell-1}}L^{e_{\ell}} v .
\end{align*}
\end{proposition}

\begin{proof}As a preamble let us note that we will be able to enforce $e_0\neq 0 \neq e_\ell$ because  $\oi{L}={\rho\tau}$ is a short word over $\rho$ and $\tau$.

The assertion is trivial for $m=1$. Hence we assume $m\geq 2$. Using short words $u',v'\in \os{\rho,\tau}^*$, we obtain a matrix
$$M'= u'Mv'= \vdmatrix{a_0}{a_1}{c_0}{c_1}$$
with $m= {a_0}> {a_1}>0$ and ${c_0}>{c_1}> 0$.  
Since $m\geq 2$  it is enough to see that we can choose
$u'=\tau^{2+e_0}U^{e_1}\tau^{e_2}$ and $v'= \tau^{e_3}U^{e_4}$ where the exponents $e_j$ are in $\os{0,1}$. We have 
$M'\in  \SL(2,\N)$ and therefore the result follows from \prref{lem:gcd}. 
\end{proof}

\subsubsection*{Proof of \prref{cor:sl2z}}
\prref{prop:GS07} shows that the size of the exponential expression $$uL^{e_0}U^{e_1}\cdots L^{e_{{g-2}}}U^{e_{\ell-1}}L^{e_{\ell}}v$$ is linear in $\Absbin{M}$. 
Thus, we can apply  \prref{cor:Psi2} based on the explicit embedding of $\SL(2,\Z)$ into the semi-direct product as depicted in \prref{fig:embsl2z}.

\section*{Acknowledgments}
The authors are indebted to anonymous reviewers for their careful reading and extremely helpful feedback on the initial submission of this manuscript. We also 
 thank  Armin Wei\ss{} for various suggestions and Igor Potapov for  pointing out  the paper \cite{GurevichS07} of Gurevich and Schupp.

\bibliographystyle{abbrv}
\bibliography{../../../TRACES/traces}

\end{document}